\documentclass[11pt]{article}
\usepackage[utf8]{inputenc}
\pdfoutput=1

\usepackage{fullpage}
\usepackage{amsmath}
\usepackage{palatino}
\usepackage{macros}
\usepackage{latexsym} 
\usepackage{mathtools}
\usepackage{bbm}
\usepackage{authblk}
\usepackage[export]{adjustbox}

\usepackage{tikz}
\usetikzlibrary{arrows,positioning, shapes.geometric}

\usepackage{enumitem}

\usepackage{thmtools} 
\usepackage{thm-restate}
\def\cF{{\cal F}}

\def\cH{{\cal H}}

\def\eps{\epsilon}

\def\rk{\bm{\mathrm{rk}}}
\def\crk{\bm{\mathrm{crk}}}
\def\cE{{\cal E}}

\def\cC{{\cal C}}
\def\cL{{\cal L}}

\def\prt{T}
\def\b1{{\bf 1}}

\def\dimension{d}
\allowdisplaybreaks
\DeclareMathOperator{\Var}{Var}
\def\cond{\phi}
\title{Optimal Trickle-Down Theorems for Path Complexes via $\cC$-Lorentzian Polynomials\\
with Applications to Sampling and Log-Concave Sequences}

\author[1]{Jonathan Leake}
\affil{\small University of Waterloo, \textsf{jonathan.leake@uwaterloo.ca}}
\author[2]{Kasper Lindberg}
\affil{\small ETH Z{\"u}rich, \textsf{kasper.lindberg@inf.ethz.ch}}
\author[3]{Shayan Oveis Gharan}
 \affil{\small University of Washington,  \textsf{shayan@cs.washington.edu}}

\begin{document}
\maketitle
\begin{abstract}
    Let $X$ be  a $d$-partite $d$-dimensional simplicial complex with parts $T_1,\dots,T_d$ and let $\mu$ be a distribution on the facets of $X$. Informally, we say $(X,\mu)$ is a path complex if for any $i<j<k$ and $F \in T_i,G \in T_j, K\in T_k$, we have $\P_\mu[F,K | G]=\P_\mu[F|G]\cdot\P_\mu[K|G].$
    We develop a new machinery with $\cC$-Lorentzian polynomials to show that if all links of $X$ of co-dimension 2 have spectral expansion at most $1/2$, then $X$ is a $1/2$-local spectral expander. 
    We then prove that one can derive fast-mixing results and log-concavity statements for top-link spectral expanders.
    
    We use our machinery to prove fast mixing results for sampling maximal flags of flats of distributive lattices (a.k.a. linear extensions of posets) subject to external fields, and to sample maximal flags of flats of ``typical'' modular lattices. We also use it to re-prove the Heron-Rota-Welsh conjecture and to prove a conjecture of Chan and Pak which gives a generalization of Stanley's log-concavity theorem. Lastly, we use it to prove near optimal trickle-down theorems for ``sparse complexes'' such as constructions by Lubotzky-Samuels-Vishne, Kaufman-Oppenheim, and O'Donnell-Pratt. 
\end{abstract}
\newpage
\tableofcontents
\newpage
\section{Introduction}
A simplicial complex $X$ on a finite ground set $[n]=\{1,\dots,n\}$
is a downwards closed collection of subsets of $[n]$, i.e., if $\tau\in X$ and $\sigma \subset \tau$, then $\sigma \in X$. 
The elements of $X$ are called {\bf faces}, and the maximal faces are called {\bf facets}. 
We say that a face $\tau$ is of dimension $k$ if $|\tau| = k$ and write $\dimfn(\tau) = k$\footnote{Note that this differs from the typical topological definition of dimension for faces of a simplicial complex.}. %\textcolor{blue}{(JL: typically the dimension is annoyingly 1 less than the size.)\textcolor{red}{Sh: That is ok I think we can shift by 1 to make our life easier}}
A simplicial complex $X$ is a {\bf pure} $\dimension$-dimensional complex if every facet  has dimension $\dimension$. 
%In this paper, all  simplicial complexes are assumed to be pure. 

Given a  $\dimension$-dimensional complex $X$,
for any   $0 \leq i \leq \dimension$,    define  $X(i) = \{\tau \in X : \dimfn(\tau) = i\}$.  Moreover, the co-dimension of a face $\tau \in X$ is defined as $\codim (\tau)  = \dimension - \dimfn (\tau)$. For a face $\tau \in X$, define the  {\bf link} of $\tau$ as the simplicial complex $X_{\tau} = \{\sigma \setminus \tau : \sigma \in X, \sigma \supset \tau\}$. 
%Note that $X_\tau$ is a $(\codim(\tau) -1)$- dimensional complex.
\par
A  $\dimension$-partite  complex is a $\dimension$-dimensional complex such that $X(1)$ can be (uniquely) partitioned into sets $\prt_1\cup \dots \cup \prt_{\dimension}$ such that for  every facet  $\tau \in X (\dimension)$,  we have $|\tau \cap \prt_i|=1$ for all $i \in [\dimension]$. Let $\mu$ be a probability distribution on facets of $X$. We denote this by $(X,\mu)$. 

Given a complex $(X,\mu)$, the (weighted)  1-skeleton of $X$, $G_\varnothing$, is the weighted graph with vertex set $X(1)$ and edge set $X(2)$ where the weight of an edge $$w(F,G)=\P_{\sigma\sim\mu}[F,G\in\sigma].$$
We let $A_\varnothing$ be the adjacency matrix of this graph, and $P_\varnothing$ be the transition probability matrix of the simple random walk. More generally, for any face $\tau\in X$ with $\codim(\tau)\geq 2$, we let $A_\tau, P_\tau$ be the adjacency matrix and the random walk matrix of the 1-skeleton of link of $\tau$. Note that for any $\tau \in X$, $\mu$ induces a probability distribution, $\mu_{|\tau}$ on the facets of $X_\tau$ where the probability of a facet $\sigma'$ is $\P_{\sigma\sim\mu}[\sigma'\in\sigma | \tau\in \sigma]$. Lastly, we say $X$ is a {\bf connected} complex if for every $\tau$ of co-dimension at least 2, the 1-skeleton of $X_\tau$ is a connected graph.

\begin{definition}[(Top-link) Spectral Expanders]
We say that $(X,\mu)$ is an $\alpha$-local spectral expander if for any link $\tau$,
 $\lambda_2(P_\tau)\leq \alpha$.
 
    We say that $(X,\mu)$ is a $\alpha$-top-link spectral expander if for any $\tau\in X$ with $\codim(\tau)=2$, $\lambda_2(P_\tau)\leq \alpha$.
\end{definition}

One of the surprising theorems in the theory of local spectral expanders is the trickle-down theorem.
\begin{theorem}[Oppenheim, \cite{Opp18}]\label{thm:oppenheim}
If $(X,\mu)$ is a connected $d$-dimensional $\frac{1-\eps}{d}$-top-link spectral expander then for any $\sigma$ with $\codim(\sigma)=k$, $\lambda_2(P_\sigma)\leq \frac{1-\eps}{d-(k-2)(1-\eps)}$. In particular, $X$ is a $\frac{1-\eps}{\eps d}$-local spectral expander.
\end{theorem}

Let $X$ be a $d$-partite complex with probability distribution $\mu:X(d)\to\R_{> 0}$ on its facets. We say $(X,\mu)$ is a {\bf path} complex if the parts (perhaps after renaming) satisfy the following {\bf conditional independence} property:
For any $1\leq i<k-1\leq d-1$ and $F\in T_i,K\in T_k$ and any face $\tau$ such that  $\tau\cap \bigcup_{j=i+1}^{k-1} T_j \neq \varnothing$ we have
$$ \P_{\sigma\sim\mu}[F,K\in\sigma|\tau\subset\sigma] = \P_{\sigma\sim\mu}[F\in\sigma|\tau\subset\sigma] \cdot \P_{\sigma\sim\mu}[K\in\sigma|\tau\subset\sigma]. $$
We also define a weaker notion: we say $(X,\mu)$ is a {\bf top-link path} complex if the above condition is satisfied whenever $\codim(\tau) = 2$.

Observe that any link of a path complex is also a path complex, and any link of a top-link path complex is a top-link path complex.

%the parts (perhaps after renaming)  satisfy the following: For any $1\leq i<k-1\leq d-1$, $F\in T_i,K\in T_k$, and $\tau$ with $\codim(\tau)=2$ and $\tau\cap T_i=\tau\cap T_k=\varnothing$ we have
%$$ \P_{\sigma\sim\mu}[F,K\in\sigma | \tau\subset\sigma] =\P_{\sigma\sim\mu}[F\in\sigma | \tau\subset\sigma] \cdot \P_{\sigma\sim\mu}[K\in\sigma | \tau\subset\sigma].  $$

The trickle-down theorem was recently improved for path complexes (and more generally for partite complexes):
\begin{theorem}[\cite{AO23}]
    There is a universal constant $c>0$ such that if $(X,\mu)$ is a connected $d$-partite top-link path complex which is a $c$-top-link expander, then for any $\tau\in X$ we have $\lambda_2(P_\tau)\leq \frac{O(c)}{\codim(\tau)}$.
\end{theorem}

Our main result is to obtain the ``optimal'' trickle-down constant for path complexes.

\begin{theorem}[Trickle-down Theorem for Path Complexes]\label{thm:main}
    If $(X,\mu)$ is a connected $d$-partite top-link path complex which is a $1/2$-top-link spectral expander then $X$ is a $1/2$-local spectral expander. % for $\alpha\leq \pi/2-1$.
\end{theorem}
%We conjecture that the optimal $\alpha$ in the above theorem is $\alpha=1/2$. %We note that $\alpha=1/2$ is attained for 
Our proof uses the Hereditary Lorentzian polynomial machinery of \cite{BL23}. Most surprisingly the proof unifies and generalizes two different types of trickle-down theorem: one in combinatorics/Hodge theory, and the Oppenheim trickledown theorem cited above. These machineries has been used to prove  the log-concavity of the coefficients of characteristic polynomial of matroids \cite{BL21} and to analyze mixing time of Markov chains to sample edge colorings of graphs \cite{ALO22}.

We prove the optimality of our theorem by showing that the constant $1/2$ is tight.

\begin{restatable}[Lower Bound]{theorem}{THMlowerbound}\label{thm:lower-bound}
For any $\eps>0$ and (even) $d$, there is a $d$-partite path complex $(X,\mu)$ that is a $\frac{1+\eps}{2+\eps}$-top-link spectral expander  but $ \lambda_2(P_\varnothing)\geq 1-\frac{4}{\eps(1+\eps)d}$.
\end{restatable}

For a $d$-partite path complex $(X,\mu)$ the down-up Markov chain is defined as follows: Given a facet $\sigma=(F_1,\dots,F_d)\in X(d)$ where $F_i\in T_i$, choose $i\in [d]$ uniformly at random, and for any $F\in T_i$ such that $\sigma'=\{F_1,\dots,F_{i-1},F,F_{i+1},\dots,F_d\}\in X(d)$ move to $\sigma'$ w.p. proportional to $\mu(\sigma')$. We let $P^\vee_d$ be the transition probability matrix of this chain.

It is not hard to see (assuming connectivity of $X$) that $\mu$ is the unique stationary distribution of the chain, i.e., $\mu^\top P^\vee_d=\mu^\top$. So, to generate samples from $\mu$, we just need to study its mixing time.
For a Markov chain with transition probability matrix $P^\vee_d$ and stationary distribution $\mu$ 
the {\bf total variation mixing time} started at a state $\tau$ is defined as follows:
	\[ t_{\tau}(\epsilon)=\min\{t\in \Z_{\geqslant 0}: \norm{{P^\vee_d}^t(\tau,\cdot)-\mu}_1 \leqslant \epsilon\},\]
 	%where $P^t(\tau,\cdot)$ is the distribution of the chain started at $\tau$ at time $t$.

\begin{theorem}[Main Algorithmic]\label{thm:mainalg}
    If $(X,\mu)$ is a connected $d$-partite path complex which is a $1/2$-top-link spectral expander then  the spectral gap of the  down-up walk $P^\vee_d$  is at least $\frac{4}{d(d+1)^2}$. %\geq d^{-4.447}$ where  $\gamma\geq\max\{3, \frac{2}{1-\alpha}(1-2^{-\gamma})\}$. 
    So, for any state $\tau\in X(d)$, 
    $$ t_{\tau}(\eps)\leq \frac{d(d+1)^2}{4}\cdot \log \frac{1}{\eps\cdot\mu(\tau)}$$
\end{theorem}
See \cref{sec:markovmixing} for the relation between the spectral gap and the mixing time.
We remark that the bound $\Omega(1/d^3)$ on the spectral gap is tight (up to constants). A tight example can be constructed by a distributive lattice (see the following section). 
Note also that this is the main theorem where we need the simplicial complex to be a path complex, and not just a top-link path complex.
%We remark that in the special case that $X$ is a distributive lattice (see below for the definition) the Poincar\'e constant of the down-up walk is at least $\frac{1}{d^3}$ in the worst case which is slightly larger than our bound above. 

To put the above result in perspective, this shows that similar to a recent long line of works which bound mixing time by proving spectral independence, the $1/2$-top-link path complexes are "weakly" spectrally independent as $\lambda_2(P_\varnothing)$ is bounded away from 1. However, unlike most recent results \cite{ALO24,CLV21,AJKPV22}, $\lambda_2(P_\varnothing)$ is only a constant as opposed to $O(1/d)$. Such a bound in general is not enough to bound mixing time of the down-up walk. Nonetheless, as we explain in \cref{sec:localtoglobalpath}, for path complexes such a bound is all one needs. In fact, unlike most recent works, one may not be able to obtain a (near) linear time mixing time bounds by establishing entropic independence for path complexes. This is because as we see in the next section there are families $d$-dimensional $1/2$-top-link path expanders for which the down-up walk mixes in $\Omega(d^3\log d)$-steps. 

\subsection{Applications to Sampling and Counting}
\paragraph{Lattices.} A lattice is a Poset with a partial order $<$ and meet and join operations $\wedge,\vee$. We call the elements of the lattice {\bf flats}. We say the lattice $\cL$ has a minimum and maximum elements if there are flats $\hat{0},\hat{1}$ such that $\hat{0}\leq F$ and $F\leq {\hat 1}$ for all $F\in \cL$. A proper flat is flat not equal to $\hat{0}$ or $\hat{1}$. All lattices that we consider in this paper have minimum and maximum elements. We say $G$ covers $F$, i.e.,  $F\prec G$, if $F<G$ but  there is no $K\in \cL$ such that $F<K<G$.

We say $\cL$ is ranked/graded if it can be equipped with rank function $\rk:\cL\to\mathbb{Z}_{\geq 0}$ such that $F\leq G$ implies $\rk(F)\leq \rk(G)$ and $F\prec G$ implies $\rk(G)=\rk(F)+1$.
Given a lattice $\cL$ a chain, also called a flag of flats, is a set of proper flats $(F_1, F_2, \dots, F_\ell)$ for some $\ell\leq d$ satisfying $F_1<\dots<F_\ell$. If $\cL$ is of rank $d+1$, i.e., $\rk(\hat{1})=d+1$, a maximal chain $(F_1,\dots,F_d)$ satisfies $\hat{0}<F_1<\dots <F_d<\hat{1}$.

Given a ranked lattice $\cL$, one can associate a $d$-partite complex $X$ to $\cL$: We let $X(1)$ be the set of proper flats of $\cL$ and we let faces of $X$ correspond to all (unordered) flags of flats of $\cL$.
The following fact is immediate:
\begin{fact}
    For any ranked lattice $\cL$, and the uniform distribution $\mu$ over the maximal flags of flats of $\cL$ or any of its external fields, $\psi*\mu$, $(X,\mu)$ and $(X,\psi*\mu)$ are path partite complexes.

    Recall that given $\psi:\cL\to \R_{>0}$, the probability of a maximal flag of flat $\hat{0}< F_1 < \dots < F_d<\hat{1}$ in the external field $\mu*\psi(\{F_1,\dots,F_d\})\propto \prod_{i=1}^d \psi(F_i).$ 
\end{fact}
We remark that this fact holds more generally for any ranked Poset.

\subsubsection{Distributive Lattices}
\begin{definition}[Distributive Lattices]
A lattice $(\cL,\vee,\wedge)$ is distributive if for any three flats $F,G,K$ we have $F \wedge (G \vee K) = (F \wedge G) \vee (F \wedge K)$.
\end{definition}

Given a Poset $(P,<)$ with $|P|=d+1$ elements, 
a set $T \subseteq P$ is a downset of $P$ if, for any $i \in T$ and $j < i$, then $j \in T$.
%an anti-chain is a $S\subseteq P$ such that any $i,j\in S$ are incomparable. \textcolor{orange}{KL: Is there a particular reason we use this definition versus the standard if $j \leq i$ and $i \in T$ then $j \in T$?} \phantom{} A set $T\subseteq P$ is a downset of $P$ if there is an antichain $S$ such that for every $i\in T$, there is a $j\in S$ such that $i\leq_P j$. 

A linear extension of $P$ is a bijection $\ell:P\to[d+1]$ such that for any $a,b\in P$ with $a<b$ we have $f(a)<f(b)$.

Birkhoff's representation theorem states that any finite distributive lattice is isomorphic to downsets of a Poset $P$. Consequently, maximal chains of $\cL$ are in bijection with linear extensions of $P$. In particular, the maximal chain corresponding to a linear extension $\ell$ is 
$$(\{\ell^{-1}(1)\}, \{\ell^{-1}(1),\ell^{-1}(2)\},\dots, \{\ell^{-1}(1),\dots,\ell^{-1}(d)\}).$$

The problem of sampling/counting the number of linear extensions of a Poset has been intensely studied around two decades ago with two general methods: (i) Counting the number of linear extensions by computing the volume of the order polytope of $P$ \cite{DFK91}, (ii) Sampling a uniformly random linear extension by running the Karzanov-Khachiyan's chain \cite{KK91}. Here we focus on the the second. To bound the mixing time of the Markov chain, the first method of choice was to use the isoperimetric inequality of the order polytope \cite{KK91} and the second successful approach was the (path) coupling method \cite{FW97,BD99,Hub14}.

Here, we prove the spectral independence framework of \cite{ALO24} can be used to bound the mixing time of the chain. In fact, we prove a stronger statement:
%Let $\cL$ be a distributive lattice and $P$ be the corresponding Poset with $d+1$ elements. 
We say a weight function $\psi:P\to\R_{> 0}$ is {\bf order-reversing} if for any $i\leq j$ we have $\psi(i)\geq \psi(j)$. Given such a $\psi$, the weight of a linear extension $\ell:P\to [d+1]$ is $\psi(\ell):=\prod_{i\in P} \psi(i)^{d+1-\ell(i)}$.

%For a concrete example, consider a total ordering of elements of $P$ consistent with $P$, say $a_1, \dots, a_{d+1}$ and let $\psi(a_i)=c^{-(d+1-i)}$

Let $\mu$ be the uniform distribution over all maximal flag of flats.
We also define the weight of a flat $F\in \cL$ as $\psi(F):=\prod_{i\in F}\psi(i)$.
%and for a maximal chain $(F_1\subsetneq \dots\subsetneq F_d)$ let $\mu(F_1,\dots,F_d)\propto \prod_i \psi(F_i).$ 
It follows that the weight of a maximal chain $\psi*\mu(F_1,\dots,F_d)$ is the same as the weight of its corresponding linear extension (up to a normalizing constant).
%Let $X$ be the $d$-dimensional, $d$-partite complex with facets corresponding to maximal chains of $\cL$ with part $T_i$ corresponding to flats of rank $i$. It follows that $(X,\mu)$ is a path complex. Note that that in the case that $\psi(i)=1$ for all $i$, i.e., $\mu$ is the uniform distribution  over maximal chains of  $\cL$, 
Let $(X,\psi* \mu)$ be the corresponding path complex. It turns out that the down-up Markov chain is essentially the same as the Karzanov-Khachiyan's  \cite{KK91}.

\begin{theorem}\label{thm:distributivelattice}
    Given a $(d+1)$-dimensional distributive lattice $\cL$ with the corresponding Poset $P$ and an order-reversing weight function $\psi:P\to\R_{> 0}$, %Let $\mu$ be the corresponding distribution on the maximal chains of $\cL$.
    then, the path complex $(X,\psi*\mu)$ is a $1/2$-top-link expander.
    
    %Assume, that perhaps after renaming we have $\psi(1)\geq \psi(2)\geq \dots\geq \psi(d+1)$.
    Consequently, $(X,\mu)$ is a $1/2$-local spectral expander  and the down-up walk on facets of  $(X,\mu)$ started
    at the chain with largest weight %$\{1\}\subsetneq \{1,2\}\subsetneq\dots \subsetneq\{1,\dots,n\}$ 
    mixes in time $O(d^{4}\log d)$.
\end{theorem}
For a concrete application, suppose we partition $P$ into two sets $P=A\cup B$ such that there is no $a\in A, b\in B$ where $b<a$; think of the set $A$ as the set that we want to {\em prioritize} in our linear extension. Then, for $c>1$, $\psi(a)=c,\psi(b)=1$ for all $a\in A$ and $b\in B$ is order-reversing. Then, by the above theorem we can sample a linear extension $\ell$ of $P$ with weight proportional to $c^{\sum_{a\in A} -\ell(a)}$. In other words, the more the elements of $A$ arrive late in the extension $\ell$ the less likely we are to choose $\ell$.

Given a (distributive) lattice $\cL$ and $F<G\in \cL$ the {\bf interval sub-lattice} $\cL_{[F,G]}$ is defined as the lattice consisting of all flats $K$ such that $F\leq K\leq G$. It follows that any interval sub-lattice of a distributive lattice is also a distributive lattice. With this and the conditional independence property of path complexes, to prove the above theorem it is enough to show that the (path) complex associated to a distributive lattice of rank 3 is a $1/2$-spectral expander (see \cref{thm:quadraticcheck}).
% \textcolor{blue}{(JL: do we need to mention the something about the non-contiguous codim 2 links?)}\textcolor{red}{Sh: HOw about this?}
% \begin{fact}
%     Let $X$ be the simplicial complex on all chains of a rank $d$ lattice $\cL$ and $\mu$ be a uniform distribuiton on the facets of $X$ (or an "external field"). Then, $(X,\mu)$  form a path complex.
% \end{fact}

\subsubsection{Modular Lattices}
\begin{definition}[Modular Lattices]
    A ranked lattice $\cL$ is modular if for any two flats $F,G\in \cL$ we have 
    $$ \rk(F)+\rk(G)=\rk(F\vee G)+\rk(F\wedge G).$$
\end{definition}

%Given a ranked lattice $\cL$ and $F,G\in \cL$ with $F<G$, the {\bf sub-lattice}  $\cL_{[F,G]}$ is the ranked lattices which contains all flats $H$ with $F<H<G$.

\begin{definition}[Typical Modular Lattices]
Given a modular lattice $\cL$, we say $\cL$ is typical if it satisfies the following:
    For any interval sub-lattice $\cL_{[F,G]}$ with $\rk(G)=\rk(F)+3$, the bipartite graph $G=({\cal F}_1,{\cal F}_2,E)$ between rank 1 and rank 2 flats of $\cL_{[F,G]}$ either has no vertex of degree 1 in the ${\cal F}_1$ side or no vertex of degree 1 in the ${\cal F}_2$ side, or exactly one degree 1 vertex in both of ${\cal F}_1,{\cal F}_2$ sides.
\end{definition}
It is not hard to see that any distributive lattice is modular. 
\begin{theorem}\label{thm:modularlattice}
    Let $X$ be the complex of all chains of a modular lattice and 
    let $\mu$ be the uniform distribution over the facets of $X$, i.e., maximal chains of  $\cL$.
    The path complex $(X,\mu)$ is a $1/2$-top link expander iff $\cL$ is typical. Consequently,   the down-up walk on the maximal chains of a typical modular lattice $\cL$ mixes in time $O(d^{3}\log |X(d)|)$.
\end{theorem}

\subsection{Applications to Log-Concave Sequences}\label{sec:introcolor}
\paragraph{Colored (Top-Link) Path Complexes.}
%We generalize the definition of a path complex $(X,\mu)$ as follows:
Given a (top-link) path complex $(X,\mu)$ with types $T_1,\dots, T_d$, for $\{F,G\}\in X(2)$, we write $F<G$ if $F\in T_i,G\in T_j$ and $i<j$. Otherwise we write $G<F$. If $\{F,G\}\notin X(2)$ we say they are incomparable. We also assume $\hat{0}<F<\hat{1}$ for any $F\in X(1)$.

We color elements of $X(1)$ with a strictly order-preserving map.
Let $\phi:X(1)\cup \{\hat{0},\hat{1}\} \to\R_{>0}$ be a  strictly order-preserving map, i.e., $F<G$ implies $\phi(F)<\phi(G)$. We call the triplet $(X,\mu,\phi)$ a colored path complex.

%\textcolor{blue}{(JL: I am pretty sure something is off about the definition of $D$)}

\begin{definition}
    We say a colored top-link path complex $(X,\mu,\phi)$ is a top-link expander if, for any $\tau$ of $\codim(\tau)=2$ where $\tau\cap T_i=\tau\cap T_{i+1}=\varnothing$ for some $i\in [d-1]$ we have 
    $$ A_\tau\preceq D_{\phi,\tau} +\mathbf{ww}^T,$$
    for some real vector $\mathbf{w}$ and  as before $A_\tau$ is the (weighted) adjacency matrix of $G_\tau$, and $D_{\phi,\tau}$ is the diagonal matrix where the $(F,F)$ entry is $\sum_{G: F<G} \frac{\phi(L)-\phi(G)}{\phi(L)-\phi(F)}\mu(\tau\cup\{F,G\})$ for any $F\in T_i$ and the $(G,G)$-entry is $\sum_{F:F<G} \frac{\phi(F)-\phi(K)}{\phi(G)-\phi(K)}\mu(\tau\cup\{F,G\})$. Here $K, L$ are the elements in $\tau$ satisfying $K \prec F$ and $G \prec L$ for all $F \in T_i$ and $G \in T_j$.
\end{definition}
For example, if for $F\in T_i$ we set $\phi(F)=i$, then the top-link expansion of $(X,\mu,\phi)$ is the same as $1/2$-top link expansion of $(X,\mu)$. We call this the trivial coloring.

We prove the following log-concavity theorem for colored top-link path complexes:
\begin{theorem}[Log-Concavity theorem for Colored Top-Link Path Complexes] \label{thm:log-concavity-colored-complex}
    For any connected $d$-dimensional colored top-link path complex $(X,\mu,\phi)$, if it is a top-link expander then the sequence $c_0,c_1,\dots,c_d$ is log-concave, where
    \[
        c_k = \frac{1}{\phi(\hat{1}) - \phi(\hat{0})} \sum_{\tau \in X(d)} \mu(\tau) \cdot \frac{\prod_{i=0}^d (\phi(\tau_{i+1}) - \phi(\tau_i))}{\prod_{i=1}^k (\phi(\tau_i) - \phi(\tau_0)) \prod_{i=k+1}^d (\phi(\tau_{d+1}) - \phi(\tau_i))},
    \]
    where $\tau_0 = \hat{0}$ and $\tau_{d+1} = \hat{1}$.
    %\textcolor{red}{Sh: I like this version, just the first constant coefficient $\frac{1}{\phi(1)-\phi(0)}$ can be dropped.} \textcolor{blue}{JL: The one upside to this term is that it makes $c_k$ homogeneous in $\phi$.}
\end{theorem}
Note that if $\phi$ is the trivial coloring,  the above theorem is trivial; it just corresponds to the log-concavity of the binomial coefficients: namely the log-concavity of the sequence $\binom{d}{0}, \binom{d}{1},\dots,\binom{d}{d}$.

In \cref{sec:coloredpathcomplexes} we give the following applications of the above theorem:
\paragraph{A  proof of the Heron-Rota-Welsh conjecture.} Let $M=([n], {\cal I})$ be a matroid over $[n]$ of rank $d+1$. Recall that a set $F\subseteq [n]$ is a {\bf flat} if $F=\textup{span}(F)$. A {\bf flag} of flats of $M$ is a sequence $F_1\subsetneq F_2\subsetneq \dots\subsetneq F_\ell$. A maximal of flat of flat is a flag of flat of length $d$.
We let $X$ be the simplicial complex of all flag of proper flats of $M$ and $\mu$ be the uniform distribution over facets of $X$, i.e., all maximal flag of flats of $M$. We prove the following lemma.
\begin{lemma}\label{lem:toplinkexpansionmatroid}
    For $\phi(F)=|F|$, the colored path complex $(X,\mu,\phi)$ is a top-link expander.
\end{lemma}
Then, we invoke \cref{thm:log-concavity-colored-complex} to give  prove the Heron-Rota-Walsh conjecture on the log-concavity of the coefficients of the characteristic polynomial of a matroid (see \cref{thm:Heron-Rota-Welsh}). We remark that although this proof is similar to \cite{BL21} in spirit, it is conceptually simpler and more direct. 

\paragraph{A generalization of the Stanley's log-concavity theorem.} Given a Poset $P$ with $d+1$ elements and the corresponding distributive lattice $\cL$ and the path complex $(X,\mu)$ where $\mu$ is the uniform distribution on maximal chains of flats of $\cL$. 
Chan and Pak raised the following conjecture which generalizes Stanley's log-concavity theorem (given by the $|A| = 1$ case).
%we prove the following lemma.

\begin{conjecture}[{\cite[Conj 1.5]{CP22,CP23}}]\label{conj:CP}
    For any $A\subseteq P$, we have
$$ \P[\ell_{\min}(A)=k]^2 \geq \P[\ell_{\min}(A)=k-1]\cdot\P[\ell_{\min}(A)=k+1],$$
where $\ell$ is a u.r. linear extension of $P$ and $\ell_{\min}(A)=\min\{\ell(a): a\in A\}$.
\end{conjecture}

Given $A\subseteq P$, consider the following coloring: Let $\phi_M(F)=\I[A\cap F\neq \varnothing]\cdot M+|F|$. We observe that the conclusion of \cref{thm:log-concavity-colored-complex} for $\phi_M$  implies the above conjecture when $M\to\infty$.

Unfortunately, $(X,\mu,\phi_M)$ is not necessarily a top link expander. Surprisingly, our counter example also gives a counter example to the above conjecture:
\begin{lemma}
    \cref{conj:CP} doesn't necessarily hold for $A\subseteq P$ arbitrary. 
\end{lemma}
\begin{proof}
Consider a Poset with $4$ elements, $b<c<d$ and an incomparable $a$. The conjecture fails for $A=\{a,d\}$ and $k=2$.     
\end{proof}

Furthermore, we show that if we limit the set $A$, we get a top-link expander and we prove the conjecture for such sets. 
\begin{definition}[$P$-consistent sets]
    We say a set $A\subseteq P$ is $P$-consistent if for any triple $a,b,c\in P$ with $b\prec c$ such that $a,c\in A, b\notin A$ we have $a$ is comparable to (at least) one of $b$ and $c$.
\end{definition}
\begin{lemma}\label{lem:toplinkexpansioncolorlattice}
For any function $\phi:P\to \R_{>0}$, and $\phi(F)=\sum_{a\in F} \phi(a)$, the colored path complex $(X,\mu,\phi)$ is a top-link expander.

In addition, for any $A\subseteq P$ that is $P$-consistent,
     and for any $M\in \R_{>0}$, $(X,\mu,\phi_M)$ is a top-link expander where $\phi_M(F)=\I[F\cap A\neq\emptyset]\cdot M+|F|$ 
\end{lemma}
Having this, we invoke \cref{thm:log-concavity-colored-complex} to prove the following theorem:
\begin{theorem} \label{thm:P-consistent-ext-Stanley}
For any $P$-consistent $A\subseteq P$ and any $2\leq k<d$, we have
$$ \P[\ell_{\min}(A)=k]^2 \geq \P[\ell_{\min}(A)=k-1]\cdot\P[\ell_{\min}(A)=k+1],$$
\end{theorem}
We remark that in the special case $A=\{a\}$ for some $a\in P$, the above theorem re-proves Stanley's log-concavity theorem \cite{Sta81}.
%The type of any face $\tau \in X$ is defined as $\type (\tau) = \{\ti \in \prtin : |\tau \cap \prt_\ti|=1\}$. 

\subsection{Beyond $1/2$-top-link expansion}
Over the last decades, there has been numerous construction of ``sparse'' local-spectral expanders, see e.g., \cite{LSV05,KO18,OP22,DLW24}.
Recall that a $d$-dimensional complex $X$ is sparse if for every $F\in X(1)$, the ``degree'' of $F$, i.e., 
the number of facets that contain $F$ is $O_d(1)$ and is independent of $|X(1)|$.

 To bound the local-spectral expansion, these constructions often increase the degree to make sure that the complex is $\alpha$-top-link expander, for some $\alpha<1/d$, and then appeal to the trickledown theorem \cref{thm:oppenheim}.
 
It turns out that most of these constructions are top-link path complexes, e.g., \cite{LSV05, KO18, HS24, OP22}. In the next theorem, we show that if the top-link path complex $(X,\mu)$ is a $(1/2-\eps)$-top-link spectral expander, then for every face $\sigma$ of co-dimension $k$,  $\lambda_2(P_\sigma)\leq 1/(k-1)\sqrt{\eps}$. 
So, our theorems can be used to make these recent constructions even sparser and get nearly tight bounds on their local spectral expansion. 

Given an integer $k \geq 0$ and a parameter $s > 1$, we define the $s$-analog of $k$ and of $k!$ via
\[
    [k]_s = \frac{s^k - s^{-k}}{s - s^{-1}} \qquad \text{and} \qquad [k]_s! = \prod_{i=1}^k [i]_s.
\]
These numbers can be thought of as a ``symmetric'' version of the more standard $q$-analog of the non-negative integers (e.g., note that $[k]_s \to k$ as $s \to 1$).
\begin{theorem} \label{thm:main-s}
For any connected  $d$-dimensional top-link path complex $(X,\mu)$, if it is a $1/[2]_s$-top-link expander then for any $\sigma$ with $\codim(\sigma)=k$ we have
\[
    \lambda_2(P_\sigma) \leq \frac{2}{(k-1)(s^{\frac{k+1}{2}}+1)} \left(\frac{s^{\frac{k-1}{2}}-1}{s-1}\right).
\]
% $$ \lambda_2(P_\sigma)\leq \frac{1}{k-1} \max_{1\leq i\leq k} \alpha_i^{-1}(s)\sum_{j=1}^{i-1} \alpha_j(s) + \alpha_i(s)\sum_{j=i+1}^k\alpha_j^{-1}(s),$$
% where $\alpha_j(s):=\sqrt{[j]_s/[k-j+1]_s}$.
%where $[i]_s=\frac{s^k-s^{-k}}{s-s^{-1}}$
For a concrete bound, if $(X,\mu)$ is a $(1/2-\eps)$-top link expander then $\lambda_2(P_\sigma)\leq\frac{1-2\epsilon}{(k-1) \sqrt{\epsilon}}.$
\end{theorem}
The above theorem should be contrasted with \cref{thm:main} where we assumed the tight $1/2$-top link expansion. 
We remark
% that the bound $\frac{1-2\eps}{(k-1)\sqrt{\eps}}$ is not tight, but the bound in terms of $s$ is tighter. In particular, note
that as $s \to 1$ the bound approaches $\frac{1}{2}$ and as $s \to \infty$ the bound approaches $0$. Note that the above bound is a worst-case bound for $1/[2]_s$-top-link expanders; in fact better bounds can be achieved for $\lambda_2(P_\sigma)$ for specific family of partite/path complexes, as in \cite{GP19}.

\subsection{Discussion}
Over the last few years the method of log-concave polynomials was developed to relate four distant areas of math and computer science, namely (i) Hodge theory and log-concave sequences, (ii) Theory of high dimensional expanders, (iii) Geometry of polynomials and (iv) Analysis of Markov chains. We refer interested readers to the sequence of works \cite{AOV21, ALOV24, BH20, ALOV24-2}. 
During this time, several groups have exploited geometry of polynomials techniques to discover new families of log-concave sequences without appealing to algebraic geometric arguments and intuitions. In parallel, several groups of researchers have developed tools to analyze local spectral expansion of high-dimensional simplicial complexes, and used them to analyze long-standing open problems on mixing time of Markov chains (see e.g., \cite{ALO24,Liu21}).
% We emphasize that although
With these and other important developments, these four areas proved to be closely related, at least in the specific case of
%studying
the uniform distribution over bases or independent sets of a matroid.
That said, unfortunately most of the tools are inherently limited to matroids: for example, it was shown that the support of any homogeneous log-concave polynomial must be the set of bases of a matroid. 
% such as (matrix) trickledown, coupling, and correlation decay.
Thus it has been unclear how far the connections of the above four areas would extend.

In this paper, at a high-level, our goal is to prove a deeper and more general connection between these four areas. To do this, we study a family of Lorentzian/log-concave polynomials called $\cC$-Lorentzian (see \cref{def:C-Lorentzian}) which are only log-concave w.r.t. vectors in a given convex cone. This allows us to bypass the support constraint of log-concave polynomials studied in \cite{ALOV24,BH20}.

The main non-trivial part is how to construct these polynomials such that they can be used to derive bounds on local-spectral expansion or log-concave sequences w.r.t to a simplicial complex.
Our first technical contribution is the recursive definition of these  polynomials using a carefully chosen family of linear maps, $\bm{\pi}$, which map the cone of a polynomial to the cone of its partial derivatives. Our maps are motivated by the construction in \cite{BL23}, but they are generalized/simplified far beyond: we construct $\cC$-Lorentzian polynomials associated to several types of lattices and path complexes that we discussed in the preceding sections. \cref{thm:conn+quad=Lor} shows that to prove $p$ is $\cC$-Lorentzian, it is enough to check the quadratics. 

We note that all facets of the complex $(X,\mu)$ show up as the coefficients of multi-linear terms of our polynomial. But, the polynomial has many (non-linear) monomials with negative coefficients. These ``extra'' monomials weaken the necessary condition for quadratics to be Lorentzian: Even if the complex $(X,\mu)$ is only a $1/2$-top link expander (as opposed to being $0$-top link expander), the quadratics are still Lorentzian, and by induction our polynomial is $\cC$-Lorentzian. The drawback is that, unlike with log-concave polynomials, we cannot ``read off'' the adjacency matrix of the link of the empty-set, $A_\varnothing$ from the Hessian of $p$. Instead, we can only read off the adjacency matrix of the induced graph of $A_\varnothing$ between two types, $T_i,T_j$ after applying certain directional derivatives (see \cref{thm:eig-bounds-main-v1}). We then have to use properties of the (path) complex $(X,\mu)$ to derive bounds on the eigenvalues of $P_\varnothing$ (see \cref{lem:dpartiteeigenvalues}).

Putting these altogether, we prove/unify (new) mixing time bounds, as well as (new) log-concave sequences, all using a {\bf single unified machinery}.
In addition, we  prove significantly improved bounds on spectral expansion of (sparse) top-link path complexes.
%, a.k.a., Ramanujan complexes. 

%\textcolor{blue}{and analysis (somehow the Hessian/general log-concavity analysis is also a new contribution; how to say itte)} 
%\textcolor{red}{Sh: HOw does this sounds now?}
%In particular,
%As we discussed in the preceding sections,
%the main contribution of the paper is that, by defining $\cC$-Loretnzian polynomials using our $\bm{\pi}$-maps,

We also note that although our techniques may seem limited to path complexes, in an up-coming manuscript %\textcolor{blue}{(JL: what are you trying to say beyond ``in future work''?)} 
we will extend our constructions to {\bf general partite complexes}.

We finish this section with several open problems that may be approachable using our techniques:
\begin{enumerate}
    \item Given a matroid $M$, consider the path complex $(X,\mu)$ where faces of $X$ corresponds to flags of flats of $M$ (as discussed above). For a facet $\sigma=F_1\subsetneq \dots \subsetneq F_d$, let $\mu(\sigma)\propto \prod_{i=1}^d \frac{|F_{i+1}|-|F_i|}{|F_{i+1}|}$. Can we use the machinery developed in this paper to prove the down-up walk Markov chain mixes in time polynomial in $d$? We remark that the partition function of $\mu$, is exactly the last coefficient of the characteristic polynomial of $M$. So such a result would solve several open problems in counting and sampling as discussed in \cite{ALO23}.
    \item Given a matroid $M$, it is conjectured that the sequence, $c_0,\dots,c_d$, where $c_k$ is the number of flats of rank $k$ is log-concave. Can we use \cref{thm:log-concavity-colored-complex} to prove this conjecture?
\end{enumerate}

\subsection{Acknowledgements}
We would like to thank Dorna Abdolazimi and Kuikui Liu for stimulus discussions in early stages of this project. We also thank Max Hopkins and Igor Pak for valuable comments to earlier versions of this manuscript.

The first author acknowledges the support of the Natural Sciences and Engineering Research Council of Canada (NSERC), [funding reference number RGPIN-2023-03726]. Cette recherche a \'et\'e partiellement financ\'ee par le Conseil de recherches en sciences naturelles et en g\'enie du Canada (CRSNG), [num\'ero de r\'ef\'erence RGPIN-2023-03726].
The second author's research is supported by funding from the European Research Council (ERC) under the European Union’s Horizon 2020 research and innovation programme (grant agreement No 815464).
The last author's research is supported by  an NSF grant  CCF-2203541, a Simons Investigator Award and a Lazowska Endowed Professorship in Computer Science \& Engineering.

\subsection{Organization}
 The rest of this manuscript is organized as follows: We will prove \cref{thm:mainalg} in \cref{sec:localtoglobalpath}.
In \cref{sec:lorentzianpolynomials} we define the general notion of commutative $\bm{\pi}$-maps and use it to prove \cref{thm:conn+quad=Lor}, which says: to prove that our polynomial $p$ is $\cC$-Lorentzian it is enough to check $p$ is irreducible and the quadratics are log-concave. In \cref{sec:lorentzianpolynomials-path} we define a general class of $\bm{\pi}$-maps for path complexes based on a choice of $\bm{\alpha},\bm{\beta}$ vectors. We then use this to prove a log-concavity statement (\cref{cor:log-concavity-alpha-beta}) and a bound on eigenvalues of (induced subgraphs of) $P_\varnothing$ for path complexes (\cref{cor:eig-bound-alpha-beta}) by invoking \cref{thm:conn+quad=Lor}.
In \cref{sec:coloredpathcomplexes} we further specialize $\bm{\alpha},\bm{\beta}$ for a colored path complex $(X,\mu,\phi)$ and use this together with \cref{cor:log-concavity-alpha-beta} to prove \cref{thm:log-concavity-colored-complex}. In \cref{sec:polys-v1} we specialize to the coloring $\phi(F)=i$ when $F\in T_i$ (for all $1\leq i\leq d$) and use that together with \cref{cor:eig-bound-alpha-beta} to prove \cref{thm:main}.
In \cref{sec:scolor} we specialize $\bm{\alpha},\bm{\beta}$ using the $s$-analog of the non-negative integers, and we use this to prove \cref{thm:main-s}.
In \cref{sec:quadraticcheck} we prove various path complexes are indeed top link expanders. In particular we prove \cref{thm:distributivelattice},\cref{thm:modularlattice,lem:toplinkexpansioncolorlattice}.
Finally, we prove \cref{thm:lower-bound} in \cref{sec:lowerbound}.
\section{Preliminaries}
\subsection{Notation}
Given a simplicial complex $X$, we use greek letters $\sigma,\tau,\omega$ to denote faces of $X$ and we use capital letters $F,G,H$ to denote vertices of the complex, i.e., elements of $X(1)$.
For a facet $\sigma\in X(d)$ and $i\in [d]$, we write
$$ \sigma_i:=\sigma\cap T_i$$
For two sets $A,B$ and a matrix $P\in \R^{A\times B}$ we write $P(a,b)$ to denote the $(a,b)$-entry of $P$ for $a\in A,b\in B$.

For a probability distribution $\mu$ on $X(d)$ and a face $\sigma\in X$ we let $\mu_\sigma$ to denote the conditional distribution $\mu | \sigma$. 
For $i_1,\dots,i_k\in [d]$ we let $\mu^{i_1,\dots,i_k}$ denote the marginal distribution of $\mu$ on $F_{i_1}\in T_{i_1},\dots,F_{i_k}\in T_{i_k}$.

\subsection{Linear Algebra}
\begin{fact}\label{lem:AB}The following facts are classical:
    \begin{itemize}
    \item  Cauchy's Interlacing Thm: For symmetric matrices $A,B\in \R^{n\times n}$ and $v\in \R^n$, if $A\preceq B+vv^T$, the $\lambda_2(A)\leq \lambda_1(B)$.
    \item The adjacency matrix of any (unweighted) complete multi-partite graph has exactly one positive eigenvalue.
    \item If a symmetric $A\in \R^{n\times n}$ has one positive eigenvalue, then for any $P\in \R^{k\times n}$ then $PAP^\top$ has at most one positive eignevalue. 
    \item For any two matrices $A\in \R^{n\times k},B\in \R^{k\times n}$ non-zero eigenvalues of $AB$ are the same as the non-zero eigenvalues of $BA$ with the same multiplicity. 
    \item Let $G$ be a weighted star, i.e., a weighted bipartite graph where one side has a single vertex and let $P$ be the transition probability matrix of the simple random walk on $G$. Then, $P$ has only two non-zero eigenvalues $1,-1$.
    \item Let $P$ be the transition probability matrix of the random walk in a bipartite graph with parts $A,B$ such that $|A|+|B|\geq 3$. Then, $P$ is a block (anti-diagonal) matrix with blocks $P_{A\to B},P_{B\to A}$. Furthermore,
    $$  \lambda_2(P)^2=\lambda_2(P_{A\to B}P_{B\to A}).$$
    Note that if $P_{A\to B}P_{B\to A}$ is a $1\times 1$ matrix we let its second eigenvalue to be 0. 
    \item Rayleigh Quotient: Let $A\in \R^{n\times n}$ be a matrix which is self-adjoint w.r.t., an inner-product $\langle .,.\rangle$. Then, $$\lambda_k(A)=\max_S \min_{v\in S} \frac{\langle Av,v\rangle}{\langle v,v\rangle},$$
    where the maximum is over all $k$-dimensional subspaces  $S$ in the Hilbert space defined by the inner-product.
    \item Eigenvalues of the adjacency matrix of a bipartite graph come in pairs $\lambda,-\lambda$.
%    \item Let $A\in \R^{n\times n}$ be PSD operator that is self-adjoint w.r.t. $\mu$, i.e., $\langle f,Ag\rangle_\mu= \langle Af,g\rangle_\mu$. Let $f$ be the max eigenvector of $A$. Then, for any $g$ such that $\langle g,f\rangle_\mu=0$ we have
 %   $$  \frac{\norm{Ag}_\mu}{\norm{g}_\mu}\leq \lambda_2(A)$$
    \end{itemize}
\end{fact}
% \begin{proof}
%     We prove the third fact:
%     \begin{align*}
%     \frac{\norm{Ag}_\mu^2}{\norm{g}_\mu^2} = \frac{\langle Ag, Ag\rangle_\mu}{\langle g,g\rangle_\mu} = \frac{\langle A^2g,g\rangle_\mu}{\langle g,g\rangle_\mu} \leq \lambda_2(A^2) = \lambda_2(A)^2.
%     \end{align*}
%     The inequality follows by the Rayleigh quotient and that $\langle g,f\rangle_\mu=0.$
% \end{proof}

% \begin{lemma}
% Let $G=(V,E,w)$ be a tri-partite graph with parts $V=A\cup B\cup C$ and $E=E_{A,B}\cup E_{B,C}$, i.e., no edges between $A,C$.Let $P_{A,B}$ be the simple random walk matrix in the induced graph $G(A\cup B)$ and similarly $P_{B,C}$ be the walk in $G(B\cup C)$.
% Then,
% $$ \lambda_2(P_{A, B}P_{B, C})=\lambda_2(P_{A, B})\cdot\lambda_2(P_{B, C}).$$
% \end{lemma}
% \begin{proof}
%     First, observe that $P_{A,B}$ is a block-diagonal matrix with a block corresponding to $P_{A\to B}$ and another corresponding to $P_{B\to A}$. We have
%     %$$P_{A,B}P_{B,C}=\begin{bmatrix} 0 & P_{A\to B} P_{B\to C}\\ P^_{$$

% Therefore, by \cref{lem:AB} $\lambda_2(P_{i\to j}P_{j\to k}\P_{k\to j}P_{j\to i}) = \lambda_2(P_{ j\to i}P_{i\to j}\P_{j\to k}P_{k\to j})$

% \end{proof}
 Given a (weighted) graph $G=(V,E,w)$ with (weighted) adjacency matrix $A$; let $D$ be the diagonal matrix of vertex (weighted) degrees and let $P=D^{-1}A$ be the transition probability matrix of the simple random walk on $G$ and $\mu$ the stationary distribution of the walk, i.e., we have $\mu^T P =\mu$ and $P\bone=\bone$.
 For two functions $f,g:V\to\R$, let
  $$\langle f,g\rangle_\mu= \E_{x\sim\mu} f(x)g(x).$$ 
 It turns out that $P$ is self-adjoint  w.r.t. this inner product. So, by spectral theorem it has real eigenvalues with an orthonormal set of eigenfunctions. The largest eigenvalue is 1 corresponding to the all-ones eigenfunction.
 
 Given $f,g:V\to\R$ the Dirichlet form is defined as follows:
 $$ \cE_P(f,g) = \frac12\E_{x\sim\mu} \sum_y P(x,y) (f(x)-f(y))(g(x)-g(y))$$
 \begin{lemma}[Rayleigh Quotient]\label{lem:rayleigh}
     The second eigenvalue of $P$ is
     $$ \lambda_2(P)=1-\min_{f:\langle f,\bone\rangle_\mu=0} \frac{\cE_P(f,f)}{\langle f,f\rangle_\mu}$$ 
 \end{lemma}

The following lemma is classical, see e.g., \cite{WX97,MOA11}.
\begin{lemma}\label{lem:PSDProd}
    Let $A,B$ be two PSD self-adjoint matrices w.r.t $\langle .,.\rangle$. Then, for all $k\geq 1,$
    $$ \prod_{i=1}^k \lambda_i(AB) \leq \prod_{i=1}^k \lambda_i(A)\lambda_i(B). $$
\end{lemma}
% \begin{lemma}
%     Let $P$ be a matrix that is self-adjoint w.r.t. an inner product $\langle.,.\rangle$ and is PSD. Then,
%     $$ \max_x \frac{\norm{Px}_\mu}{\norm{x}_\mu}=\lambda_1(P).$$
%     Furthermore, if $Q$ is also self-adjoint and PSD then
%     $$ \lambda_1(PQ)\leq \lambda_1(P)\lambda_1(Q).$$
% \end{lemma}
% \begin{proof}
%     We start with the first assertion.
%     \begin{align*}
%         \frac{\norm{Px}^2}{\norm[x}^2}\frac{\langle Px, Px \rangle}{}
%     \end{align*}
% \end{proof}
Recall that for any graph $G=(V,E)$ the simple random walk $P$ on $G$ is reversible w.r.t. stationary distribution $\mu$: $\mu(x)P(x,y)=\mu(y)P(y,x)$ for all $x,y\in V$.
\begin{lemma}\label{lem:muPXY}
    Let $G=(X,Y)$ be a bipartite graph and $P$ be the simple random walk matrix with stationary distribution $\mu$. Let $P_{X\to Y}$ be the sub-matrix of $P$ corresponding to transitions from $X\to Y$ and $P_{Y\to X}$ is defined similarly. Then, $\tilde{\mu}(x)=\frac{\mu(x)}{\sum_{z\in X} \mu(z)}$ for any $x\in X$ is the stationary distribution of $P_{X\to Y} P_{Y\to X}$.
\end{lemma}
\begin{proof} Let $Z=\sum_{z\in X}\mu(x)$.
    We just check the reversibility condition. For any $x,z\in X$ we have
    \begin{align*} \tilde{\mu}(x)(P_{X\to Y}P_{Y\to X})(x,z)  &= \frac{\mu(x)}{Z} \sum_{y\in Y} P(x,y)P(y,z)\\
    &=\frac{1}{Z}\sum_{y\in Y} \mu(y) P(y,x) P(y,z)\\
    &=\frac{1}{Z}\sum_{y\in Y} \mu(z) P(z,y)P(y,x)
    =\tilde{\mu}(z)(P_{X\to Y}P_{Y\to X})(z,x)
    \end{align*}
    where we used reversibility of $P$ twice. 
\end{proof}
\begin{lemma}\label{lem:allonesshift}
Let $P$ be the transition probability matrix of the simple random walk on a (weighted) graph $G$ with stationary distribution $\mu$. Then, %for any $\alpha\geq 1$,
$$ \lambda_1(P- \bone\mu^{\intercal}) = \max\{0,\lambda_2(P)\}.$$
\end{lemma}
\begin{proof}
	First notice, $ (P-\bone\mu^{\intercal})\bone = \bone - \bone = 0$. We claim that any other eigenfunction of $P$ is also an eigenfunction of $P-\bone\mu^{\intercal}$.  For any eigenfunction $f\in \R^V$ of $P$ where $\langle f,\bone \rangle_\mu=0$, we have,
	$$ (P-\bone\mu^{\intercal})f = Pf  -  \bone\langle \bone,f\rangle_\mu = \lambda f.$$
	%where we used that $f$ is orthogonal to $\bone$.
\end{proof}

	% We let $\lambda^*(P)$ denote the second largest eigenvalue of $P$ in absolute value. That is, if $-1\leqslant \lambda_n\leqslant \dots\leqslant \lambda_1=1$ are the eigenvalues of $P$, then $\lambda^*(P)$ equals $\max\set{\abs{\lambda_2},\abs{\lambda_n}}$.
	% \begin{theorem}[{\cite[Prop 3]{DS91}}]\label{thm:mixingtime}
	% 	For any reversible irreducible  Markov chain $(\Omega, P, \pi)$,  $\eps>0$, and any starting state $\tau\in \Omega$,
	% \[ t_\tau(\eps)  \ \leqslant  \ \frac1{1-\lambda^*(P)}\cdot \log\parens*{\frac{1}{\eps\cdot \pi(\tau)}}.\]
	% \end{theorem}

For a weighted graph $G=(V,E,w)$ and a subset $S\subseteq V$ of vertices we let the {\bf conductance} of $S$, denoted by $\cond(S)$,  as 
	\[ \cond(S) \ = \ \frac{w(E(S,\overline{S}))}{\vol(S)} \ = \ \frac{\sum_{e\in E(S,\overline{S})} w(e)}{\sum_{v\in S} w(v)},\]
	where $\overline{S}:=V-S$, $E(S,\overline{S})=\{ \{u,v\}\in E:u\in S,v\notin S\}$ is the set of edges between $S$ and $\overline{S}$, $w(E(S,\overline{S}))$ is the sum of weights of these edges, and the volume $\vol(S)$ is the sum of the weighted degrees of the vertices in $S$. The conductance of $G$ is then
	\[ \cond(G)\ = \ \min_{S} \  \cond(S), \]
	where the minimum is taken over subsets $\varnothing \subsetneq S\subsetneq V$ for which $\vol(S)\leqslant \vol(\overline{S})$. 

	%We say $G$ is $d$-regular if $w(v)=d$ for all $v\in V$.
	\begin{theorem}[Cheeger's Inequalities \cite{AM85,Alon86}]\label{thm:Cheeger}
		For any  weighted graph $G=(V,E,w)$, 
		\[ \frac{1-\lambda_2(P)}{2} \ \leqslant  \ \cond(G) \  \leqslant \  \sqrt{2\left(1-\lambda_2(P)\right)}, \]
		%where $A_G$ is the weighted adjacency matrix of $G$ given by $(A_G)_{ij} = w(\set{i,j})$. 
	\end{theorem}

\begin{lemma}[One positive eigenvalue bound]\label{lem:oneposeigval}
    For $A \in \R^{n\times n}$ that is self-adjoint w.r.t. $\langle .,.\rangle_\mu$ and has one positive eigenvalue and for any vector $v \in \R^n$ such that $\langle v,Av\rangle_\mu > 0$, we have  for any $x\in \R^n$,
    $$\langle x,Ax\rangle_\mu \leq \frac{\langle x, Av\rangle^2_\mu}{\langle v,Av\rangle_\mu,}$$
%where we write $A\preceq_\mu B$ to denote $\langle x,Ax\rangle_\mu\leq \langle x,Bx\rangle_\mu$ for all $x$.
\end{lemma}
\begin{proof}
Let $D$ be a diagonal matrix corresponding to $\mu$, i.e., $D(i,i)=\mu_i$. Note that $DA$ is symmetric (as it is self-adjoint w.r.t. ordinary inner product).
    Let $x \in \R^n$ be arbitrary and set $B \in \R^{n \times 2}$ to $B = \begin{bmatrix}v, x \end{bmatrix}$. Then $$B^T DAB = \begin{bmatrix} \langle v,Av\rangle_\mu & \langle v,Ax\rangle_\mu\\ \langle x,Av\rangle_\mu & \langle x,Ax\rangle_\mu\end{bmatrix}$$ is a $2\times 2$ matrix with the first diagonal entry being non-negative so it has at least one non-negative eigenvalue. By Cauchy interlacing \cref{lem:AB}, $BDAB^T$ has at most one positive eigenvalue. Thus, the product of the eigenvalues is non-positive and so the determinant is non-positive meaning $\langle v,Av\rangle_\mu \cdot \langle x,Ax\rangle_\mu - \langle x,Av\rangle_\mu^2 \leq 0.$ Rearranging we obtain lemma's statement. %we get $$\langle x,Ax\rangle_\mu \leq \left\langle x,\left(\frac{(Av)(Av)^T}{\langle v,Av\rangle_\mu}\right)x\right\rangle$$ As $x$ was arbitrary we are done.
\end{proof}
\begin{lemma}[Eigenvalues of Bipartite graphs]\label{lem:eigenvaluebipartite}
Let $G=(X,Y)$ be a (weighted) bipartite graph, $A\in\R^{(X+Y)\times (X+Y)}$ be the adjacency matrix of $G$, $P$ be the transition probability matrix of the simple random walk on $G$, and $D$ be the diagonal matrix of vertex degrees, i.e., $P=D^{-1}A$. Let $S, \bar{S}\in \R^{(X+Y)\times (X+Y)}$ be diagonal matrices such that $S(x,x)=s_X, S(y,y)=s_Y$ for all $x\in X, y\in Y$ and $\bar{S} = \sqrt{s_Xs_Y} \cdot I_{X+Y}$. Then the following are equivalent:
\begin{enumerate}
    \item $D^{-1/2} A D^{-1/2}\preceq S+vv^T$ for some vector $v\in \R^{X+Y}$,
    \item $D^{-1/2} A D^{-1/2}\preceq \bar{S}+ww^T$ for some vector $w \in \R^{X+Y}$,
    \item $\lambda_2(P) \leq \sqrt{s_X s_Y}$,
    \item $\lambda_2(P-S) \leq 0$.
\end{enumerate}
\end{lemma}
\begin{proof}
$(1) \to (2)$: Let $Z\in \R^{(X+Y)\times (X+Y)}$ be a diagonal matrix with $Z(x,x)=z_X$ and $Z(y,y)=z_Y$ where $z_X,z_Y\geq 0$ and $z_X\cdot z_Y=1$. Assuming the first condition, we have
$$ D^{-1/2}AD^{-1/2} = Z(D^{-1/2} A D^{-1/2})Z %\underset{Z\succeq 0} 
    \preceq ZSZ + vv^T,$$
for some $v\in \R^{X+Y}$.
The first identity uses that $D^{-1/2}AD^{-1/2}$ has only two non-zero blocks corresponding to transitions from $X\to Y, Y\to X$ and that $z_X\cdot z_Y=1$.
Taking $z_X=\sqrt[4]{s_Y/s_X}, z_Y=1/z_X$, we get that $ZSZ=\sqrt{s_X s_Y}I = \bar{S}$.

$(2) \to (1)$: The same proof as for $(1) \to (2)$ works, but instead we take $z_X =\sqrt[4]{s_X/s_Y}$ and $z_Y=1/z_X$.

$(2)\to (3)$: By \cref{lem:AB}, $$\lambda_2(P)=\lambda_2(D^{-1}A)=\lambda_2(D^{-1/2}AD^{-1/2}) \leq \lambda_1(\bar{S}) = \sqrt{s_Xs_Y}.$$ 

$(1)\to (4)$: By \cref{lem:AB}, 
$$ \lambda_2(P-S)=\lambda_2(D^{-1}A-S) = \lambda_2(D^{-1/2}AD^{-1/2}-S)\leq 0.$$

$(3)\to (2)$: First, $\lambda_2(P- \bar{S})=\lambda_2(D^{-1/2}AD^{-1/2} - \bar{S}) \leq 0$, so $P-\bar{S}$ has at most one positive eigenvalue. Thus $D^{-1/2}AD^{-1/2} - \bar{S} \preceq ww^T$ where $w$ is the first eigenvector of $D^{-1/2}AD^{-1/2} - \bar{S}$. $(4)\to (1)$ can be shown similarly. 
\end{proof}

\begin{lemma}\label{lem:dpartiteeigenvalues}
  Let $G=(V,E,w)$ be a weighted $d$-partite graph with parts $T_1,\dots,T_d$ such that for $i,j\in [d]$ with $i\neq j$, and any $x\in T_i$, we have $w(x,T_j)=\sum_{y\in T_j} w(x,y)=\frac{d_w(x)}{d-1}$. Here, $d_w(x)$ is the weighted degree of $x$. Let $P$ be the simple random walk on $G$ with stationary distribution $\mu$.
  Then, $\mu(T_i)=1/d$ for all $i$.
   For every $i,j\in [d]$, let $G_{i,j}$ be the induced bipartite graph on parts $T_i,T_j$ with corresponding random walk matrix $P_{i,j}$. Suppose $P_{i,j} -M_{i,j}$ has at most one positive eigenvalue, where $M_{i, j}$ is the diagonal matrix with entries given by $m_i(j)$ for vertices in $T_i$ and $m_j(i)$ for vertices in $T_j$ such that $0\leq m_i(j)< 1$ for all $i\neq j$. Then,
     $$ \lambda_2(P) \leq 1- \min_{i\in [d]}\E_{j\in [d], j\neq i} 1-m_i(j) = \max_{i\in [d]} \E_{j\in [d],j\neq i} m_i(j)$$
 \end{lemma}
 \begin{proof}
 Recall that $P$ is self-adjoint with respect to the inner product $\langle f,g\rangle_\mu=\E_{x\sim \mu} f(x)g(x)$. We then have $\bone$ is an eigenfunction with eigenvalue 1. For  $i\in [d]$, define 
 $$e_i(x) =\begin{cases} d-1 & \text{if } x\in T_i\\ -1& \text{otherwise.}\end{cases}$$ 
 Then since for any $i\neq j$ and $x\in T_i$, $P(x,T_j)=\frac{1}{d-1}$ we get that  $e_i$ is also an eigenfunction of $P$ with eigenvalue $\frac{-1}{d-1}$, i.e, $Pe_i=\frac{-e_i}{d-1}$. We call $\bone$ and $e_i$'s trivial eigenfunctions of $P$. 

 Let $f$ be a non-trivial eigenfunction of $P$; for any $i\in [d]$  we have 
 $$0=\langle f,\bone+e_i\rangle_\mu= d\cdot \langle f,\bone_i\rangle_\mu.$$ 
 where $\bone_i$ is the indicator vector of the set $T_i$.

Note that if $\lambda_2(P)\leq 0$ we are already done since $m_i(j)\geq 0$ for all $i,j$. So we assume $\lambda_2(P)\geq 0$.
It follows by \cref{lem:rayleigh} that the second largest eigenvalue of  $P$ is%is either $\frac{-1}{d-1}$ or 
    $$ 1-\min_{f:\langle f,\bone_i\rangle_\mu =0,\forall i} \frac{\cE(f,f)}{\E_x f(x)^2}. $$
    %where $\cE_P(f,f)=\frac12 \E_{x\sim \pi} \sum_y P(x,y) (f(x)-f(y))^2$
    
     Let $g$ be the second eigenfunction of $P$. Consider the bipartite graph $G_{i,j}$ for some $i\neq j$. Let $\mu_{G_{i,j}}$ be the stationary distribution of the random in the graph $G_{i,j}$. Notice that since $d_{G_{i,j}}(x) = \frac{1}{d-1}d_w(x)$ for all $x\in T_i\cup T_j$ we have $\mu_{G_{i,j}}(x) = \frac{d}{2}\mu(x)$ for any $x\in T_i\cup T_j$. Therefore, 
     \begin{equation}\label{eq:g1i=g1j}
         \langle g,\bone_i\rangle_{\mu_{G_{i,j}}} =\langle g,\bone_j\rangle_{\mu_{G_{i,j}}}=0.
     \end{equation}

%Let $u=\frac{\alpha}{m_i(j)}\bone_i + \frac{\beta}{m_j(i)}\bone_j$ for some $\alpha,\beta>0$ that we choose later.
Let %$D_{i, j} = \mathrm{diag}(\mu_{G_{i, j}})$ and 
$v = (P_{i, j} - M_{i, j})\bone$.
    Using the above observation for the 2-partite graph $G_{i,j}$
    \begin{align}\label{eq:gPijquadratic}
        \langle g, v\rangle_{\mu_{G_{i,j}}}& = \langle g, \bone\rangle_{\mu_{G_{i, j}}} -\langle g,  M_{i, j}\bone\rangle_{\mu_{G_{i, j}}} \underset{\eqref{eq:g1i=g1j}}{=}  -\langle g,  M_{i, j}\bone\rangle_{\mu_{G_{i, j}}} \nonumber \\&= -\langle g,  \bone\rangle_{\mu_{G_{i, j}}}- \langle g,  \bone\rangle_{\mu_{G_{i, j}}} \underset{\eqref{eq:g1i=g1j}}{=} 0.
    \end{align}

    %On the other hand, by 
    Recall that $P_{i,j}$ is self-adjoint w.r.t. $\langle.,.\rangle_{\mu_{G_{i,j}}}$ and so is $M_{i, j}$ as it is diagonal. Furthermore, observe that since $P_{i,j}$ is stochastic,
    $$\langle \bone, (P_{i,j}-M_{i,j})\bone\rangle_{\mu_{G_{i,j}}} = \frac12(1-m_i(j)  + \frac12(1-m_j(i))\underset{m_i(j),m_j(i)< 1}{\geq} 0.$$
    %\textcolor{red}{sh:I added an extra condition to the Lemma's statement.}
    Therefore, by \cref{lem:oneposeigval} we get
    \begin{equation}\label{eq:PMvPSD}
        \langle g,(P_{i, j} - M_{i, j})g\rangle_{\mu_{G_{i,j}}} \leq \frac{\langle g,v\rangle^2_{\mu_{G_{i,j}}}}{\langle \bone,v\rangle_{\mu_{G_{i,j}}}}\underset{\eqref{eq:gPijquadratic}}{=} 0. 
    \end{equation}
%\textcolor{red}{Sh: I think adding the $D_{i,j}$ to these calculations is pointless. Just makes the calculations un-necessarily complicated. Can you re-write the statement/proof of Lemma 2.7 for a matrix which is self-adjoint w.r.t. to an inner-product?}
    As $\cE_{P_{i, j}}(g, g)$ is the same as the quadratic form of $I - P_{i, j}$ on $g$ under the $\mu_{G_{i,j}}$ inner product, we write
\begin{align}\begin{split}\label{eq:lambdai,jint}
         \cE_{P_{i, j}}(g, g) &= \langle g, (I - P_{i, j})g\rangle_{\mu_{G_{i, j}}} 
         \\&\underset{\eqref{eq:PMvPSD}}{\geq} \langle g, \left(I - M_{i, j} \right)g\rangle_{\mu_{G_{i, j}}}%-\frac{\langle g,v\rangle^2_{\mu_{G_{i,j}}}}{\langle \bone,v\rangle_{\mu_{G_{i,j}}}}
         %\\&= \langle g, (I - M_{i, j})g\rangle_{\mu_{G_{i, j}}} - \frac{\langle g, v\rangle^2}{\langle\bone, v\rangle_{\mu_{G_{i,j}}}}
         \\&=%\underset{\eqref{eq:gPijquadratic}}{=} 
         (1-m_i(j))\sum_{x \in T_i} \mu_{G_{i,j}}(x)g(x)^2 + (1-m_{j}(i))\sum_{x \in T_j}\mu_{G_{i,j}}(x)g(x)^2
     \end{split}
     \end{align}
%     \textcolor{red}{sh:I don't understand the above inequality. What happenned to $P$?}
     Using the fact that for $x\in T_i,y\in T_j$, $\mu(x)=\frac{2}{d}\mu_{G_{i,j}}(x)$ and that $(d-1)P(x,y)=P_{i,j}(x,y)$ it follows that,
     \begin{align*}
         \cE_P(g,g) &=  \sum_{i<j\in [d]} \frac{2}{d}\cdot \frac1{d-1} \cE_{P_{i,j}}(g,g)\\
         &\underset{\eqref{eq:lambdai,jint}}{\geq} \sum_{i<j\in [d]}  \frac{2}{d}\cdot \frac{1}{d-1}\cdot \left((1-m_i(j))\sum_{x\in T_i} \mu_{G_{i,j}}(x)g(x)^2+(1-m_j(i))\sum_{x\in T_j} \mu_{G_{i,j}}(x)g(x)^2\right)\\
         &=\sum_{i<j\in [d]}  \frac{1}{d-1}\cdot \left((1-m_i(j))\sum_{x\in T_i} \mu(x)g(x)^2+(1-m_j(i))\sum_{x\in T_j} \mu(x)g(x)^2\right)\\
         &=\sum_i \left(\sum_{x\in T_i}\mu(x)g(x)^2\right)\cdot \left(\sum_{j\in[d], j\neq i} \frac{1-m_i(j)}{d-1}\right)\\
         &\geq \min_i\left\{\frac1{d-1}\sum_{j\neq i} 1-m_i(j)\right\}\cdot \sum_x \mu(x) g(x)^2 
     \end{align*}
     Therefore, 
     $$\lambda_2(P)=1-\frac{\cE_P(g,g)}{\E_{x\sim\mu} g(x)^2}\leq 1-\min_{i\in [d]} \sum_{j\neq i} \frac{1-m_i(j)}{d-1}$$
     as desired.
 \end{proof}

\subsection{Basic Properties of (Homogeneous) Polynomials}
A polynomial $p\in \R[t_1,\dots,t_n]$ is $d$-homogeneous, if for any $\lambda\in \R$,
$$ p(\lambda t_1,\dots,\lambda t_n)=\lambda^d p(t_1,\dots,t_n).$$

\begin{definition}[Irreducible Matrices]A matrix $A \in \R^{n\times n}$  is called {\bf irreducible} if for all distinct $i,j\in [n]$ there is a sequence $i = i_0,i_1,i_2,\dots,i_\ell = j$ such that $i_{k-1} \neq  i_k$ for all $1 \leq k \leq \ell$, and $A(i_0,i_1)\cdot A(i_1,i_2)\dots A(i_{\ell-1},i_\ell) \neq 0$.
\end{definition}

For a vector $v\in \R^n$, we write
$$ \nabla_{\bm{v}} p(\bm{t}) :=\sum_{i=1}^n v_i \cdot \partial_{t_i}p(\bm{t}).$$

\begin{lemma}[Euler's formula] \label{lem:Eulers-formula}
    If $p \in \R[t_1,\ldots,t_n]$ is a $d$-homogeneous polynomial, then we have that 
    $$d \cdot p(\bm{t}) = \sum_{i=1}^n t_i \cdot \partial_{t_i} p(\bm{t}).$$
\end{lemma}

\begin{lemma}[Chain rule] \label{lem:chain-rule}
    If $f \in \R[t_1,\ldots,t_n]$ is a polynomial, $L: \R^m \to \R^n$ is a linear map, and $\bm{v} \in \R^m$ then
    \[
        \nabla_{\bm{v}}[f(L(\bm{t}))] = [\nabla_{L(\bm{v})}f](L(\bm{t})).
    \]
\end{lemma}
\begin{proof}
    Considering $L$ as an $n \times m$ matrix, let $\ell_{ij}$ be the entries of $L$. By the chain rule for multivariate functions, we have
    \[
        \partial_{t_j}[f(L(\bm{t}))] = \partial_{t_j}[f(L(\bm{t})_1, L(\bm{t})_2, \ldots, L(\bm{t})_n)] = \sum_{i=1}^n \ell_{ij} [\partial_{t_i} f](L(\bm{t})) = [\nabla_{L(\bm{e}_j)} f](L(\bm{t})).
    \]
    Therefore,
    \[
        \nabla_{\bm{v}}[f(L(\bm{t}))] = \sum_{j=1}^m v_j \partial_{t_j}[f(L(\bm{t}))] = \sum_{j=1}^m v_j [\nabla_{L(\bm{e}_j)} f](L(\bm{t})) = [\nabla_{L(\bm{v})} f](L(\bm{t}))
    \]
    as desired.
\end{proof}

\begin{lemma}[Lem.~3.2 of \cite{BL21}] \label{lem:Euler-expression}
    Suppose $p \in \R[t_1,\ldots,t_n]$ is a $d$-homogeneous polynomial and
    \[
        d \cdot p(\bm{t}) = \sum_{i=1}^n t_i \cdot Q_i(\bm{t})
    \]
    for some $(d-1)$-homogeneous polynomials $Q_1,\ldots,Q_n$. If $\partial_{t_i} Q_j(\bm{t}) = \partial_{t_j} Q_i(\bm{t})$ for all $i,j$, then $Q_i = \partial_{t_i} p(\bm{t})$ for all $i$.
\end{lemma}

\subsection{Log-concave Polynomials}

\begin{definition}[$\cC$-Lorentzian polynomials; Defn. 2.1 of \cite{BL23}] \label{def:C-Lorentzian}
    Let $p \in \R[t_1,\ldots,t_n]$ be a $d$-homogeneous polynomial for $d \geq 2$, and let $\cC \subset \R^n$ be a non-empty open convex cone. We say that $p$ is \textbf{$\cC$-Lorentzian} if
    \begin{enumerate}
        \item[(P)] $\nabla_{\bm{v}_1} \cdots \nabla_{\bm{v}_d} p(\bm{t}) > 0$ for all $\bm{v}_1,\ldots,\bm{v}_d \in \cC$, and
        \item[(Q)] the Hessian of $\nabla_{\bm{v}_1} \cdots \nabla_{\bm{v}_{d-2}} p(\bm{t})$ has at most one positive eigenvalue for all $\bm{v}_1,\ldots,\bm{v}_{d-2} \in \cC$.
    \end{enumerate}
    Note that (P) implies any Hessian in (Q) also has at least one positive eigenvalue.
    
    Additionally, any linear form for which (P) holds is defined to be $\cC$-Lorentzian of degree $d=1$, and any positive constant is defined to be $\cC$-Lorentzian of degree $d=0$. %Finally, we consider the zero polynomial to be $\cC$-Lorentzian of every degree $d$.
\end{definition}

% \begin{lemma} \label{lem:C-Lor-closed}
%     For a given open convex cone $\cC \subset \R^n$, the set of all $\cC$-Lorentzian polynomials is closed in the space of all non-zero $d$-homogeneous polynomials in $\R[t_1,\ldots,t_n]$.
% \end{lemma}

% \begin{lemma} \label{lem:closure-deriv-C-Lor}
%     If $p$ is $\cC$-Lorentzian and $\bm{v} \in \overline{\cC}$, i.e., $\bm{v}$ belongs to the closure of $\cC$, then either $\nabla_{\bm{v}} p$ is $\cC$-Lorentzian or else is the zero polynomial.
% \end{lemma}
% \begin{proof}
%     Follows from the definition of $\cC$-Lorentzian and \cref{lem:C-Lor-closed}.
% \end{proof}

%Recall from \cite{BL23} the following. 
The following theorem is the main tool that we use to prove that the polynomials that we construct by our $\bm{\pi}$-maps are $\cC$-Lorentzian. 
\begin{theorem}[{\cite[Thm. 2.9]{BL23}}] \label{thm:Lor-from-derivs}
    Let $p \in \R[t_1,\ldots,t_n]$ be a $d$-homogeneous polynomial for $d \geq 3$, and let $\cC$ be a non-empty open convex cone in $\R_{>0}^n$. If
    \begin{enumerate}
        \item $\nabla_{\bm{v}_1} \cdots \nabla_{\bm{v}_d} p > 0$ for all $\bm{v}_1,\ldots,\bm{v}_d \in \cC$, and
        \item the Hessian of $\nabla_{\bm{v}_1} \cdots \nabla_{\bm{v}_{d-2}} p$ is irreducible and its off-diagonal entries are non-negative for all $\bm{v}_1,\ldots,\bm{v}_{d-2} \in \cC$, and
        \item $\partial_{t_i} p$ is $\cC$-Lorentzian for all $i$,
    \end{enumerate}
    then $p$ is $\cC$-Lorentzian.
\end{theorem}
We remark that the original statement is for ``effective'' cones as defined in \cite{BL23}, but note that any cone contained in $\R_{>0}^n$ is automatically effective. We also note that if $p$ in the previous theorem does not depend on some variable $t_i$, then we can apply the theorem to $p$ restricted to all variables except $t_i$.

% \begin{proof}
%     Fix $\bm{v}_1, \ldots, \bm{v}_{d-3} \in \cK$, and consider the cubic $g = \nabla_{\bm{v}_1} \cdots \nabla_{\bm{v}_{d-3}} f$. Since $\partial_{t_i} f$ is $\cK$-Lorentzian, it follows that $\partial_{t_i} g$ is also $\cK$-Lorentzian. By choosing $\bm{v}_{d-2} = \bm{x} \in \cK$, it follows from $(2)$ that the Hessian $\nabla^2 g(\bm{x})$ is irreducible and its off-diagonal entries are non-negative. Further, we have that $\partial_{t_i} g(\bm{x}) > 0$ for all $i$ by definition of $\cK$-Lorentzian. Therefore by Lemma 2.8 of Lorentzian on cones, the Hessian $\nabla^2 g(\bm{x})$ has exactly one positive eigenvalue. Since $\nabla^2 g(\bm{x}) = \nabla^2\left[\nabla_{\bm{v}_1} \cdots \nabla_{\bm{v}_{d-3}} \nabla_{\bm{x}} f\right]$, this implies $f$ is $\cK$-Lorentzian.
% \end{proof}

Once we prove our polynomial is $\cC$-Lorentzian, we will use the following theorem part (1) to prove that the path complex $(X,\mu)$ is a local spectral expander, and we will use part (2) to get  our log-concavity statements (see proof of \cref{thm:log-concavity-colored-complex}).

\begin{lemma} \label{lem:main-purpose-lemma}
    Let $p$ be a $d$-homogeneous polynomial on $\R^n$, and let $\cC \subset \R^n$ be a non-empty open convex cone. If $p$ is $\cC$-Lorentzian and $\bm{u},\bm{w},\bm{v}_1,\ldots,\bm{v}_{d-2} \in \overline{\cC}$ (the closure of $\cC$), then:
    \begin{enumerate}
        \item the Hessian of $\nabla_{\bm{v}_1} \cdots \nabla_{\bm{v}_{d-2}} p$ has at most one positive eigenvalue, and
        \item the coefficients of $f(x) = p(\bm{u} \cdot x + \bm{w})$ form an ultra log-concave sequence.
    \end{enumerate}
\end{lemma}
\begin{proof}
    For (1), let $(\bm{v}_1(i),\ldots,\bm{v}_{d-2}(i))_{i=1}^\infty$ be a sequence of vectors in $\cC$ which approach $(\bm{v}_1,\ldots,\bm{v}_{d-2})$. Thus the Hessian of $\nabla_{\bm{v}_1(i)} \cdots \nabla_{\bm{v}_{d-2}(i)} p$ has at most one positive eigenvalue for all $i \geq 1$. Therefore the limit of these Hessian matrices also has at most one positive eigenvalue.

    The statement of (2) follows from Remark 2.4 of \cite{BL23}.
\end{proof}

\def\cM{{\cal M}}
\subsection{Markov Chains and Mixing 
Time}
\label{sec:markovmixing}
%Given a $d$-partite path complex $(X,\mu)$ the down-up Markov chain is defined as follows: Given a facet $\sigma=(F_1,\dots,F_d)\in X(d)$, where $F_i\in T_i$ choose $i\in [d]$ uniformly at random, and for any $F\in T_i$ such that $F_{i-1} <F < F_{i+1}$ move to $\sigma'=(F_1,\dots,F_{i-1},F,F_{i+1},\dots,F_d)$ w.p. proprotional to $\mu(\sigma')$. We let $P^\vee_d$ be the transitional probability matrix of this chain.

%We remark that in the special case $(X,\mu)$ is the (uniform) distribution over maximal chains of a distributive lattice $\cL$, the down-up Markov chain is essentially the same as the Karzanov-Khachiyan's studied before \cite{KK91}.

%It is not hard to see (assuming connectivity of $X$) that $\mu$ is the unique stationary distribution of the chain, i.e., $\mu P^\vee_d=\mu$. So, to generate samples from $\mu$, we just need to study its mixing time.

Recall that for a Markov chain with transition probability matrix $P^\vee_d$ and stationary distribution $\mu$ 
the {\bf total variation mixing time} started at a state $\tau$ is defined as follows:
	\[ t_{\tau}(\epsilon)=\min\left\{t\in \Z_{\geqslant 0}: \norm{{P^\vee_d}^t(\tau,\cdot)-\mu}_1 \leqslant \epsilon\right\},\]
 	%where $P^t(\tau,\cdot)$ is the distribution of the chain started at $\tau$ at time $t$.

For a function $f:X(d)\to\R$ let
\begin{equation}\label{eq:var} \Var_\mu(f) :=\E_{\sigma\sim \mu}(f(\sigma)-\E_\mu f)^2 = \langle f,f\rangle_\mu-\langle f,\b1\rangle_\mu^2. 
\end{equation}
be the variance of $f$ w.r.t. $\mu$.

The Poincar\'e constant/spectral gap of $P^\vee_d$ is defined as,
$$ \lambda(P^\vee_d):=\inf_{f:X(d)\to\R} \frac{\cE_{P^\vee_d}(f,f)}{\Var_\mu(f)}.$$
The following is classical:
\begin{theorem}
    For any $\tau\in X(d)$,
    $$ t_\tau(\eps) \leq \frac{\log(\eps^{-1}\cdot \frac{1}{\mu(\tau)})}{\lambda(P^\vee_d)}.$$
\end{theorem}

The following theorem is re-proved multiple times (see e.g., \cite{CLV21}). Here, we include the proof for completeness.
\begin{theorem}\label{thm:eigPemptytoVar}
    Suppose that for a $d$-dimensional path complex $(X,\mu)$ we have $\lambda_2(P_\varnothing)\leq \alpha$ for some $\alpha\geq 0$. Then, for any function $f:X(d)\to\R$, we have
        \begin{equation}\label{eq:totalvargoal1} \Var_\mu(f) \leq \frac{d}{d-1}\cdot \frac{1}{1-\alpha} \cdot \E_{i\sim [d]} \E_{F_i\sim\mu^i} \Var_{\mu_{F_i}}(f) \end{equation}
\end{theorem}
%\textcolor{red}{Sh: We should move this proof to the appendix}
\begin{proof}
   For a function $f:X(d)\to\R$, let $f^1:[d]\times X(1)\to\R$ defined as
$$f^1(i,F) = \E_{\sigma\sim\mu}[f(\sigma) | \sigma_i=F].$$
We also let $\mu^1(i,F):=\frac1{d}\P[\sigma_i=F].$ Note that $\mu^1$ defines a probability distribution as $\sum_{i,F} \mu^1(i,F)=1$.
 \begin{fact}[Law of Total Variance]
    $\Var_\mu(f) = \Var_{\mu^1}(f^1) + \E_{i\sim[d]}\E_{F_i\sim\mu^i} \Var_{\mu_{F_i}}(f) $
 \end{fact}
 Having this, to prove the theorem it is enough to show that 
 \begin{equation}\label{eq:totalvargoal2}
\forall f:X(d)\to \R:\quad     \frac{(d-1)\alpha+1}{d}\cdot \Var_\mu(f) \geq \Var_{\mu^1}(f^1).
 \end{equation}
 In particular, this (together with the law of total variance)  implies that 
 $$\frac{(d-1)\alpha+1}{(d-1)(1-\alpha)}\cdot \E_{i\sim[d]}\E_{F_i\sim\mu^i} \Var_{\mu_{F_i}}(f) \geq \Var_{\mu^1}(f^1) $$
Adding $\E_{i\sim[d]}\E_{F_i\sim\mu^i}\Var_{\mu_{F_i}}(f)$ to both sides proves \eqref{eq:totalvargoal1}.

So, in the rest of the proof we show \eqref{eq:totalvargoal2}.

For any $i\in [d]$ and $F\in F_i$, define a function $\mu_{i,F}: X(d) \to\R$ where 
$$\mu_{i,F}(\sigma)=\begin{cases} \frac{\mu(\sigma)}{\P_\tau[\tau_i=F]} &\text{if } \sigma_i=F\\ 0 &\text{otherwise}\end{cases}.$$ 
Similarly, let $\b1_{i,F} = \I[\sigma_i=F]$.
Then, observe that
%$$ f^1(i,s)=\langle \pi_{i,s},f\rangle $$
%Let $P$ be an operator that $Pf= \sum_{i,s} \frac{1_{i,s}}{n\P{\sigma_i=s}}\langle 1_{i,s},f\rangle_{\pi}$
\begin{align*} \Var_{\mu^1}(f^1) &= \sum_{i,F} \mu^1(i,F) f^1(i,F)^2 - \langle f^1,\b1\rangle_{\mu^1}^2\\
&=	f^\top\left(\sum_{i,s} \mu^1(i,s) \mu_{i,s}\mu_{i,s}^\top \right)f - \langle f,\b1\rangle^2_{\mu}=:\langle Pf,f\rangle_\mu - \langle f,\bone\rangle_\mu^2,
\end{align*}
where in the second identity  we used that $$\langle f^1,\b1\rangle_{\mu^1}=\sum_{i,F}\frac1d \P[\sigma_i=F]\cdot \E_{\sigma\sim\mu}[f(\sigma)|\sigma_i=F]=\E_{\sigma\sim\mu]}[f(\sigma)]=\langle f,\b1\rangle_\mu.$$ 
Furthermore, the matrix  $P$ on the RHS is defined as $P:=\sum_{i,F} \frac1d \b1_{i,F}\mu_{i,F}^\top$. Note that $P$ is a stochastic matrix. 
Therefore, by \eqref{eq:var} we have
\begin{align*}
\max_f \frac{\Var_{\mu^1}(f^1)}{\Var(f)} = \max_f \frac{\langle Pf,f\rangle_\mu - \langle f,\b1\rangle_\mu^2}{\langle f,f\rangle_\mu - \langle f,\b1\rangle_\mu^2}	=\lambda_2(P)
\end{align*}
The second identity uses the Rayleigh quotient and that $P$ is stochastic and self-adjoint w.r.t. $\mu(.)$. (so its first eigenvalue is 1).
So, to prove \eqref{eq:totalvargoal2} it is enough to show that $\lambda_2(P)\leq \frac{(d-1)\alpha+1}{d}$.
The observation is that $P$ is a low-rank matrix. In particular, let $U$ be the matrix with columns $\frac1d \b1_{i,F}$ and $R$ be the matrix with rows $\mu_{i,F}$ (for $i\in [d]$). Then, $P=UR$. But, by \cref{lem:AB}, $\lambda_2(P)=\lambda_2(UR)=\max\{0,\lambda_2(RU)\}$.
First, notice $RU\in \R^{X(1)\times X(1)}$. In particular, for $i,j\in [d]$ and $F\in T_i,G\in T_j$, 
\begin{align*} RU(F,G) &= \frac1d \sum_\sigma \frac{\mu(\sigma)}{\P_{\sigma\sim\mu}[\sigma_i=F]} \I[\sigma_i=F,\sigma_j=G] \\
&= \frac{\P[\sigma_i=F,\sigma_j=G]}{d\cdot \P[\sigma_i=F]} = \frac{1}{d}\cdot \P_{\sigma\sim\mu}[\sigma_j=G | \sigma_i=F]
\end{align*}
But this is exactly the (off-diagonal) entries of the simple random walk matrix on the link of the empty set (up to a normalization). Note that $RU(F,F)=1/d$ for any $F\in X(1)$, and thus $RU=\frac{d-1}{d} P_\varnothing + \frac{1}{d} I$. Therefore, $$\lambda_2(P)=\lambda_2(UR)\leq \max\{0,\lambda_2(RU)\}\leq \max\left\{0,\frac{d-1}{d}\lambda_2(P_\varnothing)+\frac{1}{d}\right\}\leq \frac{(d-1)\alpha+1}{d}$$
as desired.
\end{proof}

We end this section by the following factorization of variance lemma for independent distribution (See \cite[Lemma 4.1]{Cap22}  for a proof).
\begin{lemma}[Factorization of Variance]\label{lem:varfactorization}Let $\nu_1, \nu_2$ be probability measures on (finite) state spaces $A_1,A_2$ respectively. Let
     $\mu=\nu_1 \times \nu_2$ be the  product distribution on $A_1 \otimes A_2$ defined as $\mu(a_1,a_2)=\nu_1(a_1)\cdot \nu_2(a_2)$ for $a_1\in A_1,a_2\in A_2$.
     For a function $f:A_1\times A_2$ and $a_1\in A_1$, let $f^{a_1}(.)=f(a_1,.)$ and $f^{a_2}=f(.,a_2)$. Then, for any function $f:A_1 \otimes A_2\to\R$ we have
    $$ \Var_\mu(f) \leq \E_{a_1\sim \nu_1}\Var(f^{a_1}) +  \E_{a_2\sim \nu_2}\Var(f^{a_2}).$$
\end{lemma}
%\href{http://www.mat.uniroma3.it/users/caputo/entropy.pdf}this Lemma 4.1 for a proof.

\section{Local-to-Global Theorems for Path Partite Complexes}\label{sec:localtoglobalpath}
In this section we prove the following theorem.
\begin{theorem}\label{thm:localtoglobal}
    Given a $d$-dimensional $d$-partite path complex $(X,\mu)$ that is an $\alpha$-local spectral expander for some $\alpha\leq 1/2$. %Suppose that for $X$ (and any of partite path links of $X$ \textcolor{blue}{(JL: do you mean the links?)}) \textcolor{red}{Sh:Not just links, the links which form a path complex} we have the following: For any function $f:X(d)\to \R$,
    Then, the spectral gap of the down-up walk is at least $\frac{4}{d(d+1)^2}$. % where $\gamma \geq 3$ such that $\gamma\geq \frac{2}{1-\alpha}(1-2^{-\gamma})$.
\end{theorem}
Before proving this theorem we use it to prove \cref{thm:mainalg}.
\begin{proof}[Proof of \cref{thm:mainalg}]
If the path complex $(X,\mu)$ is a $1/2$-top-link expander then by \cref{thm:main}, $X$ is a $1/2$-local spectral expander. Therefore, by the above theorem the Poincar\'e constant of the down-up walk is at least $\frac{4}{d(d+1)^2}$%\geq \frac{1}{d^{4.447}}.$ 
 as desired.
\end{proof}
%{\color{Sepia}
 Given a function $f:X(d)\to \R$, we prove the following inequality
    \begin{equation}\label{eq:localtoglobal1} \Var_\mu(f) \leq \sum_{i\in[d]} i\cdot (d+1-i)\E_{F_1,\dots,F_{i-1},F_{i+1},\dots,F_d\sim \mu^{-i}} \Var_{\mu_{F_1,\dots,F_{i-1},F_{i+1},\dots,F_d}}(f)
    \end{equation}
    %where $\gamma>1$ is a parameter that we choose later. 
   % where $\beta=\frac{1}{1-\alpha}$.
Here, we write $\mu^{-i}$ as a shorthand for $\mu^{1,2,\dots,i-1,i+1,\dots,d}$. On the other hand, it is not hard to see that indeed, 
$$ \E_{i\sim d}\E_{F_1,\dots,F_{i-1},F_{i+1},\dots,F_d\sim \mu^{-i}} \Var_{\mu_{F_1,\dots,F_{i-1},F_{i+1},\dots,F_d}}(f) = \cE_{P^\vee_d}(f,f).$$
Therefore, \eqref{eq:localtoglobal1}  implies 
$$ \Var_\mu(f) \leq d\cdot \max\{i\cdot (d+1-i): i\in [d]\}\cdot  \cE_{P^\vee_d}(f,f)\Rightarrow \lambda(P^\vee_d)\geq \frac{4}{d(d+1)^2}$$

In the rest of the section we prove 
 \eqref{eq:localtoglobal1}. We prove this by induction.
 Suppose the claim holds for any $\ell$-partite path complex for $\ell<d$.
 The base case is when $d=1$ which holds trivially.
 
 Given an arbitrary function $f:X(d)\to\R$.
First, by \cref{thm:eigPemptytoVar}, we have
\begin{equation}\label{eq:useofvar1} \Var_\mu(f) \leq \beta \cdot \E_{i\sim [d]} \E_{F_i\sim\mu^i} \Var_{\mu_{F_i}}(f) 
\end{equation}
where $\beta=\frac{d}{d-1}\cdot \frac{1}{1-\alpha}$.
Now, by the conditional independence property, $\mu_{F_i}$ is a product distribution. Namely, let $\nu_1=\mu_{F_i}^{1,\dots,{i-1}}$ and $\nu_2=\mu_{F_i}^{i+1,\dots,d}.$ Then, $\mu=\nu_1\otimes\nu_2$. Therefore, by \cref{lem:varfactorization},  we have
$$ \Var_{\mu_{F_i}}(f) \leq \E_{F_{i+1},\dots,F_d\sim \mu_{F_i}^{i+1,\dots,d}}\Var_{\mu_{F_{i+1},\dots,F_d}}(f) + \E_{F_{1},\dots,F_{i-1}\sim \mu_{F_i}^{1,\dots,i-1}}\Var_{\mu_{F_{1},\dots,F_{i-1}}}(f)$$
Each of terms on the RHS is a partite path complex. So, by IH we have,
\begin{align*}
  \Var_{\mu_{F_i}}(f)  \leq&  \E_{F_{i+1},\dots,F_d\sim \mu_{F_i}^{i+1,\dots,d}} \sum_{j\in [i-1]} j\cdot (i-j) \E_{F_1,\dots,F_{i-1}\sim \mu^{-j}_{F_i,\dots,F_d}} \Var_{\mu_{F_1,\dots,F_{j-1},F_{j+1},\dots,F_d}}(f)  +\\
  & \E_{F_{1},\dots,F_{i-1}\sim \mu_{F_i}^{1,\dots,i-1}} \sum_{j\in i+[d-i]} (j-i)\cdot (d+1-j)\E_{F_{i+1},\dots,F_{d}\sim \mu^{-j}_{F_1,\dots,F_i}} \Var_{\mu_{F_1,\dots,F_{j-1},F_{j+1},\dots,F_d}}(f)
\end{align*}
Summing this up over all $i$ and using \cref{eq:useofvar1} we have
\begin{align*}
    \Var_\mu(f) \leq &\beta \cdot\E_{i\sim [d]} \E_{F_i\sim \mu^i} \Var_{\mu_{F_i}}(f)\\
    &\leq \beta \cdot \E_{i\sim [d]} \Big( \sum_{j\in [i-1]}j\cdot (i-j) \E_{F_1,\dots,F_{j-1},F_{j+1},\dots,F_d\sim \mu^{-j}} \Var_{\mu_{F_1,\dots,F_{j-1},F_{j+1},\dots,F_d}}(f)\\
    &+\E_{j\in i+[d-i]} (j-i)\cdot (d+1-j)\E_{F_1,\dots,F_{j-1},F_{j+1},\dots,F_d} \Var_{\mu_{F_1,\dots,F_{j-1},F_{j+1},\dots,F_d}}(f)\Big)%\\
    %&=\beta \cdot \E_{i\sim [d]} \Big((i-1)^{\gamma-1} \sum_{j\in [i-1]} \E_{F_1,\dots,F_{j-1},F_{j+1},\dots,F_d\sim \mu^{-j}} \Var_{\mu_{F_1,\dots,F_{j-1},F_{j+1},\dots,F_d}}(f)%\\
    %&+(d-i)^{\gamma-1} \sum_{j\in i+[d-i]} \E_{F_1,\dots,F_{j-1},F_{j+1},\dots,F_d} \Var_{\mu_{F_1,\dots,F_{j-1},F_{j+1},\dots,F_d}}(f)\Big)
\end{align*}
Now, fix some arbitrary $j\in [d]$. The coefficient of $\E_{F_1,\dots,F_{j-1},F_{j+1},\dots,F_d\sim \mu^{-j}} \Var_{\mu_{F_1,\dots,F_{j-1},F_{j+1},\dots,F_d}}(f)$ is 
\begin{align*}
    &\beta \left(\sum_{i=1}^{j-1} \frac{(j-i)(d+1-j)}{d} + \sum_{i=j+1}^d \frac{j\cdot (i-j)}{d}\right) \\
    &= \frac{1}{(1-\alpha)(d-1)}\left(\frac{(d+1-j)j(j-1) }{2} + \frac{j(d-j)(d+1-j)}{2}\right)\\
    &\underset{\alpha=1/2}{=} j\cdot (d+1-j)
\end{align*}
as desired. This finishes the proof of
\autoref{thm:localtoglobal}.

\section{Conic Lorentzian Polynomials}
\label{sec:lorentzianpolynomials}
%\textcolor{red}{THe first two paragraphs are redundant}
%Let $X$ be a $d$-partite path complex as defined above, where the elements of $X(1)$ can be uniquely partitioned into sets $T_1 \cup \cdots \cup T_d$ such that for every facet $\tau \in X(d)$, we have $|\tau \cap T_i| = 1$ for all $1 \leq i \leq d$. Throughout this section, we fix a $d$-partite path complex $X$.

%Given a simplicial complex $X$ and $S \in X$, we will let $X_S$ denote the link of $S$ in $X$.
% Given $S \in X$ and $F \in X_S(1)$ and $\tau \in X_S$, we also define $S+F = S \cup \{F\}$ and $S+\tau = S \cup \tau$.

%Given a simplicial complex $X$, we will associate a family of polynomials which will help us to bound eigenvalues. Our polynomials are heavily inspired by and can be seen as a generalization of the polynomials constructed there.

%Before writing down our polynomials, we give a high-level intuition. 
Given a (path) complex $(X,\mu)$, the main object we want to study (for analysis of down-up Markov chains) is the multilinear polynomial 
$$ q(\{t_F\}_{F\in X(1)})=\sum_{\sigma\in X(d)}\mu(\sigma)\prod_{F\in \sigma} t_F.$$ 
This polynomial perfectly encodes all facets of $X$. It is shown in \cite{ALOV24} that this polynomial is log-concave if $X$ is the independence complex of the matroid, i.e., the facets correspond to bases. In other words, $\nabla^2\log q(x)\preceq 0$ at any $x\in \R^{X(1)}_{\geq 0}$. 
Equivalently, in such a case $X$ is a $0$-local spectral expander.

Unfortunately, it is also shown in \cite{BH20} that a homogeneous polynomial $q$ is log-concave in the positive orthant only if the support of $q$ corresponds to the bases of a matroid which is not necessarily the case for general complexes.

So, we follow in the footsteps of \cite{BL23} (see also \cite{BL21}), and we deviate from this ideal case in two ways i) We construct a polynomial $p$ such that the multi-linear part of $p$ corresponds to $q$; so in principal one can still use $p$ to argue about properties of $q$. ii) Similar to \cite{BL23} our polynomial $p$ is no longer log-concave w.r.t. all vectors in the positive orthant. Instead, $p$ is log-concave only w.r.t. a cone $\cC$. We prove log-concavity of $p$ (w.r.t. $\cC$) inductively. The main challenge is that the cone of log-concavity changes as we take partial derivatives of $p$, i.e., as we induct. Because of that, we cannot concretely describe all vectors in the cone $\cC$.

In \cite{BL23}, the polynomial $p$ is defined via a certain ``lineality space'' $V$ which is compatible with $(X,\mu)$. This lineality space then uniquely defines (a family of) {\em linear} maps $\pi$ which map the cone of $p$ to the cone of its partial derivatives and they use to inductively prove $p$ is $\cC$-Lorentzian. In other words, $V$ is used both to define $p$ and to prove that it is $\cC$-Lorentzian. But, the use of such a lineality space limits the class of polynomials that can be studied/derived using this method.

%In \cite{BL23}, in order to get a handle on how the cone of $p$ relates to the cone of its partial derivatives, they define a "lineality space", $V$, that is compatible with $(X,\mu)$; roughly speaking this means $p$ is invariant under shifting by vectors in $V$. Using this lineality space, they uniquely define a (family of) linear maps $\pi$ which map the cone of $p$ to the cone of its derivatives which they use to inductively prove $p$ is $\cC$-Lorentzian. 

Our main observation which was the starting point of this project is that in order to define the polynomial we just need to have (a family of) linear maps $\pi$. And, as long as these maps satisfy a certain  commutativity property (see \cref{def:pi-maps}), the same machinery will work. This enables us to develop a new, more general way to construct commutative $\pi$ maps (see \cref{def:alpha-beta} and \cref{prop:pi-maps-from-alpha-beta}),
%In other words, we
completely eliminating any need of a lineality space.
%This allows to
With this, we can
define a much larger class of $\cC$-Lorentzian polynomials
%in the rest of the paper
and use them for sampling tasks or to prove log-concavity statements.

Throughout this section, we fix a pure $d$-dimensional simplicial complex $X$ and a probability distribution $\mu$ on its facets $X(d)$ such that the support of $\mu$ is $X(d)$, i.e., $\mu(\sigma)>0$ for all facets $\sigma$.
In the rest of this section we define $\cC$-Lorentzian polynomials for commutative $\pi$ maps and prove a classification theorem: If $X$ is connected and all quadratic partial derivatives of $p$ are log-concave then $p$ is $\cC$-Lorentzian; see \cref{thm:conn+quad=Lor}.

% One subsection, for any pure simplicial complex (not necessarily partite):
% \begin{enumerate}
%     \item Define $\pi$ abstractly: some linear maps so that commutativity holds
%     \item Define the cones and comment ``easy'' properties of the cones
%     \item Define the polynomials, and prove/state the partial derivative expression lemma
%     \item Comment about positivity property of the cones and polynomials
%     \item Finally state connected + quadratics = Lorentzian
% \end{enumerate}

% Next subsection:
% \begin{enumerate}
%     \item Define functions $f$ and $g$ to be ``nice'' if they satisfy \cref{lem:rk-crk-identities}
%     \item Such a pair of functions allows us to define $\pi$ maps (and prove commutativity)
%     \item prove analog of \cref{lem:general-mixed-derivs-expression} and log-concavity
% \end{enumerate}

%\subsection{Trickledown Theorems for Complexes using Hereditary Lorentzian Polynomials}

%We now construct a family of polynomials associated to the links of $X$. To do this, we need a family of linear maps called the $\pi$ maps which will help us to define the polynomials inductively.

\begin{definition}[Commutative $\pi$ maps] \label{def:pi-maps}
    Let $\bm\pi = \{\pi_{\sigma+\tau}: \R^{X_\sigma(1)} \to \R^{X_{\sigma \cup \tau}(1)}\}_{\sigma \in X,\tau \in X_\sigma}$ be a family of linear maps having the following property:
    \[
        \text{$\pi_{\sigma+\tau}(\bm{t})_H = t_H - L\big((t_F)_{F \in \tau}\big)$, for all  $H \in X_{\sigma \cup \tau}(1)$, where $L = L_{\sigma,\tau,H}$ is some linear functional.}
    \]
    We will also use the shorthands $\pi_{\sigma+F} := \pi_{\sigma+\{F\}}$ and $\pi_\sigma := \pi_{\varnothing+\sigma}$.
    
    We say that $\bm\pi$ is \textbf{commutative} if for all $\sigma \in X$, $\tau \in X_\sigma$, $\omega \in X_{\sigma \cup \tau}$, we have:
    \[
        \text{$\pi_{\sigma+(\tau \cup \omega)}(\bm{t}) = \pi_{(\sigma \cup \tau) + \omega}(\pi_{\sigma+\tau}(\bm{t}))$.}
    \]
    % \begin{enumerate}
    %     % \item $\pi_{\sigma+\tau}(\bm{t})$ only depends on $t_F$ for $F \in \tau \cup X_{\sigma \cup \tau}(1)$,
    %     % \item $\pi_{\sigma+\tau}(\bm{t})$ is the identity map when restricted to $\R^{X_{\sigma \cup \tau}(1)} \subset \R^{X_\sigma(1)}$, and
    %     \item $\pi_{\sigma+\tau}(\bm{t})_H = t_H - L\big((t_F)_{F \in \tau}\big)$ for all $H \in X_{\sigma \cup \tau}(1)$, where $L = L_{\sigma,\tau,H}$ is some linear map, and
    %     \item $\pi_{\sigma+(\tau \cup \omega)}(\bm{t}) = \pi_{(\sigma \cup \tau) + \omega}(\pi_{\sigma+\tau}(\bm{t}))$.
    % \end{enumerate}
    
\end{definition}

For the rest of this section, we fix a commutative family of linear maps $\bm\pi = \{\pi_{\sigma+\tau}\}_{\sigma \in X,\tau \in X_\sigma}$.

\begin{definition}[Polynomials associated to $X$] \label{def:polys-X}
    We inductively define a family of polynomials $p_\sigma(\bm{t})$ on $\R^{X_\sigma(1)}$ of degree $d-|\sigma|$ associated to every $\sigma \in X$. If $\sigma$ is a facet of $X$ (so that $|\sigma| = d$) we define $p_\sigma(\bm{t}) = \mu(\sigma)$, and if $|\sigma| < d$ we define $p_\sigma(\bm{t})$ via
    \[
        (d-|\sigma|) \cdot p_\sigma(\bm{t}) = \sum_{F \in X_\sigma(1)} t_F \cdot p_{\sigma \cup \{F\}}(\pi_{\sigma+F}(\bm{t})).
    \]
    In particular, if $|\sigma| = d-1$ then $p_\sigma$ is the linear form given by
    \begin{equation}\label{eq:linearpcase}
        p_\sigma(\bm{t}) = \sum_{F \in X_\sigma(1)} t_F \cdot \mu(\sigma \cup \{F\}).
    \end{equation}
\end{definition}
\noindent See \cref{sec:coloredpathcomplexes} for examples of $\bm\pi$ maps and the associated polynomials.

We now prove that the definition of $p_\sigma(\bm{t})$ is precisely what one would obtain from Euler's formula for homogeneous polynomials (\cref{lem:Eulers-formula}) by showing that the partial derivatives of $p_\sigma(\bm{t})$ have a nice description.

\begin{lemma}[cf. Lem.~3.3 of \cite{BL21}] \label{lem:derivs-expression}
    Given $\sigma \in X$ and $F \in X_\sigma(1)$, we have
    \[
        \partial_{t_F} p_\sigma(\bm{t}) = p_{\sigma \cup \{F\}}(\pi_{\sigma+F}(\bm{t})).
    \]
    Further, for $F,G \in X_\sigma(1)$ with $F \neq G$, we have
    \[
        \partial_{t_G} \partial_{t_F} p_\sigma(\bm{t}) = \begin{cases}
            p_{\sigma \cup \{F,G\}}(\pi_{\sigma+\{F,G\}}(\bm{t})), & \{F,G\} \in X_\sigma(2) \\
            0, & \text{otherwise}.
        \end{cases}
    \]
\end{lemma}
\begin{proof}
    We prove the first statement by induction on $\codim(\sigma)$, noting that the result is clear for $\codim(\sigma) = 1 \iff |\sigma| = d-1$. By \cref{def:polys-X} and \cref{lem:Euler-expression}, we now just need to show that 
    $$\partial_{t_G} p_{\sigma \cup \{F\}}(\pi_{\sigma+F}(\bm{t})) = \partial_{t_F} p_{\sigma \cup \{G\}}(\pi_{\sigma+G}(\bm{t}))$$
     for all $F,G \in X_\sigma(1)$. 
     
    First, consider the case that $\{F,G\} \not\in X_\sigma(2)$. Then by the definition of the $\pi$ maps, the polynomial $p_{\sigma\cup \{F\}}(\pi_{\sigma+F}(\bm{t}))$ does not depend on $t_G$ (recall that the linear map $\pi_{\sigma+F}$ only depends on $\sigma, F$ and $H\in X_{\sigma\cup\{F\}}(1)$). Similarly, $ p_{\sigma \cup \{G\}}(\pi_{\sigma+G}(\bm{t}))$ does not depend on $t_F$. So, both sides of the above equation are 0 and we are done.
    
    %note that both expressions are $0$ if $\{F,G\} \not\in X_\sigma(2)$ since $\bm\pi$ hereditary implies $\pi_{\sigma+F}(\bm{t})$ does not depend on $t_G$ and vice versa. 
    Now, suppose $\{F,G\} \in X_\sigma(2)$. Then by induction %the chain rule (\cref{lem:chain-rule}), and the fact that $\bm\pi$ is commutative, we have that
    we have
    \begin{align*}
        \partial_{t_F} p_{\sigma \cup \{G\}}(\pi_{\sigma+G}(\bm{t})) &\underset {\text{\cref{lem:chain-rule}}}{=} \left[\nabla_{\pi_{\sigma+G}(\bm{1}_F)} p_{\sigma \cup \{G\}}\right](\pi_{\sigma+G}(\bm{t})) \\
        &\underset{\text{def. of } \bm\pi}{=}\left[\partial_{t_F} p_{\sigma \cup \{G\}}\right](\pi_{\sigma+G}(\bm{t})) \\
        &\underset{\text{induction}}{=} p_{\sigma \cup \{F,G\}}(\pi_{(\sigma \cup \{G\})+F}(\pi_{\sigma+G}(\bm{t}))) 
        \underset{\text{ commutativity of } \bm\pi}{=} p_{\sigma \cup \{F,G\}}(\pi_{\sigma+\{F,G\}}(\bm{t}))
    \end{align*}
    Similarly, $
        \partial_{t_G} p_{\sigma \cup \{F\}}(\pi_{\sigma+F}(\bm{t}))= p_{\sigma \cup \{F,G\}}(\pi_{\sigma+\{F,G\}}(\bm{t})),$
    and thus they are equal. %$\partial_{t_G} p_{\sigma \cup \{F\}}(\pi_{\sigma+F}(\bm{t})) = \partial_{t_F} p_{\sigma \cup \{G\}}(\pi_{\sigma+G}(\bm{t}))$.
    
    For the second statement, note that the first statement and the above arguments imply
    \[
        \partial_{t_F} \partial_{t_G} p_\sigma(\bm{t}) = \partial_{t_F} p_{\sigma \cup \{G\}}(\pi_{\sigma+G}(\bm{t})) = p_{\sigma \cup \{F,G\}}(\pi_{\sigma+\{F,G\}}(\bm{t}))
    \]
    whenever $\{F,G\} \in X_\sigma(2)$, and $\partial_{t_F} \partial_{t_G} p_\sigma(\bm{t}) = 0$ otherwise.
\end{proof}

% \begin{corollary} \label{cor:double-deriv-expression}
%     Given $\sigma \in X$ and $F,G \in X_\sigma(1)$, we have that
%     \[
%         \partial_{t_G} \partial_{t_F} p_\sigma(\bm{t}) = \begin{cases}
%             p_{\sigma \cup \{F,G\}}(\pi_{\sigma+\{F,G\}}(\bm{t})), & \{F,G\} \in X_\sigma(2) \\
%             0, & \text{otherwise}.
%         \end{cases}
%     \]
% \end{corollary}
% \begin{proof}
%     By \cref{lem:derivs-expression} we have
%     \[
%         \partial_{t_F} p_\sigma(\bm{t}) = p_{\sigma \cup \{F\}}(\pi_{\sigma+F}(\bm{t})).
%     \]
%     If $\{F,G\} \not\in X_\sigma(2)$, then $\partial_{t_G} \partial_{t_F} p_\sigma(\bm{t}) = 0$ since $p_{\sigma \cup \{F\}}(\pi_{\sigma+F}(\bm{t}))$ does not depend on $t_G$ (recall that the linear map $\pi_{\sigma+F}$ only depends on $\sigma, F$ and $H\in X_{\sigma\cup\{F\}}(1)$). Otherwise $\{F,G\} \in X_\sigma(2)$, and we have
%     \begin{align*}
%         \partial_{t_G} p_{\sigma \cup \{F\}}(\pi_{\sigma+F}(\bm{t})) &\underset{\text{\cref{lem:chain-rule}}}{=} \left[\nabla_{\pi_{\sigma+F(\bm{1}_G)}} p_{\sigma \cup \{F\}}\right](\pi_{\sigma+F}(\bm{t}))  \\
%             &\underset{\text{\cref{lem:derivs-expression}}}{=} p_{\sigma \cup \{F,G\}}(\pi_{(\sigma \cup \{F\}) + G}(\pi_{\sigma+F}(\bm{t}))) \\
%             &\underset{\text{commutativity of $\bm\pi$}}{=} p_{\sigma \cup \{F,G\}}(\pi_{\sigma + \{F,G\}}(\bm{t})).
%     \end{align*}
%     Therefore $\partial_{t_G} \partial_{t_F} p_\sigma(\bm{t}) = p_{\sigma \cup \{F,G\}}(\pi_{\sigma + \{F,G\}}(\bm{t}))$.
% \end{proof}

We now construct a family of cones for the polynomials defined above.

\begin{definition}[Cone of positive vectors] \label{def:pos-vector}
    Given $\sigma \in X$, we say that a vector $\bm{v} \in \R^{X_\sigma(1)}_{>0}$ is \textbf{$\bm\pi$-positive}  if $\pi_{\sigma+\tau}(\bm{v}) \in \R^{X_{\sigma \cup \tau}(1)}_{>0}$ for all $\tau \in X_\sigma$. We also say $\bm{v}$ is \textbf{$\bm\pi$-non-negative} if the inequalities are not strict.  We further define $\cC_\sigma$ to be the set of all $\bm\pi$-positive vectors. By definition of $\bm\pi$, it follows that $\cC_\sigma$ is an open convex (possibly empty) cone.
\end{definition}

\begin{proposition}[Properties of the cones] \label{prop:cones}
    For all $\sigma \in X$ and $\tau \in X_\sigma$, we have:
    \begin{enumerate}
        \item $\pi_{\sigma+\tau}(\cC_\sigma) \subseteq \cC_{\sigma \cup \tau}$,
        \item $\nabla_{\bm{v}_1} \nabla_{\bm{v}_2} \cdots \nabla_{\bm{v}_{d-|\sigma|}} p_\sigma > 0$ for all $\bm{v}_1,\bm{v}_2,\ldots,\bm{v}_{d-|\sigma|} \in \cC_\sigma$, and
        \item if $\cC_\sigma$ is non-empty, then the closure $\overline{\cC}_\sigma$ is the set of all $\bm\pi$-non-negative vectors in $\R^{X_\sigma(1)}$.
    \end{enumerate}
\end{proposition}
\begin{proof}
    First, $(1)$ is a straightforward computation from the definitions. We next prove $(2)$ by induction on $d-|\sigma|$. Note that if $d-|\sigma| = 0$ (so that $\sigma$ is a facet), we have $p_\sigma = \mu(\sigma) > 0$. For $|\sigma| < d$, let $\bm{u} = \bm{v}_{d-|\sigma|}$. By induction we have
    \begin{align*}
        \nabla_{\bm{v}_1} \cdots \nabla_{\bm{v}_{d-|\sigma|}} p_\sigma &\underset{\text{\cref{lem:derivs-expression}}}{=} \nabla_{\bm{v}_1} \cdots \nabla_{\bm{v}_{d-|\sigma|-1}} \sum_{F \in X_\sigma(1)} u_F \cdot p_{\sigma \cup \{F\}}(\pi_{\sigma+F}(\bm{t})) \\
            &= \sum_{F \in X_\sigma(1)} u_F \cdot \nabla_{\bm{v}_1} \cdots \nabla_{\bm{v}_{d-|\sigma|-1}} p_{\sigma \cup \{F\}}(\pi_{\sigma+F}(\bm{t})) \\
            &\underset{\text{\cref{lem:chain-rule}}}{=} \sum_{F \in X_\sigma(1)} u_F \cdot \left[\nabla_{\pi_{\sigma+F}(\bm{v}_1)} \cdots \nabla_{\pi_{\sigma+F}(\bm{v}_{d-|\sigma|-1})} p_{\sigma \cup \{F\}}\right](\pi_{\sigma+F}(\bm{t})) \\
            &\underset{\text{induction}}{>} 0
    \end{align*}
    where in the last inequality we used that for all $1\leq i\leq d-|\sigma|-1$, the vector $\pi_{\sigma+F}(\bm{v}_i)$ is in $\cC_{\sigma \cup \{F\}}$ by $(1)$ and that $\bm{v}_i\in \cC_{\sigma}$.
    
    %, and $(2)$ follows from $(1)$, the chain rule (\cref{lem:chain-rule}), and \cref{lem:derivs-expression}. 
    For $(3)$, note that if $\bm{u}$ is $\bm\pi$-non-negative and $\bm{v} \in \cC_\sigma$, then $\bm{u} + \epsilon \cdot \bm{v} \in \cC_\sigma$ for all $\epsilon > 0$.
\end{proof}

With the polynomials and associated cones defined, we can now state the main result of this section. It is a standard type of result in the Lorentzian polynomials literature, which says that a connectivity condition along with Lorentzianness of the quadratic $p_\sigma(\bm{t})$, for $\sigma\in X$ of co-dimension 2, implies all $p_\sigma(\bm{t})$ are Lorentzian. This result follows from the same proof as \cite[Thm. 3.8]{BL23}. First we need a lemma.

\begin{lemma} \label{lem:induct-Lorentzian}
    Fix $\sigma \in X$ and $\tau \in X_\sigma$. If $p_{\sigma \cup \tau}(\bm{t})$ is $\cC_{\sigma \cup \tau}$-Lorentzian, then $p_{\sigma \cup \tau}(\pi_{\sigma+\tau}(\bm{t}))$ is $\cC_\sigma$-Lorentzian.
\end{lemma}
\begin{proof}
    For any $\bm{v}_1,\ldots,\bm{v}_k \in \cC_\sigma$, the chain rule (\cref{lem:chain-rule}) implies
    \[
        \nabla_{\bm{v}_1} \cdots \nabla_{\bm{v}_k} p_{\sigma\cup\tau}(\pi_{\sigma+\tau}(\bm{t})) = \left[\nabla_{\pi_{\sigma+\tau}(\bm{v}_1)} \cdots \nabla_{\pi_{\sigma+\tau}(\bm{v}_k)} p_{\sigma\cup\tau}\right](\pi_{\sigma+\tau}(\bm{t})).
    \]
    We now use this to prove (P) and (Q) in \cref{def:C-Lorentzian}. Let $d$ be the homogeneous degree of $p_{\sigma \cup \tau}(\bm{t})$. Setting $k=d$, since $\pi_{\sigma+\tau}(\bm{v}_i) \in \cC_{\sigma\cup\tau}$ for all $i$ by \cref{prop:cones} (1), the fact that $p$ is $\cC_{\sigma\cup\tau}$-Lorentzian implies
    \[
        \left[\nabla_{\pi_{\sigma+\tau}(\bm{v}_1)} \cdots \nabla_{\pi_{\sigma+\tau}(\bm{v}_d)} p_{\sigma\cup\tau}\right](\pi_{\sigma+\tau}(\bm{t})) > 0.
    \]
    Thus (P) is satisfied. Setting $k=d-2$, since $\pi_{\sigma+\tau}(\bm{v}_i) \in \cC_{\sigma\cup\tau}$ for all $i$, the fact that $p$ is $\cC_{\sigma\cup\tau}$-Lorentzian implies the Hessian of
    \[
        \nabla_{\pi_{\sigma+\tau}(\bm{v}_1)} \cdots \nabla_{\pi_{\sigma+\tau}(\bm{v}_{d-2})} p_{\sigma\cup\tau}(\bm{t})
    \]
    has at exactly one positive eigenvalue. Denote this Hessian by $H$. Now since $\pi_{\sigma+\tau}$ is a linear map, the Hessian of $\left[\nabla_{\pi_{\sigma+\tau}(\bm{v}_1)} \cdots \nabla_{\pi_{\sigma+\tau}(\bm{v}_{d-2})} p_{\sigma\cup\tau}\right](\pi_{\sigma+\tau}(\bm{t}))$ is equal to $M^\top H M$, where $M$ is the matrix associated to $\pi$ (by \cref{lem:chain-rule}). Thus by \cref{lem:AB}, the Hessian of $\left[\nabla_{\pi_{\sigma+\tau}(\bm{v}_1)} \cdots \nabla_{\pi_{\sigma+\tau}(\bm{v}_{d-2})} p_{\sigma\cup\tau}\right](\pi_{\sigma+\tau}(\bm{t}))$ has at most one positive eigenvalue. Thus (Q) is satisfied, and therefore $p_{\sigma\cup\tau}(\pi_{\sigma+\tau}(\bm{t}))$ is $\cC_\sigma$-Lorentzian.
\end{proof}

\begin{theorem}[Connected + Quadratics $\implies$ Lorentzian] \label{thm:conn+quad=Lor}
    Suppose $(X,\mu)$ is a $d$-dimensional connected simplicial complex and $\cC_\varnothing$ is non-empty. If the Hessian of $p_\sigma$ has at most one positive eigenvalue for all $\sigma \in X(d-2)$\footnote{Note that $p_\sigma$ is a quadratic, so $\nabla^2 p_\sigma$ is just a matrix}, then $p_\sigma$ is $\cC_\sigma$-Lorentzian for all $\sigma \in X$.
\end{theorem}
\begin{proof}
    Let $\sigma\in X$. First note that by \cref{prop:cones} (2), the (P) condition of the definition of $\cC_\sigma$-Lorentzian (\cref{def:C-Lorentzian}) is satisfied, and thus we only need to prove the (Q) condition.
    
    Note that if co-dimension of $\sigma$ is 1 then the polynomial is linear, and if it is 2 then the result follows immediately from the Hessian condition of the theorem.

    Now we assume $\codim(\sigma)>2$, and we prove the result by induction.
    We will apply \cref{thm:Lor-from-derivs}. First note that condition (1) of \cref{thm:Lor-from-derivs} is already verified. Next, we check condition (3). By \cref{lem:derivs-expression} we have
    \[
        \partial_{t_F} p_\sigma(\bm{t}) = p_{\sigma \cup \{F\}}(\pi_{\sigma+F}(\bm{t})),
    \]
    and by induction $p_{\sigma \cup \{F\}}$ is $\cC_{\sigma \cup \{F\}}$-Lorentzian. Thus $\partial_{t_F} p_\sigma(\bm{t}) = p_{\sigma \cup \{F\}}(\pi_{\sigma+F}(\bm{t}))$ is $\cC_\sigma$-Lorentzian by \cref{lem:induct-Lorentzian}. Thus we have verified condition (3) of \cref{thm:Lor-from-derivs}.

    It remains to check for condition (2). Fix $F,G \in X_\sigma(1)$ with $F \neq G$. By \cref{lem:derivs-expression} we have
    \[
        \partial_{t_G} \partial_{t_F} p_\sigma(\bm{t}) = \begin{cases}
            p_{\sigma \cup \{F,G\}}(\pi_{\sigma+\{F,G\}}(\bm{t})), & \{F,G\} \in X_\sigma(2) \\
            0, & \text{otherwise}.
        \end{cases}
    \]
    If $\{F,G\} \in X_\sigma(2)$, then by induction $p_{\sigma \cup \{F,G\}}$ is $\cC_{\sigma \cup \{F,G\}}$-Lorentzian, and thus $\partial_{t_G} \partial_{t_F} p_\sigma(\bm{t})$ is $\cC_\sigma$-Lorentzian by \cref{lem:induct-Lorentzian}.
    % By the same argument as above, $p_{\sigma \cup \{F,G\}}(\pi_{\sigma + \{F,G\}}(\bm{t}))$ is $\cC_\sigma$-Lorentzian.
    Thus for $\bm{v}_1,\ldots,\bm{v}_{d-|\sigma|-2} \in \cC_\sigma$, we have
    \[
        \nabla_{\bm{v}_1} \cdots \nabla_{\bm{v}_{d-|\sigma|-2}} \partial_{t_G} \partial_{t_F} p_\sigma(\bm{t}) > 0.
    \]
    Therefore the Hessian $H$ of $\nabla_{\bm{v}_1} \cdots \nabla_{\bm{v}_{d-|\sigma|-2}} p_\sigma(\bm{t})$ has non-negative off-diagonal entries, with $H(F,G)$ positive whenever $\{F,G\} \in X_\sigma(2)$. Connectivity of $X$ implies connectivity of the $1$-skeleton of $X_\sigma$, which implies $H$ is irreducible. Thus we have verified condition (2) of \cref{thm:Lor-from-derivs}.
    %
    % Note that for $F \neq G$ that
    % \begin{align*}
    %     \partial_{t_F} \partial_{t_G} p_\sigma(\bm{t}) &\underset{\text{\cref{lem:derivs-expression}}}{=} \partial_{t_F} p_{\sigma\cup\{G\}}(\pi_{\sigma+G}(\bm{t})) \\
    %         &\underset{\text{\cref{lem:chain-rule}}}{=} \left[\nabla_{\pi_{\sigma+G}(\bm{1}_F)} p_{\sigma\cup\{G\}}\right](\pi_{\sigma+G}(\bm{t})) \\
    %         &\underset{\text{def. of $\bm\pi$}}{=} \left[\partial_{t_F} p_{\sigma\cup\{G\}}\right](\pi_{\sigma+G}(\bm{t})) \\
    %         % &= p_{\sigma \cup \{F,G\}}(\pi_{(\sigma \cup \{G\})+F}(\pi_{\sigma+G}(\bm{t})) \\
    %         &\underset{\substack{\text{\cref{lem:derivs-expression},} \\ \text{commutativity of $\bm\pi$}}}{=} \begin{cases}
    %             p_{\sigma \cup \{F,G\}}(\pi_{\sigma+\{F,G\}}(\bm{t})), & \{F,G\} \in X_\sigma(2) \\
    %             0, & \{F,G\} \not\in X_\sigma(2).
    %         \end{cases}
    % \end{align*}
    % Condition (2) then follows from the connectivity condition and the fact that $p_{\sigma \cup \{F,G\}}(\pi_{\sigma+\{F,G\}}(\bm{t}))$ is $\cC_\sigma$-Lorentzian by induction and \cref{prop:cones} (1).
\end{proof}

Combining the following lemma with the above theorem, to check that our polynomials $p_\sigma$ are $\cC$-Lorentzian it is enough to check that the matrix $A+D$ for a link of co-dimension 2 has at most one positive eigenvalue. 
\begin{lemma} \label{lem:quadratic-entries}
    Fix $\sigma \in X(d-2)$, and let $A$ be the adjacency matrix of the $1$-skeleton of $X_\sigma$, weighted according to $\mu$. That is, for all $\{F,G\} \in X_\sigma(2)$ the entry $A(F,G)$ is $\mu(\sigma \cup \{F,G\})$. Further, let $D$ be the diagonal matrix with entries given by
    \[
        D(F,F) = \sum_{G \in X_{\sigma \cup \{F\}}(1)} A(F,G) \cdot \partial_{t_F} \pi_{\sigma+F}(\bm{t})_G = \nabla\left[A \cdot \pi_{\sigma+F}(\bm{t})\right]_F.
    \]
    Then the Hessian of $p_\sigma$ is $A+D$.
\end{lemma}
\begin{proof}
    % By \cref{lem:derivs-expression} we have
    % \[
    %     \partial_{t_F} p_\sigma(\bm{t}) = p_{\sigma \cup \{F\}}(\pi_{\sigma+F}(\bm{t})).
    % \]
    % If $\{F,G\} \not\in X_\sigma(2)$, then $\partial_{t_G} \partial_{t_F} p_\sigma(\bm{t}) = 0$ since $p_{\sigma \cup \{F\}}(\pi_{\sigma+F}(\bm{t}))$ does not depend on $t_G$. % (recall that the linear map $\pi_{\sigma+F}$ only depends on $\sigma, F$ and $H\in X_{\sigma\cup\{F\}}(1)$). 
    % Otherwise $\{F,G\} \in X_\sigma(2)$, and we have
    % \[
    %     \partial_{t_G} p_{\sigma \cup \{F\}(1)}(\pi_{\sigma+F}(\bm{t})) \underset{\text{\cref{lem:derivs-expression}}}{=} p_{\sigma \cup \{F,G\}}(\pi_{(\sigma \cup \{F\}) + G}(\pi_{\sigma+F}(\bm{t}))) = \mu(\sigma \cup \{F,G\}).
    % \]
    By \cref{lem:derivs-expression} we have
    \[
        \partial_{t_G} \partial_{t_F} p_\sigma(\bm{t}) = \begin{cases}
            p_{\sigma \cup \{F,G\}}(\pi_{\sigma+\{F,G\}}(\bm{t})), & \{F,G\} \in X_\sigma(2) \\
            0, & \text{otherwise},
        \end{cases}
    \]
    and if $\{F,G\} \in X_\sigma(2)$, then we further have
    \[
        \partial_{t_G} \partial_{t_F} p_\sigma(\bm{t}) = p_{\sigma \cup \{F,G\}}(\pi_{\sigma+\{F,G\}}(\bm{t})) = \mu(\sigma \cup \{F,G\}).
    \]
    Thus the off-diagonal entries of the Hessian of $p_\sigma$ and of $A$ coincide. For the diagonal entries, we have
    \begin{align*}
        \partial_{t_F}^2 p_\sigma(\bm{t}) = \partial_{t_F} p_{\sigma \cup \{F\}}(\pi_{\sigma+F}(\bm{t})) &\underset{\text{\cref{lem:chain-rule}}}{=} \left[\nabla_{\pi_{\sigma+F}(\bm{1}_F)} p_{\sigma\cup\{F\}}\right](\pi_{\sigma+F}(\bm{t})) \\
        &\underset{\text{linearity of $\pi_{\sigma+F}$}, \eqref{eq:linearpcase}}{=} \left[\sum_{G \in X_{\sigma \cup \{F\}}(1)} (\partial_{t_F} \pi_{\sigma+F}(\bm{t})_G) \cdot \mu(\sigma\cup\{F,G\})\right](\pi_{\sigma+F}(\bm{t})) \\
        &= \sum_{G \in X_{\sigma \cup \{F\}}(1)} (\partial_{t_F} \pi_{\sigma+F}(\bm{t})_G) \cdot A(F,G).
    \end{align*}
\end{proof}

\section{Conic Lorentzian Polynomials for Top-Link Path Complexes}
\label{sec:lorentzianpolynomials-path}

In the rest of the paper we mainly focus on $d$-dimensional top-link path complexes $(X,\mu)$ with parts $T_1,\ldots,T_d$ and we employ \cref{thm:conn+quad=Lor} to study their properties. 
In this section, we give a method for determining a family of $\pi$ maps with the commutativity property.
Armed with such a family of $\pi$ maps, we invoke \cref{thm:conn+quad=Lor} to conclude that $p_\varnothing$ is $\cC_\varnothing$-Lorentzian.
We will conclude the section by showing that this implies a log-concavity inequality associated to $X$ and an eigenvalue bound on the induced weighted subgraph $G_{k,k+1}$ of the 1-skeleton of $X_\varnothing$ for all $1\leq k\leq d-1$.
%and we also do some explicit computations which yield general eigenvalue bounds and log-concavity statements. From now on, we will assume that $X$ is a path complex.

As alluded to in the introduction, one of our main examples of a top-link path  complex is the set of chains of a lattice, and thus we will use some suggestive notation. Specifically, fix $\{K,L\} \in X(2)$ with $K \in T_i$ and $L \in T_j$ for $i < j$. Since $i < j$ we write $K < L$, and if $j = i+1$ then we write $K \prec L$. We also write $(K,L)$ for %\textcolor{red}{Sh: Added below}
$$ (K,L):=\{F\in X(1): K<F<L\}.$$
%the set $X_{\{K,L\}}(1) \cap (T_{i+1} \cup \cdots \cup T_{j-1})$, 
and we write $[K,L]$ for the set $(K,L) \cup \{K,L\}$.

% Let $\cL$ be a ranked bounded lattice with minimal and maximal elements $\hat{0}$ and $\hat{1}$ respectively. Given $K,L \in \cL$ with $K \leq L$, we denote by $[K,L]$ and $(K,L)$ the closed and open intervals in $\cL$, respectively. We further denote the rank difference between $K$ and $L$ as $r_K^L$ and define $d_K^L := r_K^L-1$. Whenever $K$ is covered by $L$, we denote this by $K \prec L$. As described in Fact, the chains in $\cL$ are the faces of a pure partite path complex $X$.

Finally, we abuse notation by defining $T_0 = \{\hat{0}\}$ and $T_{d+1} = \{\hat{1}\}$. The elements $\hat{0}$ and $\hat{1}$ will not actually be a part of any face in $X$, but we will sometimes allow them to replace elements of $X(1)$ in certain contexts. For example, we will denote $(\hat{0},\hat{1}) = X(1)$ and denote $(\hat{0},F)$ to be the set of all $K \in X_{\{F\}}(1)$ such that $K < F$.
% As another example, if $F \in T_k$ then we will denote $r_{\hat{0}}^F = k$ and $r_F^{\hat{1}} = r-k$ as expected. 

\begin{definition} \label{def:alpha-beta}
    Let $\bm\alpha = \{\alpha_K^L: (K,L) \to \R_{\geq 0}\}_{K<L \text{ in } X(1)}$ and $\bm\beta = \{\beta_K^L: (K,L) \to \R_{\geq 0}\}_{K<L \text{ in } X(1)}$ be families of maps, where we allow $K = \hat{0}$ and/or $L = \hat{1}$. We say that these maps are \textbf{commutative} if for all $K < F < G < L$ in $X(1)$, we have:
    \begin{enumerate}
        \item $\alpha_K^L(F) = \alpha_K^G(F) \cdot \alpha_K^L(G)$,
        \item $\beta_K^L(G) = \beta_F^L(G) \cdot \beta_K^L(F)$,
        \item $\alpha_F^L(G) = \alpha_K^L(G) - \beta_F^L(G) \cdot \alpha_K^L(F)$,
        \item $\beta_K^G(F) = \beta_K^L(F) - \alpha_K^G(F) \cdot \beta_K^L(G)$.
    \end{enumerate}
    We will sometimes abuse notation and consider $\alpha_K^L$ and $\beta_K^L$ to be vectors in $\R^{X(1)}$ which are defined to be $0$ outside of $(K,L)$. Finally, we will often suppress $\hat{0}$ and/or $\hat{1}$ in the notation; e.g. $\alpha := \alpha_{\hat{0}}^{\hat{1}}$, $\alpha_F := \alpha_F^{\hat{1}}$, $\beta^F := \beta_{\hat{0}}^F$, and so on.
\end{definition}

\begin{proposition} \label{prop:pi-maps-from-alpha-beta}
    Given a commutative family of maps $\bm\alpha$ and $\bm\beta$, we define a family of maps $\bm\pi$ as follows. Given $\sigma = \{F_1,F_2,\ldots,F_{i-1},F_{i+1},\ldots,F_m\} \in X$ and $F_i \in X_\sigma(1)$ such that $F_1 < F_2 < \cdots < F_m$, define
    \begin{equation} \label{eq:pi-def-from-alpha-beta}
        \pi_{\sigma+F_i}(\bm{t})_H = \begin{cases}
            t_H - \alpha_{F_{i-1}}^{F_i}(H) \cdot t_{F_i}, & H \in (F_{i-1},F_i) \\
            t_H - \beta_{F_i}^{F_{i+1}}(H) \cdot t_{F_i}, & H \in (F_i,F_{i+1}) \\
            t_H, & \text{otherwise},
        \end{cases}
    \end{equation}
    where we denote $F_0 = \hat{0}$ and $F_{m+1} = \hat{1}$.
    %The more general maps $\pi_{\sigma+\tau}$ for $\tau \in X_\sigma$ are then defined inductively using the commutativity property of $\bm\pi$ and any ordering of $\tau$.
    More generally, for $\tau = \{G_1,\ldots,G_k\} \in X_\sigma$, we define
    \[
        \pi_{\sigma+\tau} = \pi_{(\sigma\cup\{G_1,\ldots,G_{k-1}\})+G_k} \circ \cdots \circ \pi_{(\sigma\cup\{G_1,G_2\})+G_3} \circ \pi_{(\sigma\cup\{G_1\})+G_2} \circ \pi_{\sigma+G_1}
    \]
    for any ordering of $\tau$. Then $\bm\pi$ is commutative.
\end{proposition}
\begin{proof}
    To show that $\bm\pi$ is commutative, we only need to show that the definition of $\pi_{\sigma+\tau}$ given above is independent of the choice of ordering of $\tau$. We will do this by showing
    \begin{equation} \label{eq:commutativity-proof-condition}
        \pi_{(\sigma \cup F_i) + F_j}(\pi_{\sigma+F_i}(\bm{t})) = \pi_{(\sigma \cup F_j) + F_i}(\pi_{\sigma + F_j}(\bm{t}))
    \end{equation}
    for any $\sigma \in X$ and $\{F_i,F_j\} \in X_\sigma(2)$.
    
    Fix $\sigma = \{F_1,F_2,\ldots,F_{i-1},F_{i+1},\ldots,F_{j-1},F_{j+1},\ldots,F_m\} \in X$ and $\{F_i,F_j\} \in X_\sigma(2)$ such that $F_k < F_{k+1}$ for all $k$, where $j > i$ and possibly $j = i+1$. Note that \eqref{eq:commutativity-proof-condition} is clear whenever $j > i+1$, since the individual $\pi$ maps $\pi_{\sigma+F_i}(\bm{t})$ and $\pi_{\sigma+F_j}(\bm{t})$ affect disjoint parts of the vector $\bm{t}$.
    
    For $j = i+1$, we define $K = F_{i-1}$, $F = F_i$, $G = F_{i+1}$, and $L = F_{i+2}$, where $F_0 = \hat{0}$ and $F_{m+1} = \hat{1}$. For $H \not\in (K,L)$, \eqref{eq:commutativity-proof-condition} follows from $\pi_{(\sigma \cup F) + G}(\pi_{\sigma+F}(\bm{t}))_H = t_H = \pi_{(\sigma \cup G) + F}(\pi_{\sigma + G}(\bm{t}))_H$. 
    
    Case 1: $H \in (K,F)$ we have
    \[
    \begin{split}
        \pi_{(\sigma \cup G)+F}(\pi_{\sigma+G}(\bm{t}))_H &= \pi_{\sigma+G}(\bm{t})_H - \alpha_K^F(H) \cdot \pi_{\sigma+G}(\bm{t})_F \\
            &= \left(t_H - \alpha_K^G(H) \cdot t_G\right) - \alpha_K^F(H) \cdot \left(t_F - \alpha_K^G(F) \cdot t_G\right) \\
            &= \pi_{\sigma+F}(\bm{t})_H - t_G \cdot \left(\alpha_K^G(H) - \alpha_K^F(H) \cdot \alpha_K^G(F)\right)
            \\
            &\underset{\text{(1) of \cref{def:alpha-beta}}}{=} \pi_{\sigma+F}(\bm{t})_H=
            \pi_{(\sigma \cup F) + G}(\pi_{\sigma+F}(\bm{t}))_H.
            %&= \pi_{(\sigma \cup F)+G}(\pi_{\sigma+F}(\bm{t}))_H - t_G \cdot \left(\alpha_K^G(H) - \alpha_K^F(H) \cdot \alpha_K^G(F)\right),
    \end{split}
    \]
    Note that in the last identity we use the fact that $\pi_{(\sigma \cup F) + G}(\bm{s})_H = s_H$ since $H < F < G$. %$\pi_{\sigma+F}(\bm{t})_G=$
    
    Case 2: $H \in (G,L)$:
    \[
    \begin{split}
        \pi_{(\sigma \cup F)+G}(\pi_{\sigma+F}(\bm{t}))_H &= \pi_{\sigma+F}(\bm{t})_H - \beta_G^L(H) \cdot \pi_{\sigma+F}(\bm{t})_G \\
            &= \left(t_H - \beta_F^L(H) \cdot t_F\right) - \beta_G^L(H) \cdot \left(t_G - \beta_F^L(G) \cdot t_F\right) \\
            &= \pi_{\sigma+G}(\bm{t})_H - t_F \cdot \left(\beta_F^L(H) - \beta_G^L(H) \cdot \beta_F^L(G)\right) \\
            &\underset{\text{(2) of \cref{def:alpha-beta}}}{=} \pi_{\sigma+G}(\bm{t})_H=\pi_{(\sigma \cup G)+F}(\pi_{\sigma+G}(\bm{t}))_H.
            %- t_F \cdot \left(\beta_F^L(H) - \beta_G^L(H) \cdot \beta_F^L(G)\right).
    \end{split}
    \]
    Note that in the last identity we use the fact that $\pi_{(\sigma \cup G) + F}(\bm{s})_H = s_H$ since $F < G < H$.
    
    Case 3: $H \in (F,G)$:
    \[
    \begin{split}
            \pi_{(\sigma \cup F)+G}(\pi_{\sigma+F}(\bm{t}))_H &= \pi_{\sigma+F}(\bm{t})_H - \alpha_F^G(H) \cdot \pi_{\sigma+F}(\bm{t})_G \\
            &= \left(t_H - \beta_F^L(H) \cdot t_F\right) - \alpha_F^G(H) \cdot \left(t_G - \beta_F^L(G) \cdot t_F\right) \\
            &= t_H - t_F \cdot \left(\beta_F^L(H) - \alpha_F^G(H) \cdot \beta_F^L(G)\right) - t_G \cdot \left(\alpha_F^G(H)\right)
    \end{split}
    \]
    and
    \[
    \begin{split}
        \pi_{(\sigma \cup G)+F}(\pi_{\sigma+G}(\bm{t}))_H &= \pi_{\sigma+G}(\bm{t})_H - \beta_F^G(H) \cdot \pi_{\sigma+G}(\bm{t})_F \\
            &= \left(t_H - \alpha_K^G(H) \cdot t_G\right) - \beta_F^G(H) \cdot \left(t_F - \alpha_K^G(F) \cdot t_G\right) \\
            &= t_H - t_F \cdot \left(\beta_F^G(H)\right) - t_G \cdot \left(\alpha_K^G(H) - \beta_F^G(H) \cdot \alpha_K^G(F)\right).
    \end{split}
    \]
    The coefficient of $t_G$ in the above two equations are the same by (3) of \cref{def:alpha-beta} and the coefficients of $t_F$ are the same by (4) of \cref{def:alpha-beta}.
%    In all cases, the fact that $\pi_{(\sigma \cup F)+G}(\pi_{\sigma+F}(\bm{t}))_H = \pi_{(\sigma \cup G)+F}(\pi_{\sigma+G}(\bm{t}))_H$ follows from the fact that $\bm\alpha$ and $\bm\beta$ are hereditary.
\end{proof}

For the rest of this section,  fix a commutative $\bm\alpha$ and $\bm\beta$, and $\bm\pi$ as defined in \cref{prop:pi-maps-from-alpha-beta}.

\begin{lemma}[Interaction of $\bm\pi$, $\bm\alpha$, $\bm\beta$] \label{lem:pi-alpha-beta-relations}
    Fix $\sigma = \{F_1,F_2,\ldots,F_{i-1},F_{i+1},\ldots,F_m\} \in X$ and $F_i \in X_\sigma(1)$ such that $F_1 < F_2 < \cdots < F_m$, where we denote $F_0 = \hat{0}$ and $F_{m+1} = \hat{1}$. Letting $K = F_{i-1}$, $F = F_i$, and $L = F_{i+1}$ we have
    \[
        \pi_{\sigma+F}\left(\alpha_K^L\right) = \alpha_F^L \qquad \text{and} \qquad \pi_{\sigma+F}\left(\beta_K^L\right) = \beta_K^F,
    \]
    where $\alpha_K^L$ and $\beta_K^L$ are considered to be $0$ outside of $(K,L)$ as discussed in \cref{def:alpha-beta}.
\end{lemma}
\begin{proof}
    For $H \in (K,F)$ we compute
    \[
        \pi_{\sigma+F}(\alpha_K^L)_H = \alpha_K^L(H) - \alpha_K^F(H) \cdot \alpha_K^L(F) \underset{\text{(1) of \cref{def:alpha-beta}}}{=} 0
    \]
    and
    \[
        \pi_{\sigma+F}(\beta_K^L)_H = \beta_K^L(H) - \alpha_K^F(H) \cdot \beta_K^L(F) \underset{\text{(4) of \cref{def:alpha-beta}}}{=} \beta_K^F(H),
    \]
    and for $H \in (F,L)$ we compute
    \[
        \pi_{\sigma+F}(\alpha_K^L)_H = \alpha_K^L(H) - \beta_F^L(H) \cdot \alpha_K^L(F) \underset{\text{(3) of \cref{def:alpha-beta}}}{=} \alpha_F^L(H)
    \]
    and
    \[
        \pi_{\sigma+F}(\beta_K^L)_H = \beta_K^L(H) - \beta_F^L(H) \cdot \beta_K^L(F) \underset{\text{(2) of \cref{def:alpha-beta}}}{=} 0.
    \]
    For $H \not\in (K,L)$ we compute
    \[
        \pi_{\sigma+F}(\alpha_K^L)_H = \alpha_K^L(H) = 0
    \]
    and
    \[
        \pi_{\sigma+F}(\beta_K^L)_H = \beta_K^L(H) = 0
    \]
    since we consider $\alpha_K^L$ and $\beta_K^L$ to be $0$ outside of $(K,L)$.
\end{proof}

\begin{corollary} \label{cor:alpha-beta-in-cone}
    If $\cC_\varnothing$ is non-empty, then the vectors $\alpha,\beta$ are elements of $\overline{\cC}_\varnothing$.
\end{corollary}
\begin{proof}
    By inductively applying \cref{lem:pi-alpha-beta-relations}, we have that $\pi_\sigma(\alpha), \pi_\sigma(\beta) \in \R^{X_\sigma(1)}_{\geq 0}$ for all $\sigma \in X$. Thus $\alpha$ and $\beta$ are $\bm\pi$-non-negative. Therefore $\alpha,\beta \in \overline{\cC}_\varnothing$ by (3) of \cref{prop:cones}.
\end{proof}

% \begin{proof}
%     To see that $\rk,\crk \in \overline{\cC}_\varnothing$, let $K = \hat{0}$ and $L = \hat{1}$ and fix $\sigma = \{F_1,F_2,\ldots,F_k\} \in X$ where $F_1 < F_2 < \cdots < F_k$. Using the commutativity of the $\pi$ maps (\cref{lem:pi-maps-commutativity}) to compute $\pi_\sigma(\rk)$ and $\pi_\sigma(\crk)$, we obtain
%     \[
%         \pi_\sigma(\rk)_H = \begin{cases}
%             0, & H < F_k \\
%             \rk_{F_k}^{\hat{1}}(H), & H > F_k
%         \end{cases}
%     \]
%     and
%     \[
%         \pi_\sigma(\crk)_H = \begin{cases}
%             \crk_{\hat{0}}^{F_1}(H), & H < F_1 \\
%             0, & H > F_1
%         \end{cases}.
%     \]
%     This implies $\rk$ and $\crk$ are non-negative vectors, and thus $\rk,\crk \in \overline{C}_\varnothing$ by \cref{lem:non-neg-closure}.
% \end{proof}

\begin{definition}[Contiguous face]
    We say that $\sigma \in X$ is \textbf{contiguous} if there exist $i \leq j$ such that $\sigma \cap T_k \neq \varnothing$ if and only if $i \leq k \leq j$. Note in particular that $\varnothing$ is not contiguous. In such a case, we also say $\sigma$ is $[i,j]$-contiguous. 
\end{definition}

\begin{definition}[Contiguous link]
We say that $\sigma\in X$ of co-dimension (at least) 2 is \textbf{$[i,j]$-link-contiguous} if every facet $\tau\in X_\sigma$ is $[i,j]$-contiguous. 
\end{definition}

We now prove a version of and \cref{lem:quadratic-entries} applied to top-link path complexes. In particular this shows that to prove a  $p_\varnothing$ corresponding to a connected top-link path complex $(X,\mu)$ is $\cC_\varnothing$-Lorentzian it is enough to check Hessians of the contiguous links of co-dimension 2 have one positive eigenvalue. 

\begin{theorem}\label{thm:quadraticcheck}
    Suppose $(X,\mu)$ is a $d$-dimensional connected top-link path complex and $\cC_\varnothing$ is non-empty. Suppose for any $1 \leq i \leq d-1$ and any $[i,i+1]$-link-contiguous $\sigma\in X(d-2)$ the following holds:

    Let $A$ be the adjacency matrix of the 1-skeleton of $X_\sigma$ weighted according to $\mu$, i.e., $A(F,G)=\mu(\sigma\cup \{F,G\})$ and let $D$ be the diagonal matrix, given by 
    $$ D(H,H)=\begin{cases}-\sum_{G\in X_\sigma(1):H\prec G} \beta_H^{\sigma_{i+2}}(G) \cdot A(H,G) &\text{if } H\in T_i\cap X_\sigma(1)\\ -\sum_{F\in X_\sigma(1): F \prec H} \alpha_{\sigma_{i-1}}^H(F) \cdot A(F,H)&\text{if } H\in T_{i+1}\cap X_\sigma(1),\end{cases} $$
    where $\sigma_0 = \hat{0}$ and $\sigma_{d+1} = \hat{1}$.
    The matrix $A+D$ has at most one positive eigenvalue.
    %the Hessian of $p_\sigma$ has at most one positive eigenvalue for all contiguous $\sigma \in X(d-2)$, then $p_\sigma$ is $\cC_\sigma$-Lorentzian for all $\sigma \in X$.

    Then $p_\sigma$ is $\cC_\sigma$-Lorentzian for all $\sigma \in X$.
\end{theorem}
\begin{proof}
    Fix $1 \leq i \leq d-1$ and any $[i,i+1]$-link-contiguous $\sigma \in X(d-2)$. For any $F \prec G$ in $X_\sigma$, we have
    \[
        \partial_{t_F} \pi_{\sigma+F}(\bm{t})_G \underset{\text{\eqref{eq:pi-def-from-alpha-beta}}}{=} -\beta_F^{\sigma_{i+2}}(G) \qquad \text{and} \qquad \partial_{t_G} \pi_{\sigma+G}(\bm{t})_F \underset{\text{\eqref{eq:pi-def-from-alpha-beta}}}{=} -\alpha_{\sigma_{i-1}}^G(F).
    \]
    Thus by \cref{lem:quadratic-entries}, $A+D$ is equal to the Hessian of $p_\sigma$.

    Now by \cref{thm:conn+quad=Lor}, we only need to check the Hessians of $p_\sigma$ for $\sigma \in X(d-2)$ which are not link-contiguous. Fix such a $\sigma$, and let $i < j$ be such that $\sigma_i = \sigma_j = \varnothing$. Now given $\{F,G\} \in X_\sigma(2)$ such that $F < G$, we compute
    \begin{align*}
        % A(F,G) = \mu(\sigma\cup \{F,G\}) &= \P[\sigma\cup\{F,G\} | \sigma_{i+1}]\cdot \P[\sigma_{i+1}]\\
        % &\underset{\text{cond Indep}}{=}\P[\sigma_1,\dots,\sigma_{i-1},F | \sigma_{i+1}]\cdot \P[\sigma_{i+2},\dots,\sigma_{j-1},G,\sigma_{j+1},\dots,\sigma_d | \sigma_{i+1}]\cdot \P[\sigma_{i+1}]
        A(F,G) = \mu(\sigma\cup \{F,G\}) &= \P[\sigma\cup\{F,G\} | \sigma]\cdot \P[\sigma]\\
        &=\P[F,G | \sigma]\cdot \P[\sigma]\\
        &\underset{\text{cond Indep}}{=}\P[F | \sigma]\cdot \P[G | \sigma]\cdot \P[\sigma]
    \end{align*}
    The last equality uses that $\sigma$ is not link-contiguous which implies $j>i+1$. Now, let $v\in \R^{X_\sigma(1)\cap T_i}, w\in \R^{X_\sigma(1)\cap T_j}$, where $v_F=\P[F | \sigma]$ and $w_G=\P[G | \sigma]\cdot \P[\sigma]$. Then, by the above equation,
    $$ A=\begin{bmatrix} 0 & vw^\top \\ wv^\top & 0 \end{bmatrix}$$
    $A$ is a symmetric rank 2 matrix so it has two non-zero eigenvalues. Further, since it is adjacency matrix of a bipartite graph, \cref{lem:AB} implies it has one positive and one negative eigenvalue. 
    Finally, since all entries of $A$ are non-negative, and the $\bm{\alpha},\bm{\beta}$-vectors are non-negative the $D$-matrix defined in \cref{lem:quadratic-entries} has non-positive entries (in fact one can see $D=0$). Therefore, the Hessian of $p_\sigma$, i.e., $A+D$, has at most one positive eigenvalue.
\end{proof}

In the rest of this section, we prove two consequences of constructing $\cC$-Lorentzian polynomials: Namely we first give an eigenvalue bound on the induced subgraphs of the 1-skeleton of $p_\varnothing$, see \cref{cor:eig-bound-alpha-beta}; and then in \cref{cor:log-concavity-alpha-beta} we prove log-concavity inequalities for some sequences of numbers which we will eventually use in \cref{sec:coloredpathcomplexes} to re-derive Heron-Rota-Welsh conjecture and to prove generalizations of Stanley's theorem. 

\begin{definition} \label{def:mu-alpha-beta}
    We define a function $\mu_{\alpha,\beta}: \{\sigma \in X : \sigma \text{ contiguous}\} \to \R_{\geq 0}$ as follows.
    %If $\sigma \in X$ is not contiguous, then $\mu_{\alpha,\beta}(\sigma) := 0$. Otherwise
    For $[i,j]$-contiguous $\sigma$, we define
    \[ \mu_{\alpha,\beta}(\sigma) := \sum_{\tau\in X(d): \sigma\subseteq \tau}\left[\prod_{k=1}^{i-1} \beta^{\tau_{k+1}}(\tau_k)\right] \cdot \mu(\tau) \cdot \left[\prod_{k=j+1}^d \alpha_{\tau_{k-1}}(\tau_k)\right], \]
    where we recall that $\tau_k$ is defined to be $\tau \cap T_k$.
    % \[
    %     \mu_{\alpha,\beta}(\sigma) = \sum_{F_1,\ldots,F_{i-1},F_{k+m+1},\ldots,F_d} \left[\prod_{k=1}^{i-1} \beta^{F_{k+1}}(F_k)\right] \cdot \mu(\{F_1,\ldots,F_{i-1}\} \cup \sigma \cup \{F_{j+1},\ldots,F_d\}) \cdot \left[\prod_{k=j+1}^d \alpha_{F_{k-1}}(F_k)\right]
    % \]
    % where $F_i=\sigma\cap T_i$ and $F_j=\sigma\cap T_j$.
\end{definition}

\begin{lemma}[General mixed derivatives lemma] \label{lem:general-mixed-derivs-expression}
    % Define a function $\mu_{\alpha,\beta}: X \setminus \{\varnothing\} \to \R_{\geq 0}$ as follows. Given $0 \leq k \leq d-1$ and $\sigma = \{F_{i_1},F_{i_2},\ldots,F_{i_m}\} \in X \setminus \{\varnothing\}$ such that $k+1 = i_1 < i_2 < \cdots < i_m$ and $F_{i_j} \in T_{i_j}$ for all $j$, if $i_j = k+j$ for all $j$ then
    % \[
    %     \mu_{\alpha,\beta}(\sigma) = \sum_{\substack{F_1,\ldots,F_k,F_{k+m+1},\ldots,F_d \\ \cF = \{F_1,\ldots,F_d\} \in X(d)}} \left[\prod_{i=1}^k \beta^{F_{i+1}}(F_i)\right] \cdot \mu(\cF) \cdot \left[\prod_{i=k+m+1}^d \alpha_{F_{i-1}}(F_i)\right]
    % \]
    % and otherwise $f(\sigma) = 0$. Then for any $k$ and $\sigma$ as above, we have that $\nabla_{\beta}^k \nabla_{\alpha}^{d-k-m} \left[p_\sigma(\pi_\sigma(\bm{t}))\right] = f(\sigma)$.
    We have that
    \[
        \nabla_\beta^{i-1} \nabla_\alpha^{d-j} p_\sigma(\pi_\sigma(\bm{t})) = \begin{cases}
            \mu_{\alpha,\beta}(\sigma), & \text{if $\sigma$ is $[i,j]$-contiguous} \\
            0, & \text{otherwise}.
        \end{cases}
    \]
\end{lemma}
\begin{proof}Fix $\sigma\in X$ and $i,j \in [d]$. We compute $\nabla_\beta^{i-1} \nabla_\alpha^{d-j} p_\sigma(\pi_\sigma(\bm{t}))$ by induction on $\codim(\sigma)$.
    Let $a\in [d]$ be the smallest number such that $\sigma_a \neq \varnothing$, and let $b\in [d]$ be the largest number such that $\sigma_b \neq \varnothing$.
    % We prove this by induction on $d-m$.
    
    First, if $\codim(\sigma) = 0$ then $\sigma$ is a facet of $X$ which implies $a = 1$ and $b=d$, and the result immediately follows since $p_\sigma(\bm{t}) = \mu(\sigma)$.
    
    For the inductive step, assume that $\codim(\sigma) > 0$. For ease of notation, we denote $K = \sigma_a$ and $L = \sigma_b$. Suppose first that $d-j > 0$. Then we have
    \[
    \begin{split}
        \nabla_{\beta}^{i-1} \nabla_{\alpha}^{d-j} \left[p_\sigma(\pi_\sigma(\bm{t}))\right] &\underset{\text{\cref{lem:chain-rule}}}{=} \nabla_{\beta}^{i-1} \nabla_{\alpha}^{d-j-1} \left[\left[\nabla_{\pi_\sigma(\alpha)} p_\sigma\right](\pi_\sigma(\bm{t}))\right] \\
            &\underset{\text{\cref{lem:pi-alpha-beta-relations}}}{=} \nabla_{\beta}^{i-1} \nabla_{\alpha}^{d-j-1} \left[\left[\nabla_{\alpha_L} p_\sigma\right](\pi_\sigma(\bm{t}))\right] \\
            &\underset{\text{\cref{lem:derivs-expression}}}{=} \nabla_{\beta}^{i-1} \nabla_{\alpha}^{d-j-1} \left[\left[\sum_{H: L < H} \alpha_L(H) \cdot p_{\sigma \cup \{H\}}(\pi_{\sigma+H}(\cdot))\right](\pi_\sigma(\bm{t}))\right] \\
            &\underset{\text{commutativity of $\bm\pi$}}{=} \sum_{H: L < H} \alpha_L(H) \cdot \nabla_{\beta}^{i-1} \nabla_{\alpha}^{d-j-1} \left[p_{\sigma \cup \{H\}}(\pi_{\sigma \cup \{H\}}(\bm{t}))\right]
    \end{split}
    \]
    First, if $b = d$ (which implies $b > j$) then the above expression is $0$ as desired, since there is no $H$ such that $L < H$. Otherwise by induction, each term of the right-hand side can only be non-zero if $\sigma \cup \{H\}$ is $[i,j+1]$-contiguous. Thus either this expression is $0$, or else $\sigma$ is $[i,j]$-contiguous and $H \in T_{j+1}$ and by induction we have
    \[
    \begin{split}
        \nabla_{\beta}^{i-1} \nabla_{\alpha}^{d-j} \left[p_\sigma(\pi_\sigma(\bm{t}))\right] &\underset{\text{induction}}{=} \sum_{H: L \prec H} \alpha_L(H) \cdot \mu_{\alpha,\beta}(\sigma \cup \{H\}) \\
            &\underset{\text{\cref{def:mu-alpha-beta}}}{=} \sum_{H: L \prec H} \alpha_L(H) \sum_{\tau\in X(d): \sigma \cup \{H\}\subseteq \tau}\left[\prod_{k=1}^{i-1} \beta^{\tau_{k+1}}(\tau_k)\right] \cdot \mu(\tau) \cdot \left[\prod_{k=j+2}^d \alpha_{\tau_{k-1}}(\tau_k)\right] \\
            &= \sum_{\tau\in X(d): \sigma \subseteq \tau}\left[\prod_{k=1}^{i-1} \beta^{\tau_{k+1}}(\tau_k)\right] \cdot \mu(\tau) \cdot \left[\prod_{k=j+1}^d \alpha_{\tau_{k-1}}(\tau_k)\right] \\
            &\underset{\text{\cref{def:mu-alpha-beta}}}{=} \mu_{\alpha,\beta}(\sigma)
    \end{split}
    \]
    as desired.
    
    In the case that $i > 1$, essentially the same proof works by first applying the computation above to $\nabla_\beta$ instead of $\nabla_\alpha$. First we would obtain
    \[
        \nabla_{\beta}^{i-1} \nabla_{\alpha}^{d-j} \left[p_\sigma(\pi_\sigma(\bm{t}))\right] = \sum_{H: H < K} \beta^K(H) \cdot \nabla_{\beta}^{i-2} \nabla_{\alpha}^{d-j} \left[p_{\sigma \cup \{H\}}(\pi_{\sigma \cup \{H\}}(\bm{t}))\right].
    \]
    If $a = 1$ (which implies $a < i$) then the above expression is $0$ as above. Otherwise by induction, each term of the right-hand side can only be non-zero if $\sigma \cup \{H\}$ is $[i-1,j]$-contiguous. Thus either this expression is $0$, or else $\sigma$ is $[i,j]$-contiguous and $H \in T_{i-1}$ and a similar argument as given above gives
    \[
        \nabla_{\beta}^{i-1} \nabla_{\alpha}^{d-j} \left[p_\sigma(\pi_\sigma(\bm{t}))\right] = \mu_{\alpha,\beta}(\sigma)
    \]
    as desired.
\end{proof}

\begin{corollary}[Eigenvalue bound for the $\alpha$-$\beta$ Hessians] \label{cor:eig-bound-alpha-beta}
    Given $1 \leq k \leq d-1$, let $B_k$ be the bipartite graph on $T_k \cup T_{k+1}$ with edge weights given by $\mu_{\alpha,\beta}(\{F,G\})$ for $F \prec G$, let $A_k$ be the weighted adjacency matrix of $B_k$, and let $D_k$ be the diagonal weighted degree matrix of $B_k$. Further, let $M_k$ be the diagonal matrix with entries given by
    \[
        m_k(F,F) = \frac{\sum_{G:F \prec G} \beta_F(G) \cdot \mu_{\alpha,\beta}(\{F,G\})}{\sum_{G:F \prec G} \mu_{\alpha,\beta}(\{F,G\})} \qquad \text{for} \quad F \in T_k
    \]
    and
    \[
        m_k(G,G) = \frac{\sum_{F:F \prec G} \alpha^G(F) \cdot \mu_{\alpha,\beta}(\{F,G\})}{\sum_{F:F \prec G} \mu_{\alpha,\beta}(\{F,G\})} \qquad \text{for} \quad G \in T_{k+1}.
    \]
    % \[
    %     m_k(H,H) = \begin{cases}
    %         \sum_{G:H \prec G} \beta_H(G) \cdot f(\{H,G\}), & H \in T_k \\
    %         \sum_{F:F \prec H} \alpha^H(F) \cdot f(\{F,H\}), & H \in T_{k+1}.
    %     \end{cases}
    % \]
    If $\cC_\varnothing$ is non-empty and $p_\varnothing$ is $\cC_\varnothing$-Lorentzian, then
    \[
        D_k^{-1/2} A_k D_k^{-1/2} \preceq M_k + \bm{w}\bm{w}^\top
    \]
    for some real vector $\bm{w}$.
\end{corollary}
\begin{proof}
    Note first that for $\{F,G\} \in X(2)$ with $F < G$, by \cref{lem:derivs-expression} we have
    \[
        \partial_{t_F} \partial_{t_G} p_\varnothing(\bm{t}) = p_{\{F,G\}}(\pi_{\{F,G\}}(\bm{t})).
    \]
    By \cref{lem:general-mixed-derivs-expression}, we then have
    \[
        \nabla_{\beta}^{k-1} \nabla_{\alpha}^{d-k-1} \partial_{t_F} \partial_{t_G} p_\varnothing(\bm{t}) = \begin{cases}
            \mu_{\alpha,\beta}(\sigma) > 0, & F \in T_k, G \in T_{k+1} \\
            0, & \text{otherwise}.
        \end{cases}
    \]
    Next for any $H \in X(1)$ we compute
    \[
    \begin{split}
        \partial_{t_H}^2 p_\varnothing(\bm{t}) &\underset{\text{\cref{lem:derivs-expression}}}{=} \partial_{t_H} p_{\{H\}}(\pi_{\{H\}}(\bm{t})) \\
            &\underset{\text{\cref{lem:chain-rule}}}{=} \left[\nabla_{\pi_{\{H\}}(\bm{1}_H)} p_{\{H\}}\right](\pi_{\{H\}}(\bm{t})) \\
            &\underset{\text{\cref{eq:pi-def-from-alpha-beta}}}{=} -\sum_{F: F < H} \alpha^H(F) \cdot p_{\{F,H\}}(\pi_{\{F,H\}}(\bm{t})) - \sum_{G: H < G} \beta_H(G) \cdot p_{\{H,G\}}(\pi_{\{H,G\}}(\bm{t})),
    \end{split}
    \]
    and by \cref{lem:general-mixed-derivs-expression} we have
    \[
        \nabla_{\beta}^{k-1} \nabla_{\alpha}^{d-k-1} \partial_{t_H}^2 p_\varnothing(\bm{t}) = \begin{cases}
            -\sum_{G: H \prec G} \beta_H(F) \cdot \mu_{\alpha,\beta}(\{H,G\}), & H \in T_k \\
            -\sum_{F: F \prec H} \alpha^H(F) \cdot \mu_{\alpha,\beta}(\{F,H\}), & H \in T_{k+1} \\
            0, & \text{otherwise}
        \end{cases}.
    \]
    Thus the Hessian $\cH$ of $\nabla_{\beta}^{k-1} \nabla_{\alpha}^{d-k-1} p_\varnothing(\bm{t})$ is supported only on $T_k \cup T_{k+1}$, and $\cH = A_k - D_k^{1/2}M_kD_k^{1/2}$. Since $p_\varnothing$ is $\cC_\varnothing$-Lorentzian and $\alpha,\beta \in \overline{\cC}_\varnothing$ by \cref{cor:alpha-beta-in-cone}, we thus have
    \[
        A_k - D_k^{1/2}M_kD_k^{1/2} = \cH \preceq \bm{v}\bm{v}^\top
    \]
    by \cref{lem:main-purpose-lemma} for some real vector $\bm{v}$, and this implies the desired result.
\end{proof}

\begin{corollary}[Log-concavity for the $\alpha$-$\beta$ coefficients] \label{cor:log-concavity-alpha-beta}
    Given $1 \leq k \leq d-1$ we have
    \[
        \sum_{F \in T_k} \beta(F) \cdot \mu_{\alpha,\beta}(\{F\}) = \sum_{G \in T_{k+1}} \alpha(G) \cdot \mu_{\alpha,\beta}(\{G\}),
    \]
    and we denote this quantity by $c_k$. Further, define
    \[
        c_0 = \sum_{G \in T_1} \alpha(G) \cdot \mu_{\alpha,\beta}(\{G\}) \qquad \text{and} \qquad c_d = \sum_{F \in T_d} \beta(F) \cdot \mu_{\alpha,\beta}(\{F\}).
    \]
    If $\cC_\varnothing$ is non-empty and $p_\varnothing$ is $\cC_\varnothing$-Lorentzian, then $(c_0,c_1,\ldots,c_d)$ forms a log-concave sequence.
\end{corollary}
\begin{proof}
    For $1 \leq k \leq d$ we have
    \[
        \nabla_{\beta}^k \nabla_{\alpha}^{d-k} p_\varnothing(\bm{t}) \underset{\text{\cref{lem:derivs-expression}}}{=} \nabla_{\beta}^{k-1} \nabla_{\alpha}^{d-k} \sum_{F \in T_k} \beta(F) \cdot p_F(\pi_F(\bm{t})) \underset{\text{\cref{lem:general-mixed-derivs-expression}}}{=} \sum_{F \in T_k} \beta(F) \cdot \mu_{\alpha,\beta}(\{F\}),
    \]
    and for $0 \leq k \leq d-1$ we have
    \[
        \nabla_{\beta}^k \nabla_{\alpha}^{d-k} p_\varnothing(\bm{t}) \underset{\text{\cref{lem:derivs-expression}}}{=} \nabla_{\beta}^k \nabla_{\alpha}^{d-k-1} \sum_{G \in T_{k+1}} \alpha(G) \cdot p_G(\pi_G(\bm{t})) \underset{\text{\cref{lem:general-mixed-derivs-expression}}}{=} \sum_{G \in T_{k+1}} \alpha(G) \cdot \mu_{\alpha,\beta}(\{G\}).
    \]
    Thus we have
    \[
        d! \cdot p_\varnothing(\beta \cdot x + \alpha) = \sum_{k=0}^d \binom{d}{k} c_k \cdot x^k,
    \]
    and therefore $(c_0,c_1,\ldots,c_d)$ forms a log-concave sequence by \cref{cor:alpha-beta-in-cone} and \cref{lem:main-purpose-lemma}.
\end{proof}

%\subsection{Example 1: Polynomials via lattice rank} 
\section{Conic Lorentzian Polynomials for Colored Top-Link Path Complexes}\label{sec:coloredpathcomplexes}
In this section, we give concrete examples of $\bm{\alpha},\bm{\beta}$ vectors and the corresponding $\bm{\pi}$-maps. We then prove \cref{thm:log-concavity-colored-complex}.

%We now instantiate the previous sections with a choice of $\bm\alpha$ and $\bm\beta$, based on a
Throughout the section we fix a colored top-link path complex $(X,\mu,\phi)$ (see \cref{sec:introcolor} for the definition).

\begin{lemma} \label{lem:alpha-beta-gen}
    Let $\phi: X(2) \to \R_{>0}$ be such that for all $K < F < G < L$ in $X(1)$ we have
    \[
        \phi(K,G) \cdot \phi(F,L) - \phi(K,L) \cdot \phi(F,G) = \phi(K,F) \cdot \phi(G,L).
    \]
    If for $K < F < L$ in $X(1)$ we define
    \[
        % \alpha_K^L(F) = \frac{r_K^F}{r_K^L} \qquad \text{and} \qquad \beta_K^L(F) = \frac{r_F^L}{r_K^L},
        \alpha_K^L(F) = \frac{\phi(K,F)}{\phi(K,L)} \qquad \text{and} \qquad \beta_K^L(F) = \frac{\phi(F,L)}{\phi(K,L)},
    \]
    then $\bm\alpha$ and $\bm\beta$ are commutative and give rise to a commutative family of maps $\bm\pi$ and their associated polynomials $p_\sigma$ for $\sigma \in X$.
\end{lemma}
\begin{proof}
    By \cref{prop:pi-maps-from-alpha-beta}, we just need to show that the conditions of \cref{def:alpha-beta} hold for $\bm\alpha$ and $\bm\beta$. The first two conditions follow from
    \[
        % \alpha_K^G(F) \cdot \alpha_K^L(G) = \frac{r_K^F}{r_K^G} \cdot \frac{r_K^G}{r_K^L} = \frac{r_K^F}{r_K^L} = \alpha_K^L(F)
        \alpha_K^G(F) \cdot \alpha_K^L(G) = \frac{\phi(K,F)}{\phi(K,G)} \cdot \frac{\phi(K,G)}{\phi(K,L)} = \frac{\phi(K,F)}{\phi(K,L)} = \alpha_K^L(F)
    \]
    and
    \[
        \beta_F^L(G) \cdot \beta_K^L(F) = \frac{\phi(G,L)}{\phi(F,L)} \cdot \frac{\phi(F,L)}{\phi(K,L)} = \frac{\phi(G,L)}{\phi(K,L)} = \beta_K^L(G).
    \]
    The last two conditions then follow from
    \begin{align*}
        \alpha_K^L(G) - \alpha_F^L(G) &= \frac{\phi(K,G)}{\phi(K,L)} - \frac{\phi(F,G)}{\phi(F,L)} 
        = \frac{\phi(K,G) \cdot \phi(F,L) - \phi(F,G) \cdot \phi(K,L)}{\phi(K,L) \cdot \phi(F,L)} \\
        &\underset{\text{Lem's assumption}}{=} \frac{\phi(K,F) \cdot \phi(G,L)}{\phi(K,L) \cdot \phi(F,L)} = \beta_F^L(G) \cdot \alpha_K^L(F)
    \end{align*}
    and
    \begin{align*}
        \beta_K^L(F) - \beta_K^G(F) &= \frac{\phi(F,L)}{\phi(K,L)} - \frac{\phi(F,G)}{\phi(K,G)} 
        = \frac{\phi(F,L) \cdot \phi(K,G) - \phi(F,G) \cdot \phi(K,L)}{\phi(K,L) \cdot \phi(K,G)}\\
        &\underset{\text{Lem's assumption}}{=} \frac{\phi(K,F) \cdot \phi(G,L)}{\phi(K,L) \cdot \phi(K,G)} = \alpha_K^G(F) \cdot \beta_K^L(G).
    \end{align*}
    % Notice the first two conditions are essentially immediate, and $(j+k) \cdot (i+j) - j \cdot (i+j+k) = i \cdot k$ is the key identity used for the last two conditions.
\end{proof}

Defining $\phi(F,G) = \phi(G)-\phi(F)$ for $F < G$, we obtain the following.

\begin{corollary} \label{cor:alpha-beta-v1}
    If for $K < F < L$ in $X(1)$ we define
    \begin{equation}\label{eq:alphaphidef}
        % \alpha_K^L(F) = \frac{r_K^F}{r_K^L} \qquad \text{and} \qquad \beta_K^L(F) = \frac{r_F^L}{r_K^L},
        \alpha_K^L(F) = \frac{\phi(F)-\phi(K)}{\phi(L)-\phi(K)} \qquad \text{and} \qquad \beta_K^L(F) = \frac{\phi(L)-\phi(F)}{\phi(L)-\phi(K)},
    \end{equation}
    then $\bm\alpha$ and $\bm\beta$ are commutative and give rise to a commutative family of maps $\bm\pi$ and their associated polynomials $p_\sigma$ for $\sigma \in X$.
\end{corollary}

To utilize the results of the previous sections, we need to show that the cone of positive vectors is non-empty.

\begin{lemma} \label{lem:cone-non-empty-v1}
    The cone $\cC_\sigma$ is non-empty for all $\sigma \in X$.
\end{lemma}
\begin{proof}
    Given $\sigma = \{F_1,F_2,\ldots,F_k\} \in X$ where $F_1 < F_2 < \cdots < F_k$, let
    \[
        \bm{t}=(c_0 \cdot \alpha^{F_1} \cdot \beta^{F_1}) \oplus (c_1 \cdot \alpha_{F_1}^{F_2} \cdot \beta_{F_1}^{F_2}) \oplus (c_2 \cdot \alpha_{F_2}^{F_3} \cdot \beta_{F_2}^{F_3}) \oplus \cdots \oplus (c_k \cdot \alpha_{F_k} \cdot \beta_{F_k}) \in \cC_\sigma
    \] 
     In other words,  for any $F_i<H<F_{i+1}$ (where assume $F_0=\hat{0},F_{k+1}=\hat{1}$), $$t_H=c_i\cdot\alpha_{F_i}^{F_{i+1}}(H)\cdot\beta_{F_i}^{F_{i+1}}(H).$$
    We show that $\bm{t}\in C_\sigma$ for all $c_0,c_1,\ldots,c_k > 0$.
    
    To that end,  let $F \in X_\sigma(1)$ be such that $F_i < F < F_{i+1}$ (for some $0\leq i\leq k$). For ease of notation, let $K=F_i,L=F_{i+1}$. We will show that $\pi_{\sigma+F}(\alpha_K^L \cdot \beta_K^L) = (c \cdot \alpha_K^F \cdot \beta_K^F) \oplus (c' \cdot \alpha_F^L \cdot \beta_F^L)$ for some $c,c' > 0$. By induction on $|\sigma|$ and the fact that $\bm\pi$ is commutative, this will prove the vector given above lies in $\cC_\sigma$.

    For $H \in (K,F)$, we have %Using the fact that $\bm\alpha,\bm\beta$ are commutative we compute
    \[
    \begin{split}
        \pi_{\sigma+F}(\alpha_K^L \cdot \beta_K^L)_H &\underset{\eqref{eq:pi-def-from-alpha-beta}}{=} \alpha_K^L(H) \cdot \beta_K^L(H) - \alpha_K^F(H) \cdot \alpha_K^L(F) \cdot \beta_K^L(F) \\
            &\underset{\text{(1) of \cref{def:alpha-beta}}}{=} \alpha_K^L(H) \cdot \beta_K^L(H) - \alpha_K^L(H) \cdot \beta_K^L(F) \\
            &\underset{\eqref{eq:alphaphidef}}{=} \alpha_K^L(H) \cdot \left(\frac{\phi(L)-\phi(H)}{\phi(L)-\phi(K)} - \frac{\phi(L)-\phi(F)}{\phi(L)-\phi(K)}\right) \\
            &= \frac{\phi(H)-\phi(K)}{\phi(L)-\phi(K)} \cdot \frac{\phi(F)-\phi(H)}{\phi(L)-\phi(K)} \\
            &\underset{\eqref{eq:alphaphidef}}{=} \alpha_K^L(F)^2 \cdot \left[\alpha_K^F(H) \cdot \beta_K^F(H)\right].
    \end{split}
    \]
    Now, fix $H \in (F,L)$, and we can similarly write %and again using the fact that $\bm\alpha,\bm\beta$ are hereditary we compute
    \[
    \begin{split}
        \pi_{\sigma+F}(\alpha_K^L \cdot \beta_K^L)_H &\underset{\eqref{eq:pi-def-from-alpha-beta}}{=} \alpha_K^L(H) \cdot \beta_K^L(H) - \beta_F^L(H) \cdot \alpha_K^L(F) \cdot \beta_K^L(F) \\
                &\underset{\text{(2) of \cref{def:alpha-beta}}}{=} \alpha_K^L(H) \cdot \beta_K^L(H) - \beta_K^L(H) \cdot \alpha_K^L(F) \\
        &\underset{\eqref{eq:alphaphidef}}{=}\beta_K^L(H) \cdot \left(\frac{\phi(H)-\phi(K)}{\phi(L)-\phi(K)} - \frac{\phi(F)-\phi(K)}{\phi(L)-\phi(K)}\right) \\
            &= \frac{\phi(L)-\phi(H)}{\phi(L)-\phi(K)} \cdot \frac{\phi(H)-\phi(F)}{\phi(L)-\phi(K)} \\
      &\underset{\eqref{eq:alphaphidef}}{=} \beta_K^L(F)^2 \cdot \left[\alpha_F^L(H) \cdot \beta_F^L(H)\right].
    \end{split}
    \]
    %Note that at the end of each computation we used the explicit definition of the $\bm\alpha$ and $\bm\beta$ vectors, rather than only using the fact that they are hereditary.
\end{proof}

The rest of this section is organized as follows: First, we prove \cref{thm:log-concavity-colored-complex}. %In the remainder of this section, we will utilize the log-concavity statement of \cref{thm:log-concavity-colored-complex} and 
Then, we utilize it to demonstrate various log-concavity results based on different colorings. %To this end, we first prove \cref{thm:log-concavity-colored-complex} which gives a nicer formula for the $c_k$ values of \cref{cor:log-concavity-alpha-beta}.

\begin{proof}[Proof of \cref{thm:log-concavity-colored-complex}]
    By \cref{thm:quadraticcheck}, to prove $p_\varnothing$ is $\cC_\varnothing$, it suffices to show that $\cC_\varnothing$ is non-empty and that  for all link-contiguous $\sigma \in X(d-2)$ the corresponding matrix $A+D$ has at most one positive eigenvalue. The latter condition is equivalent to $(X,\mu,\phi)$ being a top-link expander, and the former condition follows from \cref{lem:cone-non-empty-v1}.

    Now using \cref{def:mu-alpha-beta}, for any $k \in [d]$ and $F \in T_k$ we have
    \[
    \begin{split}
        \mu_{\alpha,\beta}(\{F\}) &= \sum_{\tau \in X(d): F \in \tau} \left[\prod_{i=1}^{k-1} \beta^{\tau_{i+1}}(\tau_i)\right] \cdot \mu(\tau) \cdot \left[\prod_{i=k+1}^d \alpha_{\tau_{i-1}}(\tau_i)\right] \\
            &= \sum_{\tau \in X(d): F \in \tau} \left[\prod_{i=1}^{k-1} \frac{\phi(\tau_{i+1}) - \phi(\tau_i)}{\phi(\tau_{i+1}) - \phi(\hat{0})}\right] \cdot \mu(\tau) \cdot \left[\prod_{i=k+1}^d \frac{\phi(\tau_i) - \phi(\tau_{i-1})}{\phi(\hat{1}) - \phi(\tau_{i-1})}\right] \\
            &= \sum_{\tau \in X(d): F \in \tau} \mu(\tau) \cdot \frac{\prod_{i=1}^{d-1} (\phi(\tau_{i+1}) - \phi(\tau_i))}{\prod_{i=2}^k (\phi(\tau_i) - \phi(\hat{0})) \prod_{i=k}^{d-1} (\phi(\hat{1}) - \phi(\tau_i))}.
    \end{split}
    \]
    % By   \cref{lem:cone-non-empty-v1}, $\cC_\varnothing$ is non-empty. Thus, by \cref{thm:conn+quad=Lor}, $p_\varnothing$ (corresponding to the $\bm{\pi}$-map and the $\bm{\alpha},\bm{\beta}$-vectors defined in \eqref{eq:alphaphidef} is $\cC_\varnothing$-Lorentzian.
    For a facet $\tau\in X(d)$, let  $\tau_0 = \hat{0}$ and $\tau_{d+1} = \hat{1}$. Therefore, for   by \cref{cor:log-concavity-alpha-beta} we obtain a log-concave sequence $(c_0,c_1,\ldots,c_d)$  where for $1 \leq k \leq d$,
    \[
    \begin{split}
        c_k &= \sum_{F \in T_k} \beta(F) \cdot \mu_{\alpha,\beta}(\{F\}) \\
            &= \sum_{F \in T_k} \frac{\phi(\hat{1}) - \phi(F)}{\phi(\hat{1}) - \phi(\hat{0})} \sum_{\tau \in X(d): F \in \tau} \mu(\tau) \cdot \frac{\prod_{i=1}^{d-1} (\phi(\tau_{i+1}) - \phi(\tau_i))}{\prod_{i=2}^k (\phi(\tau_i) - \phi(\hat{0})) \prod_{i=k}^{d-1} (\phi(\hat{1}) - \phi(\tau_i))} \\
            &= \sum_{\tau \in X(d)} \mu(\tau) \cdot \frac{\phi(\hat{1}) - \phi(\tau_k)}{\phi(\hat{1}) - \phi(\hat{0})} \cdot \frac{\prod_{i=1}^{d-1} (\phi(\tau_{i+1}) - \phi(\tau_i))}{\prod_{i=2}^k (\phi(\tau_i) - \phi(\hat{0})) \prod_{i=k}^{d-1} (\phi(\hat{1}) - \phi(\tau_i))} \\
            &= \frac{1}{\phi(\hat{1}) - \phi(\hat{0})} \sum_{\tau \in X(d)} \mu(\tau) \cdot \frac{\prod_{i=0}^{d} (\phi(\tau_{i+1}) - \phi(\tau_i))}{\prod_{i=1}^k (\phi(\tau_i) - \phi(\hat{0})) \prod_{i={k+1}}^{d} (\phi(\hat{1}) - \phi(\tau_i))},
    \end{split}
    \]
    In the special case that $k=0$ it can be similarly checked that
    \[
    \begin{split}
        c_0 &= \sum_{F \in T_{1}} \alpha(F) \cdot \mu_{\alpha,\beta}(\{F\}) \\
            &= \sum_{F \in T_{1}} \frac{\phi(F) - \phi(\hat{0})}{\phi(\hat{1}) - \phi(\hat{0})} \sum_{\tau \in X(d): F \in \tau} \mu(\tau) \cdot \frac{\prod_{i=1}^{d-1} (\phi(\tau_{i+1}) - \phi(\tau_i))}{
            %\prod_{i=2}^{2} (\phi(\tau_i) - \phi(\hat{0})) 
            \prod_{i=1}^{d-1} (\phi(\hat{1}) - \phi(\tau_i))} \\
            &= %\underset{k=0}{=} 
            \frac{1}{\phi(\hat{1}) - \phi(\hat{0})} \sum_{\tau \in X(d)} \mu(\tau) \cdot \frac{\prod_{i=0}^{d} (\phi(\tau_{i+1}) - \phi(\tau_i))}{
            %\prod_{i=2}^k (\phi(\tau_i) - \phi(\hat{0})) 
            \prod_{i=k}^{d} (\phi(\hat{1}) - \phi(\tau_i))}
    \end{split}
    \]
as desired.
    % \[
    % \begin{split}
    %     c_k &= \sum_{F \in T_{k+1}} \alpha(F) \cdot \mu_{\alpha,\beta}(\{F\}) \\
    %         &= \sum_{F \in T_{k+1}} \frac{\phi(F) - \phi(\hat{0})}{\phi(\hat{1}) - \phi(\hat{0})} \sum_{\tau \in X(d): F \in \tau} \mu(\tau) \cdot \frac{\prod_{i=1}^{d-1} (\phi(\tau_{i+1}) - \phi(\tau_i))}{\prod_{i=2}^{k+1} (\phi(\tau_i) - \phi(\hat{0})) \prod_{i=k+1}^{d-1} (\phi(\hat{1}) - \phi(\tau_i))} \\
    %         &= \sum_{\tau \in X(d)} \mu(\tau) \cdot \frac{\phi(\tau_{k+1}) - \phi(\hat{0})}{\phi(\hat{1}) - \phi(\hat{0})} \cdot \frac{\prod_{i=1}^{d-1} (\phi(\tau_{i+1}) - \phi(\tau_i))}{\prod_{i=2}^{k+1} (\phi(\tau_i) - \phi(\hat{0})) \prod_{i=k+1}^{d-1} (\phi(\hat{1}) - \phi(\tau_i))} \\
    %         &= \frac{1}{\phi(\hat{1}) - \phi(\hat{0})} \sum_{\tau \in X(d)} \mu(\tau) \cdot \frac{\prod_{i=1}^{d-1} (\phi(\tau_{i+1}) - \phi(\tau_i))}{\prod_{i=2}^k (\phi(\tau_i) - \phi(\hat{0})) \prod_{i=k+1}^{d-1} (\phi(\hat{1}) - \phi(\tau_i))}
    % \end{split}
    % \]
    % Thus for all $0 \leq k \leq d$, we have
    % \[
    % \begin{split}
    %     c_k &= \frac{1}{\phi(\hat{1}) - \phi(\hat{0})} \sum_{\tau \in X(d)} \mu(\tau) \cdot \frac{\prod_{i=1}^{d-1} (\phi(\tau_{i+1}) - \phi(\tau_i))}{\prod_{i=2}^k (\phi(\tau_i) - \phi(\hat{0})) \prod_{i=k+1}^{d-1} (\phi(\hat{1}) - \phi(\tau_i))} \\
    %         &= \frac{1}{\phi(\hat{1}) - \phi(\hat{0})} \sum_{\tau \in X(d)} \mu(\tau) \cdot \frac{\prod_{i=0}^d (\phi(\tau_{i+1}) - \phi(\tau_i))}{\prod_{i=1}^k (\phi(\tau_i) - \phi(\hat{0})) \prod_{i=k+1}^d (\phi(\hat{1}) - \phi(\tau_i))},
    % \end{split}
    % \]
    % % where $\tau_0 = \hat{0}$ and $\tau_{d+1} = \hat{1}$.
\end{proof}

\subsection{The Heron-Rota-Welsh conjecture} \label{sec:matroid-volume}

Let $\cL$ be the lattice of flats of a matroid $M$ of rank $r = d+1$ on a ground set $[n]$, let $T_k$ be the set of rank-$k$ flats in $\cL$ for all $1 \leq k \leq d$, and let $X$ be the simplicial complex with facets given by the maximal chains of proper flats in $\cL$. Note that $\hat{0}$ is the set of loops of $M$ and $\hat{1} = [n]$.

We now define $\phi: X(1) \to \R_{>0}$ by setting $\phi(F)$ to be equal to the cardinality of the flat $F$. (Note that this is a strictly order-preserving map.) This implies
\[
    \alpha_K^L(F) = \frac{|F \setminus K|}{|L \setminus K|} \qquad \text{and} \qquad \beta_K^L(F) = \frac{|L \setminus F|}{|L \setminus K|},
\]
and the $\bm\pi$ maps become precisely those defined in \cite{BL21}. Letting $\mu$ be the uniform distribution over all flag of flats of $M$, by \cref{lem:toplinkexpansionmatroid} the colored path complex $(X,\mu,\phi)$ is a top-link expander. 
Then,% \cref{def:polys-X}, 
the polynomial $p_\varnothing$ is then the matroid volume polynomial of $M$ up to scalar and by \cref{thm:conn+quad=Lor} it is $\cC_\varnothing$-Lorentzian. Next, we show that \cref{thm:log-concavity-colored-complex} implies the Heron-Rota-Welsh conjecture. 

%We show how one can use our machinery to prove the 
Heron-Rota-Welsh conjecture says that the characteristic polynomial of a matroid has log-concave coefficients. This conjecture was originally proven in \cite{AHK18}, and the proof we present here is a rephrasing of the proof found in \cite{BL21}. First, we recall the definition of the M\"obius function.

\begin{definition}[M\"obius function]
    Given a lattice $\cL$, the \textbf{M\"obius function} of $\cL$ is a function $\tilde\mu:\cL \times \cL \to \Z$ defined via $\tilde\mu(F,F) = 1$ and inductively via
    \[
        \tilde\mu(K,L) = -\sum_{F: K \leq F < L} \tilde\mu(K,F).
    \]
\end{definition}

\begin{theorem}[Heron-Rota-Welsh \cite{Wel76,AHK18}] \label{thm:Heron-Rota-Welsh}
    For all non-loops $i \in [n]$, the coefficients of the reduced characteristic polynomial of $M$,
    \[
        \overline{\chi}_M(x) = \sum_{F: i \not\in F} \tilde\mu(\hat{0},F) \cdot x^{\crk(F)-1},
    \]
    form a log-concave sequence. (Note that this implies the coefficients of the characteristic polynomial of $M$ form a log-concave sequence.)
\end{theorem}

% \begin{lemma}
%     The reduced characteristic polynomial of the lattice of flats of a matroid on ground set $[n]$ with no loops is given by
%     \[
%         \overline{\chi}_M(x) = \sum
%     \]
% \end{lemma}
% \begin{proof}
%     Since the definition given in above is independent of $i \in [n]$, we have
%     \[
%         \overline{\chi}_M(x) = \frac{1}{n} \sum_{i =1}^n \sum_{F: i \not\in F} \mu(\varnothing,F) \cdot x^{\crk(F)-1} = \sum_F \frac{|F|}{n} \cdot \mu(\varnothing,F) \cdot x^{\crk(F)-1}
%     \]
% \end{proof}

To prove the Heron-Rota-Welsh conjecture, we will utilize a standard result on the lattices of flats of matroids.

\begin{lemma}[Weisner's theorem] \label{lem:Weisner}
    Given $K \prec F < L$ in $X(1)$, we have that
    \[
        \tilde\mu(K,L) = -\sum_{H: \left\{\substack{K < H \prec L \\ F \not< H}\right\}} \tilde\mu(K,H).
    \]
\end{lemma}

\begin{corollary} \label{cor:Weisner-cor}
    Given $K < L$ in $X(1)$, we have that
    \[
        \tilde\mu(K,L) = -\sum_{G: K < G \prec L} \frac{|L|-|G|}{|L|-|K|} \cdot \tilde\mu(K,G).
    \]
\end{corollary}
\begin{proof}
    Using Weisner's theorem for all $i \in L \setminus K$, we have
    \[
    % \begin{split}
    %     \tilde\mu(K,L) &= -\frac{1}{|L \setminus K|} \sum_{i \in L \setminus K} \sum_{\substack{G: K < G \prec L \\ i \not\in G}} \tilde\mu(K,G) \\
    %         &= -\sum_{G: K < G \prec L} \frac{|L \setminus G|}{|L \setminus K|} \cdot \tilde\mu(K,G) \\
    %         &= -\sum_{G: K < G \prec L} \beta_K^L(G) \cdot \tilde\mu(K,G).
    % \end{split}
        \tilde\mu(K,L) = -\frac{1}{|L \setminus K|} \sum_{i \in L \setminus K} \sum_{\substack{G: K < G \prec L \\ i \not\in G}} \tilde\mu(K,G) = -\sum_{G: K < G \prec L} \frac{|L|-|G|}{|L|-|K|} \cdot \tilde\mu(K,G).
    \]
\end{proof}

Note that these previous results imply the M\"obius function alternates in sign by induction.

Now we utilize the previous results to prove \cref{thm:Heron-Rota-Welsh}. To do this, compute the numbers $c_k$ from \cref{thm:log-concavity-colored-complex}. Let $\hat{0}$ denote the set of loops of $M$, and set $n_0 = |\hat{0}|$. We have
\[
\begin{split}
    c_k &= \frac{1}{n-n_0} \sum_{\tau \in X(d)} 1 \cdot \frac{\prod_{i=0}^d |\tau_{i+1} \setminus \tau_i|}{\prod_{i=1}^k (|\tau_i| - n_0) \prod_{i=k+1}^d (n - |\tau_i|)} \\
        % &= \sum_{\tau \in X(d)} \frac{|\tau_{k+1}|-|\tau_k|}{n-n_0} \prod_{i=0}^{k-1} \frac{|\tau_{i+1}|-|\tau_i|}{|\tau_{i+1}|-n_0} \prod_{i=k+1}^d \frac{|\tau_{i+1}|-|\tau_i|}{n-|\tau_i|} \\
        &= \sum_{\tau \in X(d)} \frac{n-|\tau_k|}{n-n_0} \prod_{i=0}^{k-1} \frac{|\tau_{i+1}|-|\tau_i|}{|\tau_{i+1}|-n_0} \prod_{i=k}^d \frac{|\tau_{i+1}|-|\tau_i|}{n-|\tau_i|} \\
        &= \sum_{\tau_k \in T_k} \frac{n-|\tau_k|}{n-n_0} \sum_{\substack{\tau_1,\ldots,\tau_{k-1} \in X(1) \\ \tau_1 \prec \cdots \tau_{k-1} \prec \tau_k}} \prod_{i=0}^{k-1} \frac{|\tau_{i+1}|-|\tau_i|}{|\tau_{i+1}|-n_0} \sum_{\substack{\tau_{k+1},\ldots,\tau_d \in X(1) \\ \tau_k \prec \tau_{k+1} \prec \cdots \prec \tau_d}} \prod_{i=k+1}^d \frac{|\tau_{i+1}|-|\tau_i|}{n-|\tau_i|}.
\end{split}
\]
Now for any fixed $\tau_k \in T_k$ we compute
\[
    \sum_{\substack{\tau_{k+1},\ldots,\tau_d \in X(1) \\ \tau_k \prec \tau_{k+1} \prec \cdots \prec \tau_d}} \prod_{i=k}^d \frac{|\tau_{i+1}|-|\tau_i|}{n-|\tau_i|} = \sum_{\tau_{k+1} \succ \tau_k} \frac{|\tau_{k+1}|-|\tau_k|}{n-|\tau_k|} \cdots \sum_{\tau_d \succ \tau_{d-1}} \frac{|\tau_d|-|\tau_{d-1}|}{n-|\tau_{d-1}|} = 1,
\]
by successively applying the matroid partition property. And for any fixed $\tau_k \in T_k$ we compute
\[
    \sum_{\substack{\tau_1,\ldots,\tau_{k-1} \in X(1) \\ \tau_1 \prec \cdots \tau_{k-1} \prec \tau_k}} \prod_{i=0}^{k-1} \frac{|\tau_{i+1}|-|\tau_i|}{|\tau_{i+1}|-n_0} = \sum_{\tau_{k-1} \prec \tau_k} \frac{|\tau_k|-|\tau_{k-1}|}{|\tau_k|-n_0} \cdots \sum_{\tau_1 \prec \tau_2} \frac{|\tau_2|-|\tau_1|}{|\tau_2|-n_0} = |\tilde\mu(\hat{0},\tau_k)|,
\]
by successively applying \cref{cor:Weisner-cor} (since the M\"obius function alternates in sign). Thus we have
\[
    c_k = \sum_{\tau_k \in T_k} \frac{n-|\tau_k|}{n-n_0} \cdot |\tilde\mu(\hat{0},\tau_k)|
\]
for all $0 \leq k \leq d$, where $\tau_0 = \hat{0}$.

\begin{proof}[Proof of \cref{thm:Heron-Rota-Welsh}]
    Let $\gamma_k$ denote the absolute value of the $x^k$ coefficient of $\overline\chi_M(x)$. For any non-loop $i \in [n]$, the fact that the M\"obius function alternates in sign implies
    \[
        \gamma_{d-k} = \sum_{\tau_k \in T_k: i \not\in \tau_k} |\tilde\mu(\hat{0},\tau_k)|,
    \]
    and thus
    \[
        \gamma_{d-k} = \frac{1}{n-n_0} \sum_{i \in \hat{1} \setminus \hat{0}} \sum_{\tau_k \in T_k: i \not\in \tau_k} |\tilde\mu(\hat{0},\tau_k)| = \sum_{\tau_k \in T_k} \frac{n-|\tau_k|}{n-n_0} \cdot |\tilde\mu(\hat{0},\tau_k)| = c_k.
    \]
    Therefore $(\gamma_0,\gamma_1,\ldots,\gamma_d)$ forms a log-concave sequence by \cref{thm:log-concavity-colored-complex}.
\end{proof}

\subsection{Colorings of distributive lattices}

Let $P$ be a finite Poset with associated distributive lattice $\cL$ and path complex $X$. Given $F \in X(1)$, we will consider $F$ to be a set of elements of the ground set of $P$. Let $\phi: X(1) \to \R_{>0}$ be any strictly order-preserving map, and let $L(P)$ denote the set of all linear extensions of $P$. We will think of linear extensions of $P$ as functions $\ell: P \to [d+1]$, and we will denote $\ell_k := \ell^{-1}(k) \in P$.

With this, we define $\bm\alpha$, $\bm\beta$ as in \cref{cor:alpha-beta-v1}, and we let the probability distribution $\mu$ be uniform on the facets of $X$.

\subsubsection{Stanley's inequality}

% Any weight function $\phi: P \to \R_{>0}$ gives rise to a strictly order preserving map $\phi: X(1) \to \R_{>0}$ via:
% \[
%     \phi(F) = \sum_{i \in F} \phi(i).
% \]
% We note that the conditions of  hold in general in this case.

Given some $a \in P$, we define
\[
    \phi_M(i) = \begin{cases}
        M, & i = a \\
        1, & i \neq a
    \end{cases}
\]
for some large constant $M > 1$, and we let $\phi_M(F) = \sum_{i \in F} \phi_M(i)$. We now compute $c_k$ for $0 \leq k \leq d$ as defined in \cref{thm:log-concavity-colored-complex}. As $M \to \infty$, we have
\[
\begin{split}
    c_k &= \frac{1}{\phi_M(\hat{1}) - \phi_M(\hat{0})} \sum_{\tau \in X(d)} \mu(\tau) \cdot \frac{\prod_{i=0}^d (\phi_M(\tau_{i+1}) - \phi_M(\tau_i))}{\prod_{i=1}^k (\phi_M(\tau_i) - \phi_M(\tau_0)) \prod_{i=k+1}^d (\phi_M(\tau_{d+1}) - \phi_M(\tau_i))} \\
        &= \frac{1}{M + d} \sum_{\tau \in X(d)} \frac{M}{\prod_{i=1}^k (\phi_M(\tau_i) - \phi_M(\tau_0)) \prod_{i=k+1}^d (\phi_M(\tau_{d+1}) - \phi_M(\tau_i))} \\
        &\underset{M \to \infty}{\to} \sum_{\substack{\tau \in X(d) \\ \tau_{k+1} \setminus \tau_k = \{a\}}} \frac{1}{k!(d-k)!} \\
        &= \frac{|\{\tau \in X(d): \tau_{k+1} \setminus \tau_k = \{a\}\}|}{k!(d-k)!} \\
        &= \frac{|\{\ell \in L(P): \ell_{k+1} = a\}|}{k!(d-k)!}.
\end{split}
\]
Thus by applying \cref{lem:toplinkexpansioncolorlattice} and \cref{thm:log-concavity-colored-complex} we obtain a weak version of Stanley's inequality, which says that $(c_0,c_1,\ldots,c_d)$ forms a log-concave sequence.

To obtain the full Stanley's inequality, we now append chains of length $N$ below and above all elements of the Poset $P$ to obtain $P'$ and apply the above log-concavity statement. This implies $(c_N',c_{N+1}',\ldots,c_{N+d}')$ forms a log-concave sequence, where
\[
    c_{N+k}' = \frac{|\{\ell \in L(P') : \ell_{N+k+1} = a\}|}{(N+k)!(N+d-k)!} = \frac{|\{\ell \in L(P) : \ell_{k+1} = a\}|}{(N+k)!(N+d-k)!}.
\]
By multiplying by $N!(N+d)!$ and limiting $N \to \infty$, we obtain Stanley's inequality for the original Poset $P$.

\subsubsection{Extended Stanley inequality}

In this section we prove \cref{thm:P-consistent-ext-Stanley}. Fix any $P$-consistent set $A$.

We now define $\phi$ via
\[
    \phi_M(F) = \I[F \cap A \neq \varnothing] \cdot M + |F|.
\]
Given a linear extension $\ell \in L(P)$, let $\ell_{\min}(A)$ denote the smallest $i$ for which $\ell_i \in A$. We now compute the numbers $c_k$ from \cref{thm:log-concavity-colored-complex} via
\[
\begin{split}
    c_k &= \frac{1}{M+d+1} \sum_{\tau \in X(d)} 1 \cdot \frac{M+1}{\prod_{i=1}^k\phi_M(\tau_i) \cdot \prod_{i=k+1}^d(M+d+1-\phi_M(\tau_i))} \\
        &\underset{M \to \infty}{\to} \sum_{\substack{\ell \in L(P) \\ \ell_{\min}(A) = k+1}} \frac{1}{k!(d-k)!} \\
        &= \frac{|\{\ell \in L(P): \ell_{\min}(A) = k+1\}|}{k!(d-k)!}.
\end{split}
\]
Thus by applying \cref{lem:toplinkexpansioncolorlattice} and \cref{thm:log-concavity-colored-complex} we obtain a weak version of the extended Stanley inequality, which says that $(c_0,c_1,\ldots,c_d)$ forms a log-concave sequence.

% To obtain the full extended Stanley inequality, we now append chains of length $N$ below and above all elements of the Poset $P$ to obtain $P'$ and apply the above log-concavity statement. This implies $(c_N',c_{N+1}',\ldots,c_{N+d}')$ forms a log-concave sequence, where
% \[
%     c_{N+k}' = \frac{|\{\ell \in L(P') : \ell_{\min}(A) = N+k+1\}|}{(N+k)!(N+d-k)!} = \frac{|\{\ell \in L(P) : \ell_{\min}(A) = k+1\}|}{(N+k)!(N+d-k)!}.
% \]
% By multiplying by $N!(N+d)!$ and limiting $N \to \infty$,

Using the same argument as above, we finally obtain the extended Stanley inequality for the Poset $P$ and the given $P$-consistent $A$.

\subsection{Generalizing the $\alpha$ and $\beta$ vectors}

In this section, we generalize the $\alpha$ and $\beta$ vectors utilized above to what we call $\ell$ vectors. Here we will define $d$ such $\ell$ vectors in total, one vector associated to $T_i$ for all $i$. We will then analyze the $1$-skeleton of $X$ restricted to $T_i \cup T_j$ for any $i \neq j$ by applying directional derivatives with respect to all other $\ell$ vectors (i.e. for $k \not\in \{i,j\}$).

We will assume throughout this section that we have a strictly order-preserving function $\psi: X(1) \to \R$ and a function $\phi: \{(a,b) \in \R^2 : a < b\} \to \R_{>0}$. We will assume for $k < m < f < l$, that $\phi$ satisfies the condition
\begin{equation} \label{eq:phi2-condition}
    \phi(k,f) \cdot \phi(m,l) - \phi(k,l) \cdot \phi(m,f) = \phi(k,m) \cdot \phi(f,l)
\end{equation}
from \cref{lem:alpha-beta-gen}. Given any $F<G$, we define $\phi(F,G) = \phi(\psi(F),\psi(G))$, and we will make use of the notation $\phi(\psi(F),\psi(G)) = \phi(F,G) = \phi(F,\psi(G)) = \phi(\psi(F),G)$ and $\phi(F,m) = \phi(\psi(F),m)$ and $\phi(m,G) = \phi(m,\psi(G))$.

We now define the $\ell$ vectors. Given any $K,F,L$ with $k = \psi(K), f = \psi(F), l = \psi(L)$ such that $K < L$, and any $m$ such that $k < m < l$, we define
\[
    \ell_K^L(m,F) = \begin{cases}
        \frac{\phi(K,F)}{\phi(K,L)} \cdot \phi(m,L)  & \text{if } F \in (K,L) \text{ and } f \leq m \\
        \frac{\phi(F,L)}{\phi(K,L)} \cdot \phi(K,m) & \text{if } F \in (K,L) \text{ and } f \geq m \\
        0 & \text{if } F \not\in (K,L).
    \end{cases}
\]
We will also let $\ell_K^L(m)$ denote the vector indexed by $X(1)$, which is 0 outside of $(K,L)$.
Given $K<L$, we note that $\ell_K^L(\psi(K)+\epsilon) = \beta_K^L$ and $\ell_K^L(\psi(L)-\epsilon) = \alpha_K^L$ up to positive scalars (for some small $\epsilon > 0$), so that the $\ell$ vectors can be seen as a generalization of $\alpha$ and $\beta$.
%\textcolor{red}{So you assume $\ell_K^L(m)$ is a function. This could be confusing.}

% Given any $K,L$ such that $K < L$ and any $c$ such that $\phi(K) < c < \phi(L)$, we define
% \[
%     \ell_K^L(c,F) = \begin{cases}
%         \frac{\phi(F)-\phi(K)}{\phi(L)-\phi(K)} \cdot (\phi(L) - c)  & \text{if } F \in (K,L) \text{ and } \phi(F) \leq c \\
%         \frac{\phi(L)-\phi(F)}{\phi(L)-\phi(K)} \cdot (c-\phi(K)) & \text{if } F \in (K,L) \text{ and } \phi(F) \geq c \\
%         0 & \text{if } F \not\in (K,L).
%     \end{cases}
% \]

\begin{lemma} \label{lem:ell-relations}
    Let $\sigma$ and $K,L \in \sigma$ be such that $F \not\in \sigma$ for all $F \in (K,L)$. Define $k = \psi(K),l=\psi(L)$, and let $m$ be such that $k < m < l$. Then for any $F \in (K,L)$, setting $f = \psi(F)$, we have
    \[
        \pi_{\sigma+F}(\ell_K^L(m))=\begin{cases}\ell_K^F(m)&\text{if } f > m\\ \ell_F^L(m)&\text{if } f < m \\ 0 & f = m.\end{cases}
    \]
\end{lemma}
\begin{proof}
    Fix $H \in (K,L)$ and set $h = \psi(H)$. First we assume $f > m$, so that $k<m<f<l$ in this case. Let $H \in (K,F)$. If $h \leq m$ then
    \begin{align*}
        \pi_{\sigma+F}(\ell_K^L(m))_H &= \ell_K^L(m,H) - \frac{\phi(k,h)}{\phi(k,f)} \cdot \ell_K^L(m,F) \\
            &= \frac{\phi(k,h) \cdot \phi(m,l)}{\phi(k,l)} - \frac{\phi(k,h)}{\phi(k,f)} \cdot \frac{\phi(f,l) \cdot \phi(k,m)}{\phi(k,l)} \\
            &= \frac{\phi(k,h)}{\phi(k,f) \cdot \phi(k,l)} (\phi(m,l) \cdot \phi(k,f) - \phi(f,l) \cdot \phi(k,m)) \\
            &\overset{\text{\eqref{eq:phi2-condition}}}{=} \frac{\phi(k,h)}{\phi(k,f) \cdot \phi(k,l)} \cdot \phi(k,l) \cdot \phi(m,f) \\
            &= \ell_K^F(m,H).
    \end{align*}
    If $h > m$ so that $k < h < f < l$, then
    \begin{align*}
        \pi_{\sigma+F}(\ell_K^L(m))_H &= \ell_K^L(m,H) - \frac{\phi(k,h)}{\phi(k,f)} \cdot \ell_K^L(m,F) \\
            &= \frac{\phi(k,m) \cdot \phi(h,l)}{\phi(k,l)} - \frac{\phi(k,h)}{\phi(k,f)} \cdot \frac{\phi(k,m) \cdot \phi(f,l)}{\phi(k,l)} \\
            &= \frac{\phi(k,m)}{\phi(k,f) \cdot \phi(k,l)} (\phi(h,l) \cdot \phi(k,f) - \phi(k,h) \cdot \phi(f,l)) \\
            &\overset{\text{\eqref{eq:phi2-condition}}}{=}  \frac{\phi(k,m)}{\phi(k,f) \cdot \phi(k,l)} \cdot \phi(k,l) \cdot \phi(h,f) \\
            &= \ell_K^F(m,H).
    \end{align*}
    Finally, if $H \in (F,L)$ then
    \begin{align*}
        \pi_{\sigma+F}(\ell_K^L(m))_H &= \ell_K^L(m,H) - \frac{\phi(h,l)}{\phi(f,l)} \cdot \ell_K^L(m,F) \\
            &= \frac{\phi(k,m) \cdot \phi(h,l)}{\phi(k,l)} - \frac{\phi(h,l)}{\phi(f,l)} \cdot \frac{\phi(k,m) \cdot \phi(f,l)}{\phi(k,l)} \\
            &= 0.
    \end{align*}
    The proof for the case of $f < m$ is similar. Next we assume $f = m$. If $H \in (K,F)$ then
    \begin{align*}
        \pi_{\sigma+F}(\ell_K^L(m))_H &= \ell_K^L(m,H) - \frac{\phi(k,h)}{\phi(k,f)} \cdot \ell_K^L(m,F) \\
            &= \frac{\phi(k,h) \cdot \phi(m,l)}{\phi(k,l)} - \frac{\phi(k,h)}{\phi(k,f)} \cdot \frac{\phi(k,f) \cdot \phi(m,l)}{\phi(k,l)} \\
            &= 0.
    \end{align*}
    If $H \in (F,L)$ then
    \begin{align*}
        \pi_{\sigma+F}(\ell_K^L(m))_H &= \ell_K^L(m,H) - \frac{\phi(h,l)}{\phi(f,l)} \cdot \ell_K^L(m,F) \\
            &= \frac{\phi(k,m) \cdot \phi(h,l)}{\phi(k,l)} - \frac{\phi(h,l)}{\phi(f,l)} \cdot \frac{\phi(k,m) \cdot \phi(f,l)}{\phi(k,l)} \\
            &= 0.
    \end{align*}
\end{proof}

\begin{corollary} \label{cor:ell-in-cone}
    If $\cC_\varnothing$ is non-empty, then the vector $\ell(m)$ is an element of $\overline{\cC}_\varnothing$ for all $m$.
\end{corollary}
\begin{proof}
    By inductively applying \cref{lem:ell-relations}, we have that $\pi_\sigma(\ell(m)) \in \R^{X_\sigma(1)}_{\geq 0}$ for all $m$ and $\sigma \in X$. Thus $\ell(m)$ is $\bm\pi$-non-negative. Therefore $\ell(m) \in \overline{\cC}_\varnothing$ for all $m$ by (3) of \cref{prop:cones}.
\end{proof}

\begin{lemma}[General $\ell$ mixed derivatives lemma] \label{lem:general-ell-mixed-derivs-expression}
    Let $(X,\mu)$ be a $d$-path complex, and let $\psi(F) = f$ whenever $F \in T_f$. Let $\sigma = \{F_{a_1} < F_{a_2} < \cdots < F_{a_m}\}$ where $F_{a_j} \in T_{a_j}$ for all $j$, and define $F_{a_0} = \hat{0}$ and $F_{a_{m+1}} = \hat{1}$. Let $I = [d] \setminus \{a_1,\ldots,a_m\}$, and define $\ell(i) = \ell_{\hat{0}}^{\hat{1}}(i)$ for all $i \in I$. Then we have
    \[
        \left[\prod_{i \in I} \nabla_{\ell(i)}\right] p_\sigma(\pi_\sigma(\bm{t})) = C_\phi(a_1,a_2,\ldots,a_d) \sum_{\tau \in X(d): \sigma \subseteq \tau} \mu(\tau),
    \]
    for some $C_\phi(a_1,a_2,\dots,a_d) > 0$.
\end{lemma}
\begin{proof}
    We first prove by induction on $I$ that
    % \[
    %     \left[\prod_{i \in I} \nabla_{\ell(i)}\right] p_\sigma(\pi_\sigma(\bm{t})) = \sum_{\tau \in X(d): \sigma \subseteq \tau} \left[\prod_{i \in I} \ell_{K_i}^{\tau_{i+1}}(i,\tau_i)\right] \cdot \mu(\tau),
    % \]
    \begin{equation}\label{eq:productofellis}
        \left[\prod_{i \in I} \nabla_{\ell(i)}\right] p_\sigma(\pi_\sigma(\bm{t})) = \sum_{\tau \in X(d): \sigma \subseteq \tau} \left[\prod_{j=0}^m \prod_{i=a_j+1}^{a_{j+1}-1} \ell_{F_{a_j}}^{\tau_{i+1}}(i,\tau_i)\right] \cdot \mu(\tau).
    \end{equation}
    Note that the statement is trivial whenever $I = \varnothing$. Otherwise, let $j \in \{0,1,\ldots,m\}$ be such that $F_{a_j}$ is not covered by $F_{a_{j+1}}$, i.e., $a_{j+1}>a_j+1$. Further, fix any $i$ such that $a_j < i < a_{j+1}$, and define $K = F_{a_j}$ and $L = F_{a_{j+1}}$. We thus have
    \[
    \begin{split}
        \nabla_{\ell(i)} [p_\sigma(\pi_\sigma(\bm{t}))] &\overset{\text{\cref{lem:chain-rule}}}{=} [\nabla_{\pi_\sigma(\ell(i))} p_\sigma](\pi_\sigma(\bm{t})) \\
            &\overset{\text{\cref{lem:ell-relations}}}{=} [\nabla_{\ell_K^L(i)} p_\sigma](\pi_\sigma(\bm{t})) \\
            &\overset{\text{\cref{lem:derivs-expression}}}{=} \left[\sum_{H: K<H<L} \ell_K^L(i,H) \cdot p_{\sigma \cup \{H\}}(\pi_{\sigma + H}(\cdot))\right](\pi_\sigma(\bm{t})) \\
            &\overset{\text{commutativity $\bm\pi$}}{=} \sum_{H: K<H<L} \ell_K^L(i,H) \cdot p_{\sigma \cup \{H\}}(\pi_{\sigma \cup \{H\}}(\bm{t}))
    \end{split}
    \]
    Note that for $h \in \{a_j+1,a_j+2,\ldots,a_{j+1}-1\}$ and $H \in T_h$, we have $\pi_{\sigma \cup \{H\}}(\ell(h)) = 0$ by \cref{lem:ell-relations} (by our assumption on $\psi$), which implies
    \begin{equation} \label{eq:Hh-zero}
        \nabla_{\ell(h)} [p_{\sigma \cup \{H\}}(\pi_{\sigma \cup \{H\}}(\bm{t}))] \overset{\text{\cref{lem:chain-rule}}}{=} [\nabla_{\pi_{\sigma \cup \{H\}}(\ell(h))} p_{\sigma \cup \{H\}}](\pi_{\sigma \cup \{H\}}(\bm{t})) = 0.
    \end{equation}
    This in particular implies the following
    $$ \left(\prod_{h\neq i:a_j<h<a_{j+1}}\nabla_{\ell(h)}\right)\nabla_{\ell(i)} [p_\sigma(\pi_\sigma(\bm{t}))] =\left(\prod_{h\neq i:a_j<h<a_{j+1}}\nabla_{\ell(h)}\right)\sum_{H\in T_i: K<H<L} \ell_K^L(i,H) \cdot p_{\sigma \cup \{H\}}(\pi_{\sigma \cup \{H\}}(\bm{t}))
    $$
    %To prove the desired expression we now want to apply the remainder of the derivatives with respect to $\ell(k)$ for $k \neq i$. This is a priori complicated, but we can simplify things significantly by using \cref{eq:Hh-zero} above. 
    %Therefore by induction we have
    We now prove the desired expression by applying the derivatives in reverse order from $\nabla_{\ell(a_{j+1}-1)}$ to $\nabla_{\ell(a_j+1)}$.
    \begin{align*}
      \left[\prod_{i=a_j+1}^{a_{j+1}-1} \nabla_{\ell(i)}\right] p_\sigma(\pi_\sigma(\bm{t})) &= \sum_{H: K<H<L} \ell_K^L(a_{j+1}-1,H) \cdot \left[\prod_{i=a_j+1}^{a_{j+1}-2} \nabla_{\ell(i)}\right] p_{\sigma \cup \{H\}}(\pi_{\sigma \cup \{H\}}(\bm{t})) \\
            &\overset{\text{\eqref{eq:Hh-zero}}}{=} \sum_{H: K<H \prec L} \ell_K^L(a_{j+1}-1,H) \cdot \left[\prod_{i=a_j+1}^{a_{j+1}-2} \nabla_{\ell(i)}\right] p_{\sigma \cup \{H\}}(\pi_{\sigma \cup \{H\}}(\bm{t})) \\
            &\overset{\text{IH}}{=} \sum_{H: K<H \prec L} \ell_K^L(a_{j+1}-1,H) \cdot \sum_{\tau \in X(d): \sigma \cup \{H\} \subseteq \tau} \left[\prod_{i=a_j+1}^{a_{j+1}-2} \ell_K^{\tau_{i+1}}(i,\tau_i)\right] \cdot \mu(\tau) \\
            &= \sum_{\tau \in X(d): \sigma \subseteq \tau} \left[\prod_{i=a_j+1}^{a_{j+1}-1} \ell_K^{\tau_{i+1}}(i,\tau_i)\right] \cdot \mu(\tau),
    \end{align*}
    as desired. Running the same argument for all $0\leq j\leq m$ we obtain \eqref{eq:productofellis}. 

    We will next prove that
    \[
        \prod_{j=0}^m \prod_{i=a_j+1}^{a_{j+1}-1} \ell_{F_{a_j}}^{\tau_{i+1}}(i,\tau_i) = \prod_{j=0}^m \frac{\prod_{i=a_j}^{a_{j+1}-1} \phi(i,i+1)}{\phi(a_j,a_{j+1}} =: C_\phi(a_1,a_2,\ldots,a_d).
    \]
    To do this, we will show that for all $j \in \{0,1,\ldots,m\}$ and all $\tau \in X(d)$ we have
    \[
        \prod_{i=a_j+1}^{a_{j+1}-1} \ell_{F_{a_j}}^{\tau_{i+1}}(i,\tau_i) = \frac{\prod_{i=a_j}^{a_{j+1}-1} \phi(i,i+1)}{\phi(a_j,a_{j+1})}.
    \]
    To see this, note that both sides are equal to $1$ whenever $a_{j+1} = a_j+1$. For the general case, by induction on $a_{j+1}-a_j$ we have
    \begin{align*}
        \prod_{i=a_j+1}^{a_{j+1}-1} \ell_{F_{a_j}}^{\tau_{i+1}}(i,\tau_i) &= \frac{\phi(a_j,a_{j+1}-1) \cdot \phi(a_{j+1}-1,a_{j+1})}{\phi(a_j,a_{j+1})} \prod_{i=a_j+1}^{a_{j+1}-2} \ell_{F_{a_j}}^{\tau_{i+1}}(i,\tau_i) \\
            &\overset{\text{IH}}{=} \frac{\phi(a_j,a_{j+1}-1) \cdot \phi(a_{j+1}-1,a_{j+1})}{\phi(a_j,a_{j+1})} \cdot \frac{\prod_{i=a_j}^{a_{j+1}-2} \phi(i,i+1)}{\phi(a_j,a_{j+1}-1)} \\
            &= \frac{\prod_{i=a_j}^{a_{j+1}-1} \phi(i,i+1)}{\phi(a_j,a_{j+1})},
    \end{align*}
    as desired.

    Combining the two identities proven above then completes the proof of the result.
\end{proof}

\begin{definition} \label{def:mu-ell}
    We now define a function $\mu_\ell: X \to \R_{\geq 0}$ via
    \[
        \mu_\ell(\sigma) := C_\phi(i,j) \cdot \mu(\sigma) = C_\phi(i,j) \sum_{\tau \in X(d) : \sigma \subseteq \tau} \mu(\tau),
    \]
    where $C_\phi$ is the constant of \cref{lem:general-ell-mixed-derivs-expression}.
\end{definition}

\begin{corollary}[Eigenvalue bound for the $\ell$ Hessians] \label{cor:eig-bound-ell}
    Let $(X,\mu)$ be a $d$-path complex, and let $\psi(F) = f$ whenever $F \in T_f$. Given $1 \leq i < j \leq d-1$, let $B_{i,j}$ be the bipartite graph on $T_i \cup T_j$ with edge weights given by $\mu_\ell(\{F,G\})$ for $F \in T_i,G \in T_j$ and $F < G$, let $A_{i,j}$ be the weighted adjacency matrix of $B_{i,j}$, and let $D_{i,j}$ be the diagonal weighted degree matrix of $B_{i,j}$. Further, let $M_{i,j}$ be the diagonal matrix with entries given by
    \[
        m_{i,j}(F,F) = \frac{\sum_{G \in T_j:F < G} \beta_F(G) \cdot \mu_\ell(\{F,G\})}{\sum_{G \in T_j:F < G} \mu_\ell(\{F,G\})} \qquad \text{for} \quad F \in T_i
    \]
    and
    \[
        m_{i,j}(G,G) = \frac{\sum_{F \in T_i:F < G} \alpha^G(F) \cdot \mu_\ell(\{F,G\})}{\sum_{F \in T_i:F < G} \mu_\ell(\{F,G\})} \qquad \text{for} \quad G \in T_j.
    \]
    % \[
    %     m_k(H,H) = \begin{cases}
    %         \sum_{G:H \prec G} \beta_H(G) \cdot f(\{H,G\}), & H \in T_k \\
    %         \sum_{F:F \prec H} \alpha^H(F) \cdot f(\{F,H\}), & H \in T_{k+1}.
    %     \end{cases}
    % \]
    If $\cC_\varnothing$ is non-empty and $p_\varnothing$ is $\cC_\varnothing$-Lorentzian, then
    \[
        D_{i,j}^{-1/2} A_{i,j} D_{i,j}^{-1/2} \preceq M_{i,j} + \bm{w}\bm{w}^\top
    \]
    for some real vector $\bm{w}$.
\end{corollary}
\begin{proof}
    Note first that for $\{F,G\} \in X(2)$ with $F < G$, by \cref{lem:derivs-expression} we have
    \[
        \partial_{t_F} \partial_{t_G} p_\varnothing(\bm{t}) = p_{\{F,G\}}(\pi_{\{F,G\}}(\bm{t})).
    \]
    By \cref{lem:general-ell-mixed-derivs-expression} and the equality case of \cref{lem:ell-relations}, we then have
    \[
        \left[\prod_{k \not\in \{i,j\}} \nabla_{\ell(k)}\right] \partial_{t_F} \partial_{t_G} p_\varnothing(\bm{t}) = \begin{cases}
            \mu_\ell(\{F,G\}) > 0, & F \in T_i, G \in T_j \\
            0, & \text{otherwise}.
        \end{cases}
    \]
    Next for any $H \in X(1)$ we compute
    \[
    \begin{split}
        \partial_{t_H}^2 p_\varnothing(\bm{t}) &\underset{\text{\cref{lem:derivs-expression}}}{=} \partial_{t_H} p_{\{H\}}(\pi_{\{H\}}(\bm{t})) \\
            &\underset{\text{\cref{lem:chain-rule}}}{=} \left[\nabla_{\pi_{\{H\}}(\bm{1}_H)} p_{\{H\}}\right](\pi_{\{H\}}(\bm{t})) \\
            &\underset{\text{\cref{eq:pi-def-from-alpha-beta}}}{=} -\sum_{F: F < H} \alpha^H(F) \cdot p_{\{F,H\}}(\pi_{\{F,H\}}(\bm{t})) - \sum_{G: H < G} \beta_H(G) \cdot p_{\{H,G\}}(\pi_{\{H,G\}}(\bm{t})),
    \end{split}
    \]
    and by \cref{lem:general-ell-mixed-derivs-expression} and the equality case of \cref{lem:ell-relations} we have
    \[
        \left[\prod_{k \not\in \{i,j\}} \nabla_{\ell(k)}\right] \partial_{t_H}^2 p_\varnothing(\bm{t}) = \begin{cases}
            -\sum_{G \in T_j: H < G} \beta_H(F) \cdot \mu_\ell(\{H,G\}), & H \in T_i \\
            -\sum_{F \in T_i: F < H} \alpha^H(F) \cdot \mu_\ell(\{F,H\}), & H \in T_j \\
            0, & \text{otherwise}
        \end{cases}.
    \]
    Thus the Hessian $\cH$ of $\left[\prod_{k \not\in \{i,j\}} \nabla_{\ell(k)}\right] p_\varnothing(\bm{t})$ is supported only on $T_i \cup T_j$, and we further have $\cH = A_{i,j} - D_{i,j}^{1/2} M_{i,j} D_{i,j}^{1/2}$. Since $p_\varnothing$ is $\cC_\varnothing$-Lorentzian and $\ell(k) \in \overline{\cC}_\varnothing$ for all $k$ by \cref{cor:ell-in-cone}, we thus have
    \[
        A_{i,j} - D_{i,j}^{1/2} M_{i,j} D_{i,j}^{1/2} = \cH \preceq \bm{v}\bm{v}^\top
    \]
    by \cref{lem:main-purpose-lemma} for some real vector $\bm{v}$, and this implies the desired result.
\end{proof}

%\subsection{Example 2: Beyond partite rank}
\section{Trickledown Theorems for Top-Link Path Complexes}
\label{sec:polys-v1}
%Let $\phi: X(1) \to \R_{>0}$ be a strictly order-preserving map (i.e., such that $F < G$ implies $\phi(F) < \phi(G)$). We will consider this to be a sort of \textbf{coloring} of a path complex. (The strictness of the order-preserving map implies $\phi$ gives rise to a proper coloring of the $1$-skeleton of $X$ with colors in $\R_{>0}$.) Throughout the remainder of this section, we fix such a coloring $\phi$ on $X(1)$, and we let $\bm\alpha$ and $\bm\beta$ be defined as follows.
% Throughout this section we fix a $d$-dimensional path complex $(X,\mu,\phi)$.
In this section we will prove \cref{thm:main}. Throughout, for $K \in T_i$ and $L \in T_j$ such that $K < L$ we define $r_K^L = j-i$ and a coloring $\psi(K) = i$, where possibly $K = \hat{0}$ and/or $L = \hat{1}$. We further define $\phi(i,j) = j-i$ which satisfies condition \eqref{eq:phi2-condition}.
% On the other hand, the use of the function $\phi$ allows us, for example, to capture the matroid volume polynomial in the context of this paper, as we will demonstrate in \cref{sec:matroid-volume}.
By \cref{cor:alpha-beta-v1} and \cref{lem:cone-non-empty-v1} we obtain the following:

\begin{corollary} \label{cor:alpha-beta-rank}
    If for $K < F < L$ in $X(1)$ we define
    \[
        \alpha_K^L(F) = \frac{r_K^F}{r_K^L} \qquad \text{and} \qquad \beta_K^L(F) = \frac{r_F^L}{r_K^L},
    \]
    then $\bm\alpha$ and $\bm\beta$ are commmutative and give rise to a commutative family of maps $\bm\pi$ and their associated polynomials $p_\sigma$ for $\sigma \in X$.

    Furthermore, the cone $\cC_\sigma$ is non-empty for all $\sigma \in X$.
\end{corollary}
% \begin{proof}
%     The proof is essentially the same as that of \cref{lem:alpha-beta-v1} after noting that for $K<F<G<L$ in $X(1)$ we have
%     \[
%     \begin{split}
%         (\phi(L)-\phi(F)) \cdot (\phi(G)-\phi(K)) &- (\phi(G)-\phi(F)) \cdot (\phi(L)-\phi(K)) \\
%             &= (\phi(F)-\phi(K)) \cdot (\phi(L)-\phi(G)),
%     \end{split}
%     \]
%     which is the analogous key identity needed for the last two conditions of \cref{def:alpha-beta}.
% \end{proof}

% \begin{lemma} \label{lem:cone-non-empty-v1}
%     The cone $\cC_\sigma$ is non-empty for all $\sigma \in X$.
% \end{lemma}
% \begin{proof}
%     The same proof as that of \cref{lem:cone-non-empty-v1} holds here as well, after noting that for $K<F<G<L$ in $X(1)$ we have
%     \[
%         \frac{\phi(L)-\phi(F)}{\phi(L)-\phi(K)} - \frac{\phi(L)-\phi(G)}{\phi(L)-\phi(K)} = \frac{\phi(G)-\phi(F)}{\phi(L)-\phi(K)} = \frac{\phi(G)-\phi(K)}{\phi(L)-\phi(K)} - \frac{\phi(F)-\phi(K)}{\phi(L)-\phi(K)}.
%     \]
% \end{proof}

This leads to the main result of this section, which gives eigenvalue bounds on the random walk matrix associated to $X$ based on the machinery of the previous sections.

\begin{theorem} \label{thm:eig-bounds-main-v1}
    Fix $i < j$ in $[d]$, and let $G_{i,j}$ be the bipartite graph on $T_i \cup T_j$ with corresponding random walk matrix $P_{i,j}$ according to stationary distribution $\mu$. Let $A_{i,j}$ be the weighted adjacency matrix of $G_{i,j}$, let $D_{i,j}$ be the diagonal weighted degree matrix of $G_{i,j}$, and let $M_{i,j}$ be the diagonal matrix with entries given by $\frac{d-j+1}{d-i+1}$ for vertices in $T_i$ and $\frac{i}{j}$ for vertices in $T_j$. If $p_\varnothing$ is $\cC_\varnothing$-Lorentzian, then
    \[
        D_{i,j}^{-1/2} A_{i,j} D_{i,j}^{-1/2} \preceq M_{i,j} + \bm{w}\bm{w}^\top
    \]
    for some real vector $\bm{w}$. Moreover, the second eigenvalue of $P_{i,j}$ is at most $\sqrt{\frac{i(d-j+1)}{j(d-i+1)}}$.
\end{theorem}
\begin{proof}
    By \cref{cor:eig-bound-ell} and \cref{cor:alpha-beta-rank}, we have that
    \[
        D_{i,j}^{-1/2} A_{i,j} D_{i,j}^{-1/2} \preceq M_{i,j} + \bm{w}\bm{w}^\top
    \]
    where the entries of $A_{i,j}$ are given by $\mu_\ell(\{F,G\}) = C_\phi(i,j) \cdot \mu(\{F,G\})$, the entries of $D_k$ are given by the row sums of $A_k$, and the entries of $M_k$ are given by
    \[
        m_{i,j}(F,F) = \frac{\sum_{G \in T_j:F < G} \beta_F(G) \cdot \mu_\ell(\{F,G\})}{\sum_{G \in T_j:F < G} \mu_\ell(\{F,G\})} = \frac{d-j+1}{d-i+1} \qquad \text{for} \quad F \in T_i
    \]
    and
    \[
        m_{i,j}(G,G) = \frac{\sum_{F \in T_i:F < G} \alpha^G(F) \cdot \mu_\ell(\{F,G\})}{\sum_{F \in T_i:F < G} \mu_\ell(\{F,G\})} = \frac{i}{j} \qquad \text{for} \quad G \in T_j.
    \]
    Since the entries of $A_{i,j}$ are proportional to the edge weights of $G_{i,j}$ for given $i < j$, this implies the first desired result. The eigenvalue bound for $P_{i,j}$ then immediately follows from \cref{lem:eigenvaluebipartite}.
\end{proof}

\begin{proof}[Proof of \cref{thm:main}]
We prove that $\lambda_2(P_\varnothing)\leq 1/2$.  The bound also extends to the 1-skeleton of all links (of co-dimension at least 2) of $X$. In particular, note that any link of a top-link path complex is also a top-link path complex.

Let $G_\varnothing$ be the 1-skeleton of $X_\varnothing$,
$G_{i,j}$ be the induced bipartite graph between parts $T_i,T_j$ and $P_{i,j}$ be the corresponding simple random walk.
% First, by \cref{thm:eig-bounds-main-v1}, we have that for any $k\in [d-1]$, 
% $$ \lambda_2(P_{k,k+1})\leq \sqrt{\frac{k(d-k)}{(k+1)(d-k+1)}}.$$
% Second, by \cref{lem:dpartiteeigproduct}, for any $1\leq i<j\leq d$ we have
% $$ \lambda_2(P_{i,j}) \leq \prod_{k=i}^{j-1} \sqrt{\frac{k(d-k)}{(k+1)(d-k+1)}} = \sqrt{\frac{i\cdot (d-j+1)\cdot }{j\cdot (d-i+1)}}$$
Note by \cref{thm:eig-bounds-main-v1}, for any $1\leq i<j\leq d$ we have
$$ \lambda_2(P_{i,j}) \leq \sqrt{\frac{i (d-j+1)\cdot }{j (d-i+1)}}.$$

For $i < j$, let $M_{i,j}$ be a diagonal matrix with entries given by $m_i(j) := \frac{d-j+1}{d-i+1}$ for vertices in $T_i$ and $m_j(i):= \frac{i}{j}$ for vertices in $T_j$. Note that $0\leq m_i(j),m_j(i)<1$. By \cref{lem:eigenvaluebipartite}, we get that $\lambda_2(P_{i,j} - M_{i, j}) \leq 0$ and therefore has at most one positive eigenvalue. 

Next we want to use \cref{lem:dpartiteeigenvalues} to bound $\lambda_2(P_\varnothing)$.
 For any fixed $i\in [d]$, we have 
 \begin{align*}
     \E_{j \in [d], j \neq i} m_i(j) &= \frac{1}{d-1} \left(\sum_{j=1}^{i-1} \frac{j}{i} +\sum_{j=i+1}^d \frac{d-j+1}{d-i+1} \right)\\
     &=\frac{1}{d-1} \left(\frac{i-1}{2} +\frac{d-i}{2} \right)\\
     &= \frac{1}{2}.
 \end{align*}
 As this holds for all $i \in [d]$, we are done.
\end{proof}

\section{Beyond $1/2$-Top-Link Expansion}\label{sec:scolor}
%\textcolor{red}{sh:I moved the next paragraph to intro.}
%Given an integer $k \geq 0$ and a parameter $s \geq 1$, we define the $s$-analog of $k$ and of $k!$ via
%\[
%     [k]_s = \frac{s^k - s^{-k}}{s - s^{-1}} \qquad \text{and} \qquad [k]_s! = \prod_{i=1}^k [i]_s.
% \]
% These numbers can be thought of as a ``symmetric'' version of the more standard $q$-analog of the non-negative integers. Note that $[k]_1 = k$ and $[k]_1! = k!$. 

In this section we prove \cref{thm:main-s}. 
Recall the notation $[k]_s = \frac{s^k - s^{-k}}{s - s^{-1}}$ and notice $[k]_1=k$.
Throughout this section, we fix a parameter $s \geq 1$. For $F<G$ with $F\in T_i,G\in T_j$, we define a coloring $\psi(F) = i$ and define $\phi(i,j) = [j-i]_s$ so that $\phi(F,G) =[r_F^G]_s=[j-i]_s$. Invoking \cref{lem:alpha-beta-gen} we obtain the following, and we note that the same argument here also shows condition \eqref{eq:phi2-condition} holds for $\phi$ in general.

\begin{lemma} \label{lem:alpha-beta-v3}
    If for $K < F < L$ in $X(1)$ we define
    \[
        \alpha_K^L(F) = \frac{[r_K^F]_s}{[r_K^L]_s} \qquad \text{and} \qquad \beta_K^L(F) = \frac{[r_F^L]_s}{[r_K^L]_s},
    \]
    then $\bm\alpha$ and $\bm\beta$ are commutative and give rise to a commutative family of maps $\bm\pi$ and their associated polynomials $p_\sigma$ for $\sigma \in X$.
\end{lemma}
\begin{proof}
For $K<F<G<L$ in $X(1)$, let $i=r_K^F, j=r_F^G,k=r_G^L$; we have
    \begin{align*}
        &\phi(K,G)\cdot \phi(F,L)-\phi(K,L)\cdot \phi(F,G) \\
        &=[j+k]_s \cdot [i+j]_s - [j]_s \cdot [i+j+k]_s \\
            &= \frac{(s^{j+k}-s^{-j-k}) \cdot (s^{i+j}-s^{-i-j}) - (s^j-s^{-j}) \cdot (s^{i+j+k}-s^{-i-j-k})}{(s-s^{-1})^2} \\
            &= \frac{s^{i+k}+s^{-i-k}-s^{i-k}-s^{k-i}}{(s-s^{-1})^2} \\
            &= \frac{(s^i - s^{-i}) \cdot (s^k - s^{-k})}{(s-s^{-1})^2} \\
            &= [i]_s \cdot [k]_s = \phi(K,F)\cdot \phi(G,L).
    \end{align*}
    as desired.
    %hich is the analogous key identity needed for the last two conditions of \cref{def:alpha-beta}.
\end{proof}
Note that if we set $s=1$, then the definition of $\bm\alpha$ and $\bm\beta$ in \cref{lem:alpha-beta-v3} coincides with that of \cref{cor:alpha-beta-rank}. The value of the parameter $s$ then becomes useful in the way that it makes \cref{thm:eig-bounds-main-v1} more robust, as seen in \cref{thm:eig-bounds-main-v3} below.

One key difference here is that the proof of \cref{lem:cone-non-empty-v1} does not work, and so we will prove \cref{lem:cone-containment} instead which implies the cones $\cC_\sigma$ are non-empty as a corollary. First, we need a technical lemma.

\begin{lemma} \label{lem:k-n-s-decreasing}
    For all $0 \leq k \leq n$, we have that $\frac{[k]_s}{[n]_s}$ is decreasing as a function of $s \in [1,\infty)$.
\end{lemma}
\begin{proof}
    Since $s \geq 1$, we first compute
    \[
        \partial_s[k(s^n - s^{-n}) - n(s^k - s^{-k})]
            % = nk(s^{n-1}+s^{-n-1}-s^{k-1}-s^{-k-1})
            = nk(s^{k-1}-s^{-n-1}) (s^{n-k}-1)
            \geq 0,
    \]
    which implies $\frac{[k]_s}{[n]_s} \leq \frac{k}{n}$ for $0 \leq k \leq n$. Next we prove that $\partial_s \log\left(\frac{s^k-s^{-k}}{s^n-s^{-n}}\right) \leq 0$ for $s \geq 1$. We compute
    \[
        \partial_s \log\left(\frac{s^k-s^{-k}}{s^n-s^{-n}}\right) = \frac{k(s^k+s^{-k})(s^n-s^{-n}) - n(s^n+s^{-n})(s^k-s^{-k})}{s(s^k-s^{-k})(s^n-s^{-n})}.
    \]
    Thus for $s \geq 1$ we want to prove
    \[
    \begin{split}
        k(s^k+s^{-k})(s^n-s^{-n}) &\leq n(s^n+s^{-n})(s^k-s^{-k}) \\
            \iff \frac{k}{n} &\leq \frac{(s^n+s^{-n})(s^k-s^{-k})}{(s^k+s^{-k})(s^n-s^{-n})} = 1 - \frac{2(s^{n-k} - s^{k-n})}{(s^k+s^{-k})(s^n-s^{-n})} \\
            \iff \frac{n-k}{n} &\geq \frac{2}{s^k+s^{-k}} \cdot \frac{[n-k]_s}{[n]_s}.
    \end{split}
    \]
    Since $2 \leq s^k+s^{-k}$ and $\frac{[n-k]_s}{[n]_s} \leq \frac{n-k}{n}$ via the argument given above, we thus have that the above inequality holds. Therefore $\frac{s^k-s^{-k}}{s^n-s^{-n}}$ is decreasing for $s \geq 1$.
\end{proof}

\begin{lemma}[Cone containment for $s \geq 1$] \label{lem:cone-containment}
    Given $\sigma \in X$ and $1 \leq s \leq t$, we have that $\cC_\sigma^s \subseteq \cC_\sigma^t$.
    Moreover, $\cC_\sigma = \cC_\sigma^s \neq \varnothing$ for all $\sigma \in X$ and $s\geq 1$.
\end{lemma}
\begin{proof}
    The result follows if we can show for all $\bm{v} \in \cC_\sigma^s$ and $\tau \in X_\sigma$ that $\pi^s_{\sigma+\tau}(\bm{v}) \leq \pi_{\sigma+\tau}^t(\bm{v})$ entrywise. To this end, we first note that $\frac{[k]_s}{[n]_s}$ is decreasing in $s$ for $s \geq 1$ for all $0 \leq k \leq n$ by \cref{lem:k-n-s-decreasing}. Now for $F \in X_{\sigma \cup \tau}$, we have that
    \[
        \pi_{\sigma+\tau}(\bm{v})_F = v_F - \alpha_K^L(F) \cdot v_L - \beta_K^L(F) \cdot v_K
    \]
    for some $K,L \in \sigma \cup \tau \cup \{\hat{0},\hat{1}\}$ for which $F \in (K,L)$, where possibly $v_K = 0$ and/or $v_L = 0$ (if $K \in \sigma$ and/or $L \in \sigma$). With this we have
    \[
        \pi^s_{\sigma+\tau}(\bm{v})_F = v_F - \frac{[r_K^F]_s}{[r_K^L]_s} \cdot v_L - \frac{[r_F^L]_s}{[r_K^L]_s} \cdot v_K \leq v_F - \frac{[r_K^F]_t}{[r_K^L]_t} \cdot v_L - \frac{[r_F^L]_t}{[r_K^L]_t} \cdot v_K = \pi^t_{\sigma+\tau}(\bm{v})_F,
    \]
    and this completes the proof of the first assertion. The second assertion then follows from the fact that $\cC_\sigma^1\neq \varnothing$ by \cref{lem:cone-non-empty-v1}.
\end{proof}

% \begin{corollary} \label{cor:cone-non-empty-v3}
%     The cone $\cC_\sigma$ is non-empty for all $\sigma \in X$.
% \end{corollary} 

Beyond implying $\cC_\sigma\neq \varnothing$, %\cref{cor:cone-non-empty-v3}, 
the previous lemma also adds intuition to what is happening when the $s$ parameter deviates from $s=1$. Specifically, since $\cC_\sigma$ is expanding as $s$ increases, this implies the condition of being $\cC_\sigma$-Lorentzian is becoming stronger. This is why one can expect to obtain stronger eigenvalue bounds in this case.

This leads to the main result of this section, which gives eigenvalue bounds on the random walk matrix associated to $X$ based on the machinery of the previous sections.

\begin{theorem} \label{thm:eig-bounds-main-v3}
    Fix $i < j$ in $[d]$, and let $G_{i,j}$ be the bipartite graph on $T_i \cup T_j$ with corresponding random walk matrix $P_{i,j}$ according to stationary distribution $\mu$. Let $A_{i,j}$ be the weighted adjacency matrix of $G_{i,j}$, let $D_{i,j}$ be the diagonal weighted degree matrix of $G_{i,j}$, and let $M_{i,j}$ be the diagonal matrix with entries given by $\frac{[d-j+1]_s}{[d-i+1]_s}$ for vertices in $T_i$ and $\frac{[i]_s}{[j]_s}$ for vertices in $T_j$. If $p_\varnothing$ is $\cC_\varnothing$-Lorentzian, then
    \[
        D_{i,j}^{-1/2} A_{i,j} D_{i,j}^{-1/2} \preceq M_{i,j} + \bm{w}\bm{w}^\top
    \]
    for some real vector $\bm{w}$. Moreover, the second eigenvalue of $P_{i,j}$ is at most $\sqrt{\frac{[i]_s[d-j+1]_s}{[j]_s[d-i+1]_s}}$.
\end{theorem}
\begin{proof}
    By \cref{cor:eig-bound-ell} and \cref{lem:cone-containment}, we have that
    \[
        D_{i,j}^{-1/2} A_{i,j} D_{i,j}^{-1/2} \preceq M_{i,j} + \bm{w}\bm{w}^\top
    \]
    where the entries of $A_{i,j}$ are given by $\mu_\ell(\{F,G\}) = C_\phi(i,j) \cdot \mu(\{F,G\})$,
    % \[
    % \begin{split}
    %     \mu_{\alpha,\beta}(\{F,G\}) &= \sum_{\tau\in X(d): \{F,G\}\subseteq \tau} \left[\prod_{i=1}^{k-1} \beta^{\tau_{i+1}}(\tau_i)\right] \cdot \mu(\tau) \cdot \left[\prod_{i=k+2}^d \alpha_{\tau_{i-1}}(\tau_i)\right] \\
    %         &= \sum_{\tau\in X(d): \{F,G\}\subseteq \tau} \left[\prod_{i=1}^{k-1} \frac{1}{[i+1]_s}\right] \cdot \mu(\tau) \cdot \left[\prod_{i=k+2}^d \frac{1}{[d-i+2]_s}\right] \\
    %         &= \frac{\mu(\{F,G\})}{[k]_s![d-k]_s!},
    % \end{split}
    % \]
    the entries of $D_k$ are given by the row sums of $A_k$, and the entries of $M_k$ are given by
    \[
        m_{i,j}(F,F) = \frac{\sum_{G \in T_j:F < G} \beta_F(G) \cdot \mu_\ell(\{F,G\})}{\sum_{G \in T_j:F < G} \mu_\ell(\{F,G\})} = \frac{[d-j+1]_s}{[d-i+1]_s} \qquad \text{for} \quad F \in T_i
    \]
    and
    \[
        m_{i,j}(G,G) = \frac{\sum_{F \in T_i:F < G} \alpha^G(F) \cdot \mu_\ell(\{F,G\})}{\sum_{F \in T_i:F < G} \mu_\ell(\{F,G\})} = \frac{[i]_s}{[j]_s} \qquad \text{for} \quad G \in T_j.
    \]
    Since the entries of $A_{i,j}$ are proportional to the edge weights of $G_{i,j}$ for given $i < j$, this implies the first desired result. The eigenvalue bound for $P_{i,j}$ then immediately follows from \cref{lem:eigenvaluebipartite}.
\end{proof}

\begin{proof}[Proof of \cref{thm:main-s}]
We prove the bounds for the case of $k=d$ (link of $\varnothing$). By similarity, the bound extends to the 1-skeleton of all links (of co-dimension at least 2) of $X$. In particular, note that any link of a top-link path complex is also a top-link path complex.

Let $G_\varnothing$ be the 1-skeleton of $X_\varnothing$,
 $G_{i,j}$ be the induced bipartite graph between parts $T_i,T_j$ and $P_{i,j}$ be the corresponding simple random walk.
 % First, by \cref{thm:eig-bounds-main-v3}, we have that for any $k\in [d-1]$, 
 % $$ \lambda_2(P_{k,k+1})\leq \sqrt{\frac{[k]_s[d-k]_s}{[k+1]_s[d-k+1]_s}}.$$
 % Second, by \cref{lem:dpartiteeigproduct}, for any $1\leq i<j\leq d$ we have
 % $$ \lambda_2(P_{i,j}) \leq \prod_{k=i}^{j-1} \sqrt{\frac{[k]_s[d-k]_s}{[k+1]_s[d-k+1]_s}} = \sqrt{\frac{[i]_s\cdot [d-j+1]_s\cdot }{[j]_s\cdot [d-i+1]_s}}$$
%The first bound then follows from \cref{lem:dpartiteeigenvalues}.
Note by \cref{thm:eig-bounds-main-v3}, for any $1 \leq i < j \leq d$ we have
$$ \lambda_2(P_{i,j}) \leq \sqrt{\frac{[i]_s [d-j+1]_s }{[j]_s [d-i+1]_s}}.$$

For $i < j$, let $M_{i,j}$ be a diagonal matrix with entries given by $m_i(j) := \frac{[d-j+1]_s}{[d-i+1]_s}$ for vertices in $T_i$ and $m_j(i):= \frac{[i]_s}{[j]_s}$ for vertices in $T_j$. Note that since $[.]_s$ is increasing function, $0\leq m_i(j),m_j(i)<1$. By \cref{lem:eigenvaluebipartite}, we get that $\lambda_2(P_{i,j} - M_{i, j}) \leq 0$ and therefore has at most one positive eigenvalue. 
Next we want to use \cref{lem:dpartiteeigenvalues} to bound $\lambda_2(P_\varnothing)$.
 For any fixed $i\in [d]$, we have 
 \begin{align*}
     \E_{j \in [d], j \neq i} m_i(j) &= \frac{1}{d-1} \left(\sum_{j=1}^{i-1} \frac{[j]_s}{[i]_s} +\sum_{j=i+1}^d \frac{[d-j+1]_s}{[d-i+1]_s} \right)\\
     &=\frac{1}{(d-1)(s-1)}\left(\frac{s^i-s}{s^i+1} + \frac{s^{d-i+1}-s}{s^{d-i+1}+1}\right) \\
     &= \frac{1}{(d-1)(s-1)}\left(2 - (s+1)\left(\frac{1}{s^i+1} + \frac{1}{s^{d-i+1}+1}\right)\right) \\
     &\leq \frac{2}{(d-1)(s-1)} \left(1 - \frac{s+1}{s^{\frac{d+1}{2}}+1}\right) \\
     &= \frac{2}{(d-1)(s^{\frac{d+1}{2}}+1)} \left(\frac{s^{\frac{d-1}{2}}-1}{s-1}\right) \leq \frac{2}{(d-1)(s-1)}
     %&=\frac{1}{d-1} \left(\frac{i-1}{2} +\frac{d-i}{2} \right)\\
     %&= \frac{1}{2}.
 \end{align*}
Note that the second to last inequality follows from the fact that $\frac{1}{s^i+1} + \frac{1}{s^{d-i+1}+1}$ is convex for $i\in [0,d+1]$ and minimized at $i = \frac{d+1}{2}$.
This finishes the first assertion of the theorem. 

%Before, we bound $\lambda_2(P_\varnothing)$, we show that for any $1\leq j < i\leq d$ we have
% Therefore, by the first bound,
% \begin{align*}
%     \lambda_2(P_\sigma)&\leq \frac{1}{d-1} \max_{1\leq i\leq d} \alpha_i^{-1}(s)\sum_{j=1}^{i-1} \alpha_j(s) + \alpha_i(s)\sum_{j=i+1}^d\alpha_j^{-1}(s)\\
%         &\underset{\text{\cref{fact:alphasbound}}}{\leq} \frac{1}{d-1} \max_{1 \leq i \leq d} \sum_{j=1}^{i-1} s^{j-i} + \sum_{j=i+1}^d s^{i-j} \\
%         &\leq \frac{2}{d-1} \cdot \frac{s^{-1}}{1-s^{-1}} = \frac{2}{(d-1)(s-1)}.
%         % &= \max_{1 \leq i \leq d} \sum_{j=1}^{i-1} s^{j-i} + \sum_{j=1}^{d-i} s^{i+j-d-1} \\
%         % &\underset{s^x \text{ is increasing}}{\leq} s^{-i} \cdot \int_{x=0}^{i} s^x dx + s^{i-d-1}\cdot \int_{x=0}^{d-i+1}  s^x dx \\
%         % &= \frac{s^{-i} \cdot (s^i-1) + s^{i-d-1})}{\ln(s)}
% \end{align*}
Now suppose $(X,\mu)$ is a $(1/2-\eps)$-top-link expander (for $\eps>0$). Then, we determine $s>1$ such that $[1]_s/[2]_s=1/2-\eps$:
\[
    \frac{1}{s+s^{-1}} = \frac{[1]_s}{[2]_s} = \frac{1}{2} - \epsilon \iff s + s^{-1} = \frac{2}{1-2\epsilon} \iff s = \frac{1 + 2\sqrt{\epsilon - \epsilon^2}}{1-2\epsilon}.
\]
Finally, we compute
\[
    s - 1 = \frac{2\epsilon + 2\sqrt{\epsilon - \epsilon^2}}{1-2\epsilon} \geq \frac{2\sqrt{\epsilon}}{1-2\epsilon},
\]
which implies the concrete bound.
% Now since
% \[
%     s^{-i} \cdot s^{j-i} \underset{j< i}{\leq} s^{-j} \cdot 1 = s^{-j} \qquad \text{and} \qquad s^{-(d-j+1)} \cdot s^{j-i} \underset{j<i}{\leq} s^{-(d-i+1)} \cdot 1 = s^{-(d-i+1)},
% \]
% the previous numerator is non-negative

% Next we want to use \cref{lem:dpartiteeigenvalues} to bound $\lambda_2(P_\varnothing)$.
%  Fix some $i\in [d]$ such that $i\leq d/2$. For $j\in [d]$, let $\alpha_j:=\sqrt{j/(d-j+1)}$.
% \begin{align*}\sum_{1\leq j\leq n,j\neq i} \lambda_2(P_{i,j}) &\leq \alpha_i^{-1}\cdot \sum_{j=1}^{i-1} \alpha_j + \alpha_i\cdot \sum_{j=i+1}^{d} \alpha_j^{-1}\\
% &\underset{\alpha^{-1}_j=\alpha_{d-j+1}}{=}\alpha_i^{-1}\cdot \sum_{j=1}^{i-1} \alpha_j + \alpha_i\cdot \sum_{j=1}^{d-i} \alpha_j\\
% &\underset{\alpha_x \text{ is increasing}}{\leq} \alpha_i^{-1}\cdot \int_{x=0}^{i} \alpha_x dx + \alpha_i\cdot \int_{x=0}^{d-i+1}  \alpha_x dx\\
% &= -(d-i+1)+(d+1)\alpha_i^{-1}\tan^{-1}\alpha_i - i +(d+1)\alpha_i\tan^{-1}\alpha_i^{-1} \\
% &\leq \max_{x\geq 0} (d+1)\left(-1+ \frac{\tan^{-1}(x)}{x} + x\tan^{-1}\frac{1}{x}\right).
% \end{align*}
% %THe RHS is maximized at $x=1$. \textcolor{red}{Sh: Can you prove the RHS is maximized at $x=1$? I verified it with Wolframalpha but couldn't prove it.}
% %It follows that $\max_i \E_{j\sim [d],j\neq i}\lambda_{i,j} \leq \pi/2-1$.
% %Now, the observation is that if $i< d/2$ the above quantity does not decrease if we increase $i$.
% %This is because $\sqrt{(d-i+1)/i} + \sqrt{i/(d-i+1)}$
\end{proof}

\section{Top-link Expander Path Complexes}
\label{sec:quadraticcheck}
Recall that if $\cL$ is a semimodular ranked lattice, then for any $F,G \in \cL$ we have $\rk(F) + \rk(G) \geq \rk(F \vee G) + \rk(F \wedge G)$. Note that all modular lattices are semimodular, and the lattice of flats of a matroid is semimodular. Thus \cref{cor:latticeconnected} applies to all lattices that we will consider in the following sections.

\begin{lemma} \label{lem:empty-set-connected}
    Let $\cL$ be a semimodular ranked lattice with $d+1=\rk(\hat{1}) \geq 3$, and let $X$ be the $d$-partite complex consisting of {\bf all} flags of flats in $\cL$. Then $G_\varnothing$ is  connected.
\end{lemma}
\begin{proof}
    %We show the $1$-skeleton of $X_\varnothing$ is connected. 
    Given $F,G \in X(1)$. We prove the claim by induction on $\rk(F)+\rk(G)$. 
    If $F \vee G \neq \hat{1}$ or $F \wedge G \neq \hat{0}$, then $F$ and $G$ are connected in the $1$-skeleton. 
    To see that, notice in the former case, $F$ is connected to $F\vee G$, i.e., $\{F,F \vee G\} \in X(2)$, as there is a flag of flats which contains $F, F\vee G$; similarly $G$ is connected to $F\vee G$. The other case can be shown similarly.
    
    Now, suppose $F\vee G=\hat{1}, F\wedge G=\hat{0}.$ Since
    $$ \rk(F)+\rk(G) \geq \rk(F\vee G) + \rk(F \wedge G)=\rk(\hat{1}) \geq 3$$
    it must be that $\max\{\rk(F), \rk(G)\}\geq 2$. 
    
    Say $\rk(G)\geq 2$, without loss of generality. Then there is a flat $H$ such that $H\prec G$ and $\{G,H\}\in X(2)$. Since $\rk(F)+\rk(H)<\rk(F)+\rk(G)$, by IH, $F$ is connected to $H$. Therefore,      $F$ is connected to $G$.
    %But then, $\rk(H\vee F)\leq 2$, so $F\vee H\neq \hat{1}$, and $F$ is also connected to $H$.  
\end{proof}

\begin{corollary}\label{cor:latticeconnected}
    Let $\cL$ be a semimodular ranked lattice, and let $X$ be the $d$-partite complex consisting of {\bf all} flags of flats in $\cL$. Then $G_\sigma$ is  connected for $\sigma \in X$ for which $\codim(\sigma) \geq 2$.
\end{corollary}
\begin{proof}
    % If $\sigma = \varnothing$, then $\codim(\varnothing) \geq 2$ implies $d+1 \geq 3$, an thus the desired result follows from \cref{lem:empty-set-connected}. Otherwise,
    Fix $F,G \in X_\sigma(1)$. If there exists $H \in \sigma$ such that $F < H < G$ or $G < H < F$, then there is a flag of flats in $\cL$ containing $\sigma$ and $F,G$. Thus $F$ and $G$ are connected in the $1$-skeleton of $X_\sigma$.

    Otherwise $F,G$ are incomparable, and thus there exists a minimal $L \in X_\sigma(1) \cup \{\hat{1}\}$ and a maximal $K \in X_\sigma(1) \cup \{\hat{0}\}$ such that $K < F < L$ and $K < G < L$. Note that $[K,L]$ is now a ranked lattice (where $(K,L) \subset X_\sigma(1)$), and thus if $\rk(L) - \rk(K) \geq 3$ then the desired result follows from \cref{lem:empty-set-connected} applied to $[K,L]$. For the case of $\rk(L) - \rk(K) = 2$, note that $\rk(F) = \rk(G)$ in this case. Since $\codim(\sigma) \geq 2$, there must be some $H \in X_\sigma(1)$ such that $\rk(H) \neq \rk(F)$. Therefore either $F \prec L < H$ and $G \prec L < H$, or else $H < K \prec F$ and $H < K \prec G$.
    Either way, there is a flag of flats in $\cL$ containing $\sigma$ and $F,H$ and a flag of flats in $\cL$ containing $\sigma$ and $G,H$.  Thus $F$ and $G$ are connected in the $1$-skeleton of $X_\sigma$.
\end{proof}

\subsection{$1/2$-top-link expansion of Distributive Lattices}\label{sec:distributivetoplink}

In this section we prove \cref{thm:distributivelattice}.
By \cref{cor:latticeconnected}, $(X,\mu)$ is connected. So, we just need to check the $1/2$-top-link expansion. Furthermore, by \cref{thm:quadraticcheck} we just need check $\lambda_2(P_\sigma)$ for link-contiguous $\sigma$ of co-dimension 2. In particular, given a $[i,i+1]$-link contiguous $\sigma$. We need to show $-D_\sigma/2 + A_\sigma$ has (at most) one positive eigenvalue.
% \begin{equation}\label{eq:distAD}
% A_\sigma\preceq \frac12 D_\sigma+vv^\top \Leftrightarrow  \preceq vv^\top
% \end{equation}
% for some vector $v$.

Fix such a $\sigma$. Observe that $\cL_{[\sigma_{i-1},\sigma_{i+2}]}$ is distributive lattice of rank 3. We consider all possible cases and prove the above equation in each case. Let $\psi:\{a,b,c\}\to\R_{>0}$ be a order-reversing map.
\begin{figure}[htb]%{minipage}[t]{0.5\linewidth}
\centering\begin{tikzpicture}
\begin{scope}[xshift=0cm]
\node [draw,circle]at (0,0) (a) {$a$};
%\node [draw,circle]at (1.5,0) (b) {$b$};
\node [draw,circle]at (0,1) (ab) {\small$ab$} edge (a);
\node [draw,circle] at (0,-1) () {$\varnothing$} edge (a);
\node [draw,circle]at (0,2) () {\small$abc$} edge (ab);
\end{scope}
\begin{scope}[xshift=2cm]
\node [draw,circle]at (0.75,0) (a) {$a$};
\node [draw,circle]at (0,1) (ab) {\small$ab$}edge (a);
\node [draw,circle]at (1.5,1) (ac) {\small$ac$} edge (a) ;
\node [draw,circle] at (0.75,-1) () {$\varnothing$} edge (a);
\node [draw,circle]at (0.75,2) () {\small$abc$} edge (ab) edge (ac);
\end{scope}
\begin{scope}[xshift=4.5cm]
\node [draw,circle]at (0,0) (a) {$a$};
\node [draw,circle]at (1.5,0) (b) {$b$};
\node [draw,circle]at (0.75,1) (ab) {\small$ab$} edge (a) edge (b);
\node [draw,circle] at (0.75,-1) () {$\varnothing$} edge (a) edge (b);
\node [draw,circle]at (0.75,2) () {\small$abc$} edge (ab);
\end{scope}

\begin{scope}[xshift=8cm]
\node [draw,circle] at (2,0) (a) {$a$};
\node [draw,circle] at (0,0) (b) {$b$};
\node [draw,circle] at (1,1) (ab) {\small $ab$}edge (a) edge (b);
\node [draw,circle] at (-1,1) (bc) {\small$bc$} edge(b);
\node [draw,circle] at (0,2) (abc) {\small$abc$} edge(bc) edge (ab);
\node [draw,circle] at (1,-1) (empty)  {$\varnothing$}edge (a) edge (b);
\end{scope}
\begin{scope}[xshift=11.5cm]
\node [draw,circle] at (0,0) (a) {$a$};	
\node [draw,circle] at (1.5,0) (b) {$b$};	
\node [draw,circle] at (3,0) (c) {$c$};	
\node [draw,circle] at (0,1) (ab) {\small$ab$} edge (a) edge (b);	
\node [draw,circle] at (1.5,1) (ac) {\small$ac$} edge (a) edge (c);
\node [draw,circle] at (3,1) (bc) {\small$bc$} edge (b) edge (c);
\node [draw,circle] at (1.5,2) () {\small$abc$} edge (ab) edge (bc) edge (ac);
\node [draw,circle] at (1.5,-1) () {$\varnothing$} edge (a) edge (b) edge (c);
\end{scope}
\end{tikzpicture}
\caption{All possible distributive lattices for which $\rk(\hat{1})=3.$}
\label{fig:distributivelattice}
\end{figure}
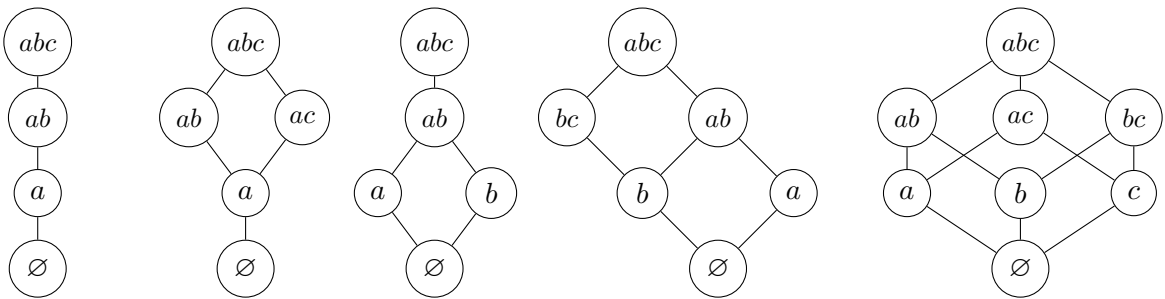

{\bf Case 1,2,3: }
Consider the left 3 lattices in \cref{fig:distributivelattice}: $\{\varnothing,a,ab,abc\}$, $\{\varnothing,a,ab,ac,abc\}$, and $\{\varnothing,a,b,ab,abc\}$. The matrices $-D_\sigma/2+A_\sigma$ are %The associated Hessians are
\[
    \begin{bmatrix}
        -\frac{a^2b}{2} & a^2b \\
        a^2b & -\frac{a^2b}{2}
    \end{bmatrix} 
    % = a^2b \begin{bmatrix}
    %     -\frac{1}{2} & 1 \\
    %     1 & -\frac{1}{2}
    % \end{bmatrix},
\quad\quad     \begin{bmatrix}
        -\frac{a^2(b+c)}{2} & a^2b & a^2c \\
        a^2b & -\frac{a^2b}{2} & 0 \\
        a^2c & 0 & -\frac{a^2c}{2}
    \end{bmatrix} 
    % = a^2 \begin{bmatrix}
    %     -\frac{b+c}{2} & b & c \\
    %     b & -\frac{b}{2} & 0 \\
    %     c & 0 & -\frac{c}{2}
    % \end{bmatrix},
%\]
%and
%\[
\quad\quad
     \begin{bmatrix}
        -\frac{a^2b}{2} & 0 & a^2b \\
        0 & -\frac{ab^2}{2} & ab^2 \\
        a^2b & ab^2 & -\frac{ab(a+b)}{2}
    \end{bmatrix} 
    % = ab \begin{bmatrix}
    %     -\frac{a}{2} & 0 & a \\
    %     0 & -\frac{b}{2} & b \\
    %     a & b & -\frac{a+b}{2}
    % \end{bmatrix},
\]
respectively. 
Note that in each case the (unweighted) graph between rank 1 and rank 2 flats is a complete multi-partite graph; so by \cref{lem:AB} its adjacency matrix has one positive eigenvalue. When we set the diagonals to 0, the above matrices can be written as $P^TAP$ where $P$ is a diagonal matrix, so they would have one positive eigenvalue (by another application of \cref{lem:AB}).
%so for any $a,b,c > 0$, each of these matrices has one positive eigenvalue when the diagonal entries are set to $0$. 
Finally, since the diagonal entries are non-positive, $a,b,c>0$, this implies each of these matrices has at most one positive eigenvalue.

{\bf Case 4: } Consider the 4th lattice of \cref{fig:distributivelattice}: $\{\varnothing,a,b,ab,bc,abc\}$.
Let $P$ be the associated Poset; in this case it is given by the single relation $b < c$. %%Let $\psi: P \to \R_{>0}$ be a (weakly) order-reversing map; 
Notice $\psi$ is order-reversing only means $\psi(b) \geq \psi(c)$. For ease of notation we abuse notation and we denote $\psi(a),\psi(b),\psi(c)$ by $a,b,c$ respectively.
 %We abuse notation assume $\psi(a)=a,\psi(b)=b,\psi(c)=c$. 
 So, e.g., $\psi(\{a,b\})=ab.$
% \begin{figure}[htb]%{minipage}[t]{0.5\linewidth}
% \centering\begin{tikzpicture}
% \node [draw,circle] at (2,0) (a) {$a$};
% \node [draw,circle] at (0,0) (b) {$b$};
% \node [draw,circle] at (1,1) (ab) {\tiny $a,b$}edge (a) edge (b);
% \node [draw,circle] at (-1,1) (bc) {\tiny$b,c$} edge(b);
% \node [draw,circle] at (0,2) (abc) {\tiny$a,b,c$} edge(bc) edge (ab);
% \node [draw,circle] at (1,-1) (empty)  {$\varnothing$}edge (a) edge (b);
% \end{tikzpicture}\end{figure}
%Consider the maximal chains to be weighted by products of weights assigned to the Poset elements $\{a,b,c\}$; e.g., the chain $\{b,ab\}$ is weighted $ab^2$. Now 

%Now consider the Hessian of the associated quadratic, with entries indexed according to the order $\{a,ab,b,bc\}$:
Note that this lattice has $3$ (maximal) flag of flats, $a < ab$, $b<ab$ and $b<bc$ with respective weights $a^2b, ab^2, b^c$. Therefore, we need to show that
\[
    H = \begin{bmatrix}
        -\frac{a^2b}{2} & a^2b & 0 & 0 \\
        a^2b & -\frac{ab(a+b)}{2} & ab^2 & 0 \\
        0 & ab^2 & -\frac{b^2(a+c)}{2} & b^2c \\
        0 & 0 & b^2c & -\frac{b^2c}{2}
    \end{bmatrix}.
\]
has (at most) one positive eigenvalue. 
If $H$ has exactly one positive eigenvalues, then its determinant will be negative. On the other hand, if the determinant is negative then $H$ has either $1$ or $3$ positive eigenvalues. If it has $3$ positive eigenvalues, then, by the Cauchy's interlacing theorem \ref{lem:AB}, every $2 \times 2$ principal minor of $H$ has at least $1$ positive eigenvalue, which is false for the above matrix. Thus $H$ has exactlye one positive eigenvalue iff $\det(H) < 0$. Note that this also holds for any weighting of the facets, not just those given by such a $\psi$.

Now, we compute the determinant of $H$:
\[
    \det(H) = -\frac{3}{16}a^3b^6c(a^2 + a(b-3c) + bc).
\]
Since $a,b,c > 0$, we have $\det(H) < 0$ iff $a^2 + a(b-3c) + bc > 0$. We compute the roots of this as a function of $a$:
\[
    r = \frac{-(b-3c) \pm \sqrt{(b-3c)^2 - 4bc}}{2} = \frac{-(b-3c) \pm \sqrt{(b-9c)(b-c)}}{2}.
\],
First notice that either $(b-3c)^2 - 4bc < 0$ or else $\sqrt{(b-3c)^2 - 4bc} < |b-3c|$, and so if $b-3c \geq 0$ then all roots $r$ are either negative or imaginary. Further, if $b-3c < 0$ and $b > c$ then $(b-9c)(b-c) < 0$, and all roots $r$ are imaginary. Thus for all $a > 0$ and $b > c$, the matrix $H$ has exactly one positive eigenvalue. Therefore, by taking limits, $H$ has at least one positive eigenvalue for all $a,b,c > 0$ such that $b \geq c$.

{\bf Case 5:} Consider the last lattice of \cref{fig:distributivelattice}  $\{\varnothing, a,b,c,ab,bc,ac,abc\}$: %We claim that for any $\psi:P\to\R_{>0}$, the Hessian of the associated quadratic has one positive eigenvalue. 
%First, without loss of generality, we can assume $\psi(c)=1$.
The matrix $-D_\sigma/2+A_\sigma$ according to the order $a,b,c,ab,ac,bc$ is
\[
    H = \begin{bmatrix}
       -\frac{a^2(b+c)}{2} & 0 & 0 & a^2b & a^2c &  0\\ 0 & -\frac{b^2(a+c)}{2} & 0 & ab^2 &0 & b^2c\\
       0 & 0 & -\frac{c^2(a+b)}{2} & 0 & ac^2& bc^2\\
   a^2b & ab^2 & 0 & -\frac{ab(a+b)}{2} & 0  & 0\\
   a^2c & 0 & ac^2 & 0 & -\frac{ac(a+c)}{2} & 0\\
   0 & b^2c & bc^2 & 0 & 0 & -\frac{bc(b+c)}{2}
    \end{bmatrix}
\]
First, assume $a,b,c$ are distinct. 
Furthermore, without loss of generality, suppose $a > b > c > 0$.
We compute
\[
    \det(H) = -\frac{3}{64} a^4b^4c^4(a^2b + ab^2 + a^2c + b^2c + ac^2 + bc^2 - 6abc)^2 < 0
\]
and letting $H'$ be the leading $5 \times 5$ principal minor of $H$, we also have
\[
    \det(H') = \frac{1}{32} a^4b^3c^3(3a^2+3ab+3ac-bc)(a^2b + ab^2 + a^2c + b^2c + ac^2 + bc^2 - 6abc).
\]
By the AM-GM inequality, we then have
\[
    \frac{a^2b + ab^2 + a^2c + b^2c + ac^2 + bc^2}{6} \geq \left(a^2b \cdot ab^2 \cdot a^2c \cdot b^2c \cdot ac^2 \cdot bc^2\right)^{1/6} = abc,
\]
with equality if and only if $a^2b = ab^2 = a^2c = b^2c = ac^2 = bc^2 \iff a=b=c$.  Thus $\det(H) < 0$ and $\det(H') > 0$. Since $\det(H) < 0$, the matrix $H$ has either $1$, $3$, or $5$ positive eigenvalues. If $H$ has at least $3$ positive eigenvalues, then, by the Cauchy's interlacing theorem (\cref{lem:AB}), $\det(H') > 0$ implies $H'$ has at least $3$ positive eigenvalues as well. This, by another application of Cauchy's interlacing theorem,  implies every $3 \times 3$ principal minor of $H'$ has at least $1$ positive eigenvalue, which is false. Therefore $H$ must have exactly one positive eigenvalue. Now, by limiting we then have that $H$ has (at most one) positive eigenvalue for any choice of $a\geq b \geq c > 0$.
This concludes the proof of \cref{thm:distributivelattice}.
%First consider the principal submatrix corresponding to rows $1,2,5,6$,
%$$
% \begin{bmatrix}
%        -\frac{a^2(b+c)}{2} & 0 & %0 & a^2b 
%         a^2c &  0\\ 
%        0 & -\frac{b^2(a+c)}{2} & %0 & ab^2 &
%        0 & b^2c\\
%        %0 & 0 & -\frac{c^2(a+b)}{2} & 0 & ac^2& bc^2\\
%    %a^2b & ab^2 & 0 & -\frac{ab(a+b)}{2} & 0  & 0\\
%    a^2c & 0 & %ac^2 & 0 & 
%    -\frac{ac(a+c)}{2} & 0\\
%    0 & b^2c & %& bc^2 & 0 & 
%    0 & -\frac{bc(b+c)}{2}
%     \end{bmatrix}$$

\subsection{Colored Distributive Lattices}
In this section we prove \cref{lem:toplinkexpansioncolorlattice}.
By \cref{cor:latticeconnected}, $(X,\mu)$ is connected. So, we just need to check the top-link expansion. Given the colorings $\phi$ of \cref{lem:toplinkexpansioncolorlattice}, by \cref{thm:quadraticcheck} we just need check $\lambda_2(-D_{\phi,\sigma} + A_\sigma)$ for link-contiguous $\sigma$ of co-dimension 2. In particular, given a $[i,i+1]$-link contiguous $\sigma$. We need to show $-D_{\phi,\sigma} + A_\sigma$ has (at most) one positive eigenvalue.

We start with the first assertion of \cref{lem:toplinkexpansioncolorlattice}. Let $\phi: \{a,b,c\} \to \R_{>0}$ be any map. For ease of notation, we denote $a = \phi(a)$, $b = \phi(b)$, and $c = \phi(c)$.

%First we will show here that the associated quadratics are Lorentzian for \emph{any} choice of $\phi$ defined via $\phi(F) = \sum_{i \in F} \phi(i)$ for all flats $F$. Recall that there are exactly $5$ distinct rank-$3$ distributive lattices. We list them now, along with the Hessians of the associated quadratic polynomials:

% \begin{enumerate}
%     \item $P = \{\varnothing,a,ab,abc\}, H = \begin{bmatrix}
%         -\frac{c}{b+c} & 1 \\ 1 & -\frac{a}{a+b}
%     \end{bmatrix}$,
%     \item $P = \{\varnothing,a,ab,ac,abc\}, H = \begin{bmatrix}
%         -1 & 1 & 1 \\
%         1 & -\frac{a}{a + b} & 0 \\
%         1 & 0 & -\frac{a}{a + c}
%     \end{bmatrix}$,
%     \item $P = \{\varnothing,a,b,ab,abc\}, H = \begin{bmatrix}
%         -\frac{c}{b + c} & 0 & 1 \\
%         0 & -\frac{c}{a + c} & 1 \\
%         1 & 1 & -1
%     \end{bmatrix}$,
%     \item $P = \{\varnothing,a,b,ab,bc,abc\}, H = \begin{bmatrix}
%         -\frac{c}{b+c} & 0 & 1 & 0 \\
%         0 & -1 & 1 & 1 \\
%         1 & 1 & -1 & 0 \\
%         0 & 1 & 0 & -\frac{b}{b + c}
%     \end{bmatrix}$,
%     \item $P = \{\varnothing,a,b,c,ab,ac,bc,abc\}, H = \begin{bmatrix}
%         -1 & 0 & 0 & 1 & 1 & 0 \\
%         0 & -1 & 0 & 1 & 0 & 1 \\
%         0 & 0 & -1 & 0 & 1 & 1 \\
%         1 & 1 & 0 & -1 & 0 & 0 \\
%         1 & 0 & 1 & 0 & -1 & 0 \\
%         0 & 1 & 1 & 0 & 0 & -1
%     \end{bmatrix}$.
% \end{enumerate}

{\bf Case 1,2,3: }
Consider the left 3 lattices in \cref{fig:distributivelattice}: $\{\varnothing,a,ab,abc\}$, $\{\varnothing,a,ab,ac,abc\}$, and $\{\varnothing,a,b,ab,abc\}$. The matrices $-D_{\phi,\sigma}+A_\sigma$ are
\[
    \begin{bmatrix}
        -\frac{c}{b+c} & 1 \\ 1 & -\frac{a}{a+b}
    \end{bmatrix}
    \quad\quad
    \begin{bmatrix}
        -1 & 1 & 1 \\
        1 & -\frac{a}{a + b} & 0 \\
        1 & 0 & -\frac{a}{a + c}
    \end{bmatrix}
    \quad\quad
    \begin{bmatrix}
        -\frac{c}{b + c} & 0 & 1 \\
        0 & -\frac{c}{a + c} & 1 \\
        1 & 1 & -1
    \end{bmatrix}
\]
Here, the same matrix but with $0$ diagonal values has at most one positive eigenvalue since it is the adjacency matrix of a complete bipartite graph. Thus each of those three matrices has at most one positive eigenvalue.

{\bf Case 4: } Consider the 4th lattice of \cref{fig:distributivelattice}: $\{\varnothing,a,b,ab,bc,abc\}$.
Let $P$ be the associated Poset; in this case it is given by the single relation $b < c$. The matrix $-D_{\phi,\sigma} + A_\sigma$ is
\[
    \begin{bmatrix}
        -\frac{c}{b+c} & 0 & 1 & 0 \\
        0 & -1 & 1 & 1 \\
        1 & 1 & -1 & 0 \\
        0 & 1 & 0 & -\frac{b}{b + c}
    \end{bmatrix}
\]
Now consider the matrix
\[
    H(t) = \begin{bmatrix}
        -t & 0 & 1 & 0 \\
        0 & -1 & 1 & 1 \\
        1 & 1 & -1 & 0 \\
        0 & 1 & 0 & -1+t
    \end{bmatrix},
\]
for which the characteristic polynomial is given by
\[
    \chi_{H(t)}(x) %= -t^2 x^2 - 2 t^2 x + t x^2 + 2 t x + x^4 + 3 x^3 - 3 x
        = x^4 + 3x^3 + x^2 (-t^2 + t) + x (-2t^2 + 2t - 3).
\]
Thus for all $t \in (0,1)$, the sign pattern of the coefficients of the above polynomial is $(+,+,+,-)$. Descartes' rule of signs then implies $H(t)$ has exactly one positive eigenvalue for all $t \in (0,1)$.

{\bf Case 5:} Consider the last lattice of \cref{fig:distributivelattice}  $\{\varnothing, a,b,c,ab,bc,ac,abc\}$: %We claim that for any $\psi:P\to\R_{>0}$, the Hessian of the associated quadratic has one positive eigenvalue. 
%First, without loss of generality, we can assume $\psi(c)=1$.
The matrix $-D_{\phi,\sigma}+A_\sigma$ according to the order $a,b,c,ab,ac,bc$ is
\[
    \begin{bmatrix}
        -1 & 0 & 0 & 1 & 1 & 0 \\
        0 & -1 & 0 & 1 & 0 & 1 \\
        0 & 0 & -1 & 0 & 1 & 1 \\
        1 & 1 & 0 & -1 & 0 & 0 \\
        1 & 0 & 1 & 0 & -1 & 0 \\
        0 & 1 & 1 & 0 & 0 & -1
    \end{bmatrix},
\]
which has one positive eigenvalue by direct computation.
This concludes the proof of the first assertion of \cref{lem:toplinkexpansioncolorlattice}.

Next, %\cref{lem:toplinkexpansioncolorlattice}, given by 
let $\phi_M(F) = \I[F \cap A \neq \varnothing] \cdot M + |F|$ for some $M>0$. We now prove the second assertion of \cref{lem:toplinkexpansioncolorlattice}.

{\bf Case 1,2,3:} The proof is exactly the same for cases $(1),(2),(3)$ given above.

{\bf Case 4:} The $P$-consistent property is relevant, and there are four cases to consider: $a \in A$, $b \in A$, $c \in A$, and $a,b \in A$. Note that we do not consider the case $a,c \in A$ since $A$ is $P$-consistent. We then don't need to consider $b,c \in A$ since $b \in A$ and $b,c \in A$ give the same function $\phi$, and we don't need to consider $a,b,c \in A$ since $a,b \in A$ and $a,b,c \in A$ give the same function $\phi$. Next note that the $a \in A$ case, the $b \in A$ case, and the $c \in A$ case are instances of cases already handled in the first assertion, so we can skip them. Finally for the $a,b \in A$ case, the matrix $-D_{\phi,\sigma} + A_\sigma$ is
\[
    \begin{bmatrix}
        -\frac{1}{2} & 0 & 1 & 0 \\
        0 & -1 & 1 & 1 \\
        1 & 1 & -\frac{2(M+1)}{M+2} & 0 \\
        0 & 1 & 0 & -\frac{M+1}{M+2}
    \end{bmatrix},
    % \preceq \begin{bmatrix}
    %     -\frac{1}{2} & 0 & 1 & 0 \\
    %     0 & -1 & 1 & 1 \\
    %     1 & 1 & -1 & 0 \\
    %     0 & 1 & 0 & -\frac{1}{2}
    % \end{bmatrix},
\]
which has at most one positive eigenvalue for all $M \geq 0$ since for $M=0$ the matrix has at most one positive eigenvalue.

{\bf Case 5:} The $P$-consistent property is irrelevant, and thus we compute the associated Hessian for three cases: where $a \in A$ (and $b,c \not\in A$), where $a,b \in A$ (and $c \not\in A$), and where $a,b,c \in A$. Note that the $a \in A$ case is an instance of cases handled in the first assertion, so we can skip it. For the $a,b \in A$ case and the $a,b,c \in A$ case, the matrices $-D_{\phi,\sigma} + A_\sigma$ are
\[
    \begin{bmatrix}
        -1 & 0 & 0 & 1 & 1 & 0 \\
        0 & -1 & 0 & 1 & 0 & 1 \\
        0 & 0 & -\frac{2}{M+2} & 0 & 1 & 1 \\
        1 & 1 & 0 & -\frac{2(M+1)}{M+2} & 0 & 0 \\
        1 & 0 & 1 & 0 & -1 & 0 \\
        0 & 1 & 1 & 0 & 0 & -1
    \end{bmatrix}, \qquad \begin{bmatrix}
        -1 & 0 & 0 & 1 & 1 & 0 \\
        0 & -1 & 0 & 1 & 0 & 1 \\
        0 & 0 & -1 & 0 & 1 & 1 \\
        1 & 1 & 0 & -\frac{2(M+1)}{M+2} & 0 & 0 \\
        1 & 0 & 1 & 0 & -\frac{2(M+1)}{M+2} & 0 \\
        0 & 1 & 1 & 0 & 0 & -\frac{2(M+1)}{M+2}
    \end{bmatrix}
\]
which has at most one positive eigenvalue for all $M \geq 0$ since for $M=0$ both matrices have at most one positive eigenvalue.
% For the $a,b,c \in A$ case, we compute the associated Hessian
% \[
%     H = \begin{bmatrix}
%         -1 & 0 & 0 & 1 & 1 & 0 \\
%         0 & -1 & 0 & 1 & 0 & 1 \\
%         0 & 0 & -1 & 0 & 1 & 1 \\
%         1 & 1 & 0 & -\frac{2(M+1)}{M+2} & 0 & 0 \\
%         1 & 0 & 1 & 0 & -\frac{2(M+1)}{M+2} & 0 \\
%         0 & 1 & 1 & 0 & 0 & -\frac{2(M+1)}{M+2}
%     \end{bmatrix},
% \]
% which has at most one positive eigenvalue since the diagonal entries are at most those of the original case $(5)$ given above.

\subsection{$1/2$-top link expansion of Typical Modular Lattices}
In this section we prove \cref{thm:modularlattice}.
\begin{definition}[Unique Neighbor Property]
    Let $G=(X,Y,E)$ be a bipartite graph. We say $G$ has the {\bf unique neighbor property} if
for any pair of vertices $x_i,x_j \in X$ they have exactly one common neighbor in $Y$ and any pair of vertices $y_i, y_j \in Y$ have exactly one common neighbor in $X$.
\end{definition}

\begin{lemma}\label{lem:uniqueneighbor}
    Let $\cL$ be a modular lattice with $\rk(\hat{1})=3$. Then, the bipartite graph defined by the rank 1 and rank 2 flats of $\cL$ has the unique neighbor property.
\end{lemma}
\begin{proof}
    The proof uses the modular property of the rank.
    For any two rank 1 flats $F,G\in \cL$ we have $F\wedge G=\hat{0}$, i.e., $\rk(F\wedge G)=0$. Therefore, we must have $rk(F\vee G)=2.$
    Conversely, consider two flats $F,G$ of rank 2. We have $F\vee G=\hat{1}$, i.e., $\rk(F\vee G)=3$. So, we must have $F\wedge G$ has rank 1.
\end{proof}

\begin{theorem}\label{thm:uniqueneighborgraphs}
    Let $G=(X,Y,E)$ be a bipartite graph with the unique neighbor property with adjacency matrix $A$ and the diagonal matrix $D$ of vertex degrees. Then, $-D/2+A$ has exactly one positive eigenvalue iff one of the following holds: 
    %\begin{itemize}
       %\item Minimum degree of $G$ is $\geq 2$.
         %$G$ has a vertex of degree 1. Let $a_x,a_y$ be the number of degree 1 vertices in $X,Y$ respectively. 
         Either $X$ has no vertex of degree 1, or $Y$ has no vertex of degree 1, or both $X,Y$ have exactly one vertex of degree 1. 
         %$a_x=0$ or $a_y=0$ or $a_x=a_y=1$.
    %\end{itemize}
\end{theorem}

\begin{proof}[Proof of \cref{thm:modularlattice}]
Given a typical modular lattice $\cL$, 
by \cref{cor:latticeconnected}, $(X,\mu)$ is connected. So, we just need to check the $1/2$-top-link expansion. Furthermore, by \cref{thm:quadraticcheck} we just need check $\lambda_2(P_\sigma)$ for link-contiguous $\sigma$ of co-dimension 2. In particular, given a $[i,i+1]$-link contiguous $\sigma$, we need to show $-D_\sigma/2 + A_\sigma$ has (at most) one positive eigenvalue.

Fix such a $\sigma$. 
$\cL_{[\sigma_{i-1},\sigma_{i+2}]}$ is a modular lattice of rank 3.
Let $G=({\cal F}_1,{\cal F}_2,E)$ be the bipartite graph between rank 1 and rank 2 flats of this lattice. Furthermore the adjacency matrix of $G$ is the same as $A_\sigma$ and the degree matrix is the same as $D_\sigma$. Finally, since $\cL$ is a modular lattice, by \cref{lem:uniqueneighbor}, $G$ has the unique neighbor property. Further, since it is typical, either $\cF_1$ has no vertex of degree 1, or $\cF_2$ has no vertex of degree 1, or both have exactly one vertex of degree 1. Therefore, by \cref{thm:uniqueneighborgraphs}, $-D_\sigma/2+A_\sigma$ has one positive eigenvalue. The converse direction simply follows by the if and only if condition in \cref{thm:uniqueneighborgraphs}
%Conversely, if $\cL$ is not typical, it means there is a $[i,i+1]$-link contiguous $\sigma$ such that the graph $G({\cal F}_1,{\cal F}_2)$ has degree 1 vertices in both of $\cF_1,\cF_2$ and it has strictly more than one such vertex in at last one of them. So, by \cref{t}
 %We consider all possible cases and prove the above equation in each case. Let $\psi:\
\end{proof}
In the rest of this section we prove \cref{thm:uniqueneighborgraphs}.
Note that since $\bm{1}^\top (-D/2+A) \bm{1} > 0$ then implies $-D/2+A$ has at least  one positive eigenvalue. So, we must prove that in the above cases it has {\em at most} one positive eigenvalue.

\begin{lemma}\label{lem:uniqueneighbordeg1}
    Let $G=(X,Y,E)$ be a bipartite graph with the unique neighbor property. Suppose that $G$ has at least 1 vertex of degree 1. %Let $a_x$ be the number of degree 1 vertices in $X$ and $a_y$ be the number of degree 1 vertices in $Y$. 
    Then, $-D/2+A$ has one positive eigenvalue iff $X$ has no vertex of degree 1, or $Y$ has no vertex of degree 1, or both $X,Y$ have exactly one vertex of degree 1. %$a_x=0$, or $a_y=0$, or $a_x=a_y=1$.
\end{lemma}
\begin{proof}
To see this, let $x_1$ be a degree 1 vertex in $X$ (the case $x_1\in Y$ can be dealt with similarly) with $y_0\in Y$ its unique neighbor. Then, by the unique neighbor property every vertex in $X$ must be adjacent to $y_0$. This immediately implies that any vertex $y\neq y_0$ in $Y$ must have degree 1; otherwise a pair of vertices in $X$ will have two common neighbors. 

{\bf Case 1:} $|Y|=1$, i.e., $Y=\{y_0\}$ and $Y$ has no vertex of degree 1. %This corresponds to $a_y=0$ case.
 $G$ corresponds to a star, i.e., it is a complete multi-partite graph and by \cref{lem:AB} the adjacency matrix of  $G$ has only one positive eigenvalue. Therefore, $-D/2+A$ also has one positive eigenvalue and we are done.

 {\bf Case 2:} $|Y|\geq 2$.  So, the $Y$ side also has a degree 1 vertex, say $y_1$. So, we must also have $|X|\geq 2$; because all vertices of $Y$ are connected to a vertex in $X$ which cannot be $x_1$ as it has degree 1. Further, with an argument similar to above, there is a vertex $x_0\in X$ adjacent to every vertex in $Y$ and all other vertices of $X$ should have degree 1.

{\bf Case 2.1:} $|X|=|Y|=2$, i.e., both $X,Y$ have exactly one degree one vertex. Then, $G$ must be a path of length 3 as shown below.
    \begin{figure}[htb]\centering\begin{tikzpicture}
        \node [draw,circle] at (0,1) (a){};
        \node [draw,circle] at (1,0) (b){} edge (a);
        \node [draw,circle] at (2,1) (c) {} edge (b);
        \node [draw,circle] at (3,0) (d) {} edge (c);
    \end{tikzpicture}
    \end{figure}
    For this graph $-D/2+A$ has one positive eigenvalue. % as shown in \cref{sec:distributivetoplink}.
    
{\bf Case 2.2:} $|X|+|Y|\geq 5$. So each side has at least one degree one vertex and one side has at least two such vertices. In this case, we show that $B=-D/2+A$ has at least 2 positive eigenvalue. 
First we observe if $|X|+|Y|=5$, then 
$$ -D/2+A = \begin{bmatrix}-1 &0& 0 &1& 1 \\ 0 &-0.5 &0 &1& 0\\ 0 &0& -0.5& 1&0\\ 1&1&1&-1.5&0\\1 & 0 &0 &0 &-0.5\end{bmatrix}$$
and it has two positive eigenvalue.
Now, assume $|X|+|Y|\geq 6$.
First notice $B\bone =D\bone/2$.  
%Let us re-name vertices and assume $x_0,y_0$ are the only vertices of degree $>1$ in $X,Y$.
So, assuming $|X|+|Y|\geq 6$. Further, let $|X|-1=m,|Y|-1=n$.
%and assume wlog that $n\geq m$.
Consider the vector 
$$v_x=\begin{cases}
-1 &\text{if }x=x_0\\
\frac{2n+1}{2m+1} &\text{otherwise}.
\end{cases},  \quad \quad v_y=\begin{cases}\frac{2n+1}{2m+1} & \text{if }y=y_0\\ -1 &\text{otherwise}\end{cases}$$
It follows that
\begin{align*}&v^\top (-D/2+A)v \\
&=  \frac{2m(2n+1)^2}{(2m+1)^2} + 2n -\frac{2(2n+1)}{2m+1} -\frac{m(2n+1)^2}{2(2m+1)^2}- \frac{n}{2} - \frac{n+1}{2}-\frac{(m+1)(2n+1)^2}{2(2m+1)^2}\\
&=\frac{4m(2n+1)^2+4n(2m+1)^2-4(2m+1)(2n+1)-(2m+1)(2n+1)^2 - (2n+1)(2m+1)^2}{2(2m+1)^2}\\
&= \frac{(2m-1)(2n+1)^2+(2n-1)(2m+1)^2-4(2n+1)(2m+1)}{2(2m+1)^2}\\
&\geq \frac{(2m-3)(2n+1)^2 + (2n-3)(2m+1)^2} {2(2m+1)^2}\underset{\text{if }m,n\geq 2}{>}0
%&\underset{n\geq m}{\geq} 
\end{align*}
%On  $m,n\geq 2$.
On the other hand, say $m=1,n\geq 3$, we have
$$ v^\top (-D/2+A)v \underset{m=1}{=} \frac{(2n+1)^2 +9(2n-1) -12(2n+1)}{2(2m+1)^2} \geq \frac{4n^2 -2n -20}{2(2m+1)^2}\underset{n\geq 3}{>}0$$
%by the assumption of the case 
%To the inequality, notice if $m=n$ then we must have $m,n\geq 2$; so $v^\top B v\geq 2m+2n-2-m-n-1>0$. Otherwise, if $n>m$, then  $v^\top B v> 2m+2n-2-m-n-1\geq 0$.

Let $d$ be the vector of degrees, i.e., $d_x$ is the degree of $x$.  Now, we claim that for any $u=\alpha \bone+\beta v$, the quadratic form $u^T B u>0$. This immediately implies that $\lambda_2(B)>0$ by the Rayleigh quotient. To see that notice
\begin{align*}
    \langle Bu,u\rangle &= \langle B(\alpha \bone + \beta v),u\rangle = \alpha^2 \langle B1,1\rangle + \beta^2 \langle Bv,v\rangle + 2\alpha\beta \langle B\bone,u\rangle \\
    &=\alpha^2 \langle B1,1\rangle + \beta^2 \langle Bv,v\rangle + 2\alpha\beta\langle d/2,v\rangle>0,
\end{align*}
where in the inequality we used that by definition  by definition of $v$, $\langle v,d\rangle=0$. 
\end{proof}

\begin{lemma}\label{lem:uniqneighbordeg2}
    Let $G=(X,Y,E)$ be a bipartite graph with unique neighbor property and  minimum degree  2. Then, every vertex of $G$  has degree $2$ except (possibly) for one from each side of the graph, $x_0$ and $y_0$, where $y_0$ is adjacent to every vertex on the other side except for $x_0$ and similarly $x_0$ is adjacent to every vertex on the other side except for $y_0$.
\end{lemma}
\begin{proof}
    Fix a vertex say $y^*$ of degree 2 in $G$. Wlog assume $y^*\in Y$.
Say, $x_1, x_2\in X$ are neighbors of $y^*$. Then, by the unique neighbor property, neighbors of $x_1,x_2$ partition $Y-\{y_1\}$. To see this notice that every vertex $y\neq y^*$ must be adjacent to exactly one of $x_1,x_2$ since it must have exactly one common neighbor with $y^*$.
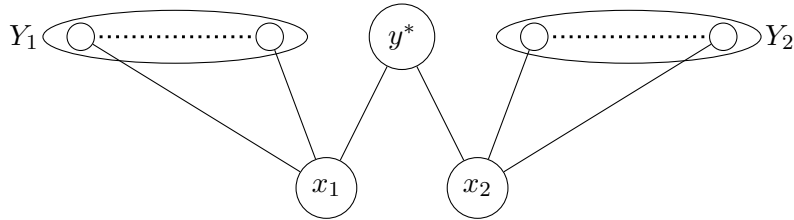
\begin{figure}[htb]
\centering
\begin{tikzpicture}
    \node [draw,circle] at (4,2) (ys) {$y^*$};
    \node [draw,circle] at (3,0) (x1) {$x_1$} edge (ys);
    \node [draw,circle]at (5,0) (x2) {$x_2$} edge (ys);
    ds
    \draw (1,2) ellipse (1.75 and 0.35);
    \draw (7,2) ellipse (1.75 and 0.35);
    \node [draw,circle] at (-0.25,2) () {} edge (x1);
    \node [draw,circle] at (2.25,2) () {} edge (x1);
    \draw [dotted,line width=1.1pt] (0,2)-- (2,2);
    \node [draw,circle] at (5.75,2) () {} edge (x2);
    \node [draw,circle] at (8.25,2) () {} edge (x2);
    \draw [dotted,line width=1.1pt] (6,2)-- (8,2);
    \node at (-1,2) ()  {$Y_1$};
    \node at (9,2) () {$Y_2$};
\end{tikzpicture}
\caption{$y^*$ has degree 2; neighbors of $x_1,x_2$ partition the vertices of $Y-\{y^*\}$}
\label{fig:deg2uniqueneighbor}
\end{figure}
Say $N(x_1)=Y_1\cup \{y^*\}, N(x_2)=Y_2\cup \{y^*\}$, where $N(x)$ is the set of neighbors of $x$  (see \cref{fig:deg2uniqueneighbor}). Now, by another application of the unique neighbor property every vertex $x\neq x_1,x_2 \in X$ cannot be adjacent to two vertices in $Y_1$ (or two vertices in $Y_2$). This is because every pair of vertices in $Y_1$ already have $x_1$ as the common neighbor.
So, every vertex $x\in X$ must have $\leq 1$  neighbor in $Y_1$ and $\leq 1$ in $Y_2$. Now if $|Y_1|, |Y_2| \geq 2$ we get a contradiction. 
To see this, suppose for contradiction that $y_{1,1},y_{1,2}\in Y_1, y_{2,1},y_{2,2}\in Y_2$. Then, $y_{1,1},y_{2,1}$ must have a common neighbor say $x$ and $y_{1,2},y_{2,2}$ must have a common neighbor say $x'$. But $x,x'$ cannot have a common neighbor.

Therefore, the only possible case is when, wlog, $Y_1=\{y_0\}$ and every vertex $x\neq x_2 \in X$ is adjacent to $y_0$ and a vertex in $Y_2\cup \{y^*\}$. 
\begin{figure}[htb]\centering
\begin{tikzpicture}
    \node [draw,circle] at (0,2) (y0) {\tiny{$y_0$}};
    \node [draw,circle] at (2,2) (ys){\tiny{$y^*$}};
    \node [draw,circle,inner sep=0.22cm] at (3,2) (y2) {};
    \node [draw,circle,inner sep=0.22cm]
    at (5,2) (y3) {};
    \node [draw,circle] at (2,0) () {\tiny{$x_1$}} edge (y0) edge (ys);
    \node [draw,circle,inner sep=0.22cm] at (3,0) () {} edge (y2) edge (y0);
    \node [draw,circle,inner sep=0.22 cm] at (5,0) () {} edge (y3) edge (y0);
    \draw [dotted,line width=1.1pt] (3.5,0) -- (4.5,0) (3.5,2) -- (4.5,2);
    \node [draw,circle] at (7,0) () {\tiny{$x_2$}} edge (ys) edge (y2) edge (y3);
\end{tikzpicture}
\end{figure}
So, every vertex in $X$ except $x_2$ has degree 2. It also follows that every vertex in $Y$ except $y_0$ has degree 2 (so $x_2$ is the $x_0$ vertex in lemma's statement).
\end{proof}

\begin{lemma}\label{lem:uniqueneighbordeg22}
    Let $G=(X,Y,E)$ be a bipartite graph  with the unique neighbor property such that the minimum degree of $G$ is 2. % Let $A\in \R^{(X+Y)\times (X+Y)}$ be the adjacency matrix of $G$ and $D$ be the diagonal matrix of vertices degrees. Then,
     Then, $-\frac12 D + A$ has  one positive eigenvalue.
\end{lemma}
\begin{proof}    
    Notice $A$ is a block-diagonal matrix, say $A=\begin{bmatrix} 0 & B \\ B^\top & 0\end{bmatrix}$. 
    %It follows that 
    %$$ A^2 = \begin{bmatrix} BB^\top & 0 \\ 0 & B^\top B\end{bmatrix}$$
%For bipartite graphs having all vertices at least degree $2$, with at least one element of degree $2$, we 
By \cref{lem:uniqneighbordeg2}, we can write
\[
    B = \begin{bmatrix}
        0 & 1 & 1 & 1 & \cdots & 1 \\
        1 & 1 & 0 & 0 & \cdots & 0 \\
        1 & 0 & 1 & 0 & \cdots & 0 \\
        1 & 0 & 0 & 1 & \cdots & 0 \\
        \vdots & \vdots & \vdots & \vdots & \ddots & \vdots \\
        1 & 0 & 0 & 0 & \cdots & 1
    \end{bmatrix}
\]
Here, the first row and column of $B$ corresponds to the large degree vertices in $X,Y$. Let $n+1=|X|=|Y|$.
Consider the matrix
\[
    B' = \begin{bmatrix}
        \frac{n-2}{2} & 1 & 1 & 1 & \cdots & 1 \\
        1 & 1 & 0 & 0 & \cdots & 0 \\
        1 & 0 & 1 & 0 & \cdots & 0 \\
        1 & 0 & 0 & 1 & \cdots & 0 \\
        \vdots & \vdots & \vdots & \vdots & \ddots & \vdots \\
        1 & 0 & 0 & 0 & \cdots & 1
    \end{bmatrix}.
\]
We have
\[
    B'(B')^\top = (B')^\top B' = \begin{bmatrix}
        \frac{n^2}{4} + 1 & \frac{n}{2} & \frac{n}{2} & \cdots & \frac{n}{2} \\
        \frac{n}{2} & 2 & 1 & \cdots & 1 \\
        \frac{n}{2} & 1 & 2 & \cdots & 1 \\
        \vdots & \vdots & \vdots & \ddots & \vdots \\
        \frac{n}{2} & 1 & 1 & \cdots & 2
    \end{bmatrix}
    = I + \left(\frac{n}{2} \bm{e}_1 + \sum_{i=2}^{n+1} \bm{e}_i\right) \left(\frac{n}{2} \bm{e}_1 + \sum_{i=2}^{n+1} \bm{e}_i\right)^\top,
\]
where $\bm{e}_i$ is the indicator vector of the $i$-th coordinate.
Thus $B'(B')^\top = (B')^\top B'$ has exactly $n$ eigenvalues equal to $1$ and exactly one eigenvalue greater than $1$. Let $A'=\begin{bmatrix} 0 & B'\\ {B'}^\top & 0\end{bmatrix}$. Therefore the matrix
\[
    (A')^2 = 
    \begin{bmatrix}
        B'(B')^\top & 0 \\ 0 & (B')^\top B'
    \end{bmatrix}
\]
has exactly $2n$ eigenvalues equal to $1$ and the other $2$ greater than $1$. Since $A'$ is the adjacency matrix of a bipartite graph, using \cref{lem:AB}, $A'$ has  $n$ many $+1$ eigenvalues, $n$  many $-1$ eigenvalues, one greater than $+1$ and one less than $-1$. Therefore,  $A' - I$ has exactly one positive eigenvalue. Since
\[
    -\frac12 D+A = A' - I - \frac{n-2}{2}(\bm{e}_1 + \bm{e}_{n+2})(\bm{e}_1 + \bm{e}_{n+2})^\top,
\]
this implies $-D/2+A$ has at most one positive eigenvalue. 
\end{proof}

\begin{lemma}\label{lem:projplane3rdprop}
   % Let $\cL$ be a rank 3 modular lattice with minimum/maximum element $\hat{0}/\hat{1}$. 
   Let $G=(X,Y,E)$ be a non-empty bipartite graph with the unique neighbor property 
    %${\cal F}_1, {\cal F}_2$ be the set of rank 1 and rank 2 flats of $\cL$ and let $G=({\cal F}_1,{\cal F}_2,E)$ be the bipartite coomparability graph between flats of rank 1 and 2. If 
    such that the minimum degree of $G$ is $\geq 3$. Then, there are 4 vertices $x_1,\dots,x_4\in X$ such that any $y\in Y$ is adjacent to at most two of them (similarly there are $4$ vertices in $Y$ that any vertex in $X$ is adjacent to at most two of them).
    %Then, for some $d\geq 3$, $G$ is $d+1$-regular and $|X|=|Y|=d^2+d+1$.
    %In fact, $\cL$ corresponds to a projective plane.
\end{lemma}
\begin{proof}
    Let $y_1,y_2\in Y$ be two distinct vertices (note that $|Y|\geq 3$ by lemma's assumption). Since degree of $y_1,y_2$ is at least 3, by the unique neighbor property, there are two vertices $x_1,x_2$ adjacent to $y_1$ and two vertices $x_3,x_4$ adjacent to $y_2$ such that $y_1$ is not connected to $x_3,x_4$ and $y_2$ is not connected to $x_1,x_2$. Now, we claim that any vertex $y\in Y$ is connected to at most two of $\{x_1,\dots,x_4\}$. For contradiction, suppose there is a vertex $y$ that is connected to, say wlog, $x_1,x_2,x_3$. Then, $y\neq y_1,y_2$ by the above assumption. But then $y$ has two common neighbors to $y_1$ and that is a contradiction with the unique neighbor property. 
\end{proof}

\begin{definition}[Projective Plane]A projective plane is an incidence structure of "points" and "lines" with the following properties.
\begin{itemize}
    \item 
Given any two distinct points, there is exactly one line incident with both of them.
\item 
Given any two distinct lines, there is exactly one point incident with both of them.
\item 
There are four points such that no line is incident with more than two of them.
\end{itemize}
\end{definition}
The above lemma implies that our bipartite graph with the unique neighbor property and min degree 3 corresponds to a projective plane.
\begin{lemma}\label{lem:projplane}
    Every point in a projective plane is incident with a constant $d + 1$ lines and every line is incident with $d + 1$ points. Such a projective plane has  $d^2 + d + 1$ points and lines.
\end{lemma}
We refer interested reader to \cite{BW11} for background on projective plane and proof of the above lemma. 

\begin{lemma}\label{lem:projeplanelorentzian}
    Given a $(d+1)$-regular bipartite graph $G=(X,Y,E)$ with $d\geq 1$, $|X|=|Y|=d^2+d+1$, and the unique neighbor property, $\frac{-\sqrt{d}}{d+1}D + A$ has one positive eigenvalue. Therefore, $-D/2+A$ also has one positive eigenvalue.
\end{lemma}
\begin{proof}
Since $G$ is $(d+1)$-regular  we have that $D = (d+1)I$. Furthermore, by the unique neighbor property and the $(d+1)$-regularity, we have
$$ A^2 = \begin{bmatrix}J+dI & 0 \\ 0 & J+dI\end{bmatrix}$$
This is because by the unique neighbor property for every pair of vertices $x,x'\in X$ there is exactly one path of length 2 from $x$ to $x'$ (and similarly for the $y$-vertices).
where $J$ is a $|X|\times |X|$ all-ones matrix. It follows that $A^2$ has two eigenvalues equal to $d^2+d+1+d=(d+1)^2$ and $2|X|-2$ many eigenvalues equal to $d$.
Therefore, since $G$ is a bipartite graph, by \cref{lem:AB}, $A$ has the following spectrum
\[
    \big(d+1, \underbrace{\sqrt{d}, \dots, \sqrt{d}}_{d^2+d\text{ many}}, \underbrace{-\sqrt{d}, \ldots, -\sqrt{d}}_{d^2+d\text{ many}}, -(d+1)\big).
\]
Therefore,
$$ \frac{-\sqrt{d}}{d+1}D + A = -\sqrt{d}I + A$$
has only one positive eigenvalue. Consequently, since $d\geq 1$, $-D/2+A$ has one positive eigenvalue since $-1/2\leq  \frac{-\sqrt{d}}{d+1}$.
\end{proof}

\begin{proof}[Proof of \cref{thm:uniqueneighborgraphs}]
If $G$ has a degree 1 vertex, then by \cref{lem:uniqueneighbordeg1} we are done.
Otherwise, if the minimum degree of $G$ is 2, then by \cref{lem:uniqueneighbordeg22} we are done.
Otherwise, the minimum degree of $G$ is at least 3. Then, by \cref{lem:projplane3rdprop} $G$ corresponds to a projective plane.
So, we are done by \cref{lem:projplane} and \cref{lem:projeplanelorentzian}.
\end{proof}

\subsection{Matroid}

In this section we prove \cref{lem:toplinkexpansionmatroid}. Fix any $\tau \in X(d-2)$ which is $[k,k+1]$-link-contiguous. Let $K = \tau_{k-1}$ and $L = \tau_{k+2}$. Note that $A_\tau(F,G) = 1$ (up to normalization) if $K \prec F \prec G \prec L$ and $A(F,G) = 0$ otherwise. We now compute the entries of $D_{\phi,\tau}$. First, for any $F \in T_k$, the $(F,F)$ entry is
\[
    \sum_{G: F < G < L} \frac{|L \setminus G|}{|L \setminus F|} = d_\tau(F)-1,
\]
where $d_\tau(F)=|\{G: F< G < L\}|$ is the degree of $F$ in the biparite graph $G_\tau$. The equality follows by the matroid partition property; in particular the sets $G\setminus F$ (over $G$'s which are $F<G<L$) partition elements of $L\setminus F$.
For any $G \in T_{k+1}$, the $(G,G)$ entry is
\[
    \sum_{F:K<F < G} \frac{|F \setminus K|}{|G \setminus K|} = 1,
\]
by the matroid partition property.
The main observation is that 
$$ -D_{\phi,\tau} + A_\tau = J_k  - \sum_{G \in X_\tau(1) \cap T_{k+1}} \left(\bm{e}_G - \sum_{F:K<F<G}\bm{e}_F\right)\left(\bm{e}_G - \sum_{F:K<F<G}\bm{e}_F\right)^\top$$
where $J_k$ is the all-ones matrix restricted to $X_\tau(1) \cap T_k$.  Therefore, it has one positive eigenvalue.

\section{Lower bounds}
\label{sec:lowerbound}

%\THMlowerbound*

\begin{proof}[Proof of \cref{thm:lower-bound}]
We construct a $d$-dimensional distributive lattice such that $\lambda_2(P_\varnothing)\geq 1-\frac{4}{\eps(1+\eps)d}$ for $d$ {\bf even}. %$d_i = 2i$ for $i \geq 1$, set $d = d_i$ and 
Let $P$ be a Poset on $[d+1]$ with relation $2<3<\dots<d$ and $1$ (a free element). Let $\cL$ be the corresponding lattice with flats:
$$ 1\leq i\leq d: \quad  F_i = \{2, 3,\dots,i+1\},\text{ and }   G_i=\{1,2,\dots,i\}$$

%Let $X$ be the distributive lattice as in in \cref{fig:badexternalfielddistrlattice}. We have the collection of flats on the left side $\mathcal{L} = \{L_1, \dots, L_{d-1}\}$ and the collection of flats on the right side $\mathcal{R} = \{R_1, \dots, R_{d}\}$. Note that the indices are also denote their grading.

\begin{figure}[htb]\centering
\begin{tikzpicture}[every node/.style={minimum size=1cm},scale=0.9]

\node[draw, circle] (t) at (-1, 3) {\footnotesize$\hat{1}$};
\node[draw, circle] (l5) at (-2, 2) {{$F_d$}} edge (t);
\node[circle] (l4) at (-1, 1) {}; 
\node[draw, circle] (r5) at (0, 2) {\small $G_d$} edge (l4) edge (t);

\node [circle] at (1,1) (r4) {};
\node[draw, circle] (l3) at (0, 0) {\small$F_{2}$} edge (r4) edge [line width=1.2pt,dotted] (l5);
\node[draw, circle] (r3) at (2, 0) {\small$G_{2}$}  edge [dotted,line width=1.2pt](r5);
\node[draw, circle] (l2) at (1, -1) {\small$F_{1}$} edge (r3) edge (l3);
\node[draw, circle] (r2) at (3, -1) {\small$G_{1}$}  edge (r3);
\node[draw, circle] (b) at (2, -2) {\small$\hat{0}$} edge (l2) edge (r2);
\end{tikzpicture}
%\caption{Not mixing for different external field}
%\label{fig:badexternalfielddistrlattice}
\end{figure}
%No left flats lie above any right flat and any maximal chain always contains $L_1$ and $R_{d}$. 
Observe that this lattice has exactly $d+1$ flag of flats where each one is of the form:
$$ \hat{0} < F_1 < \dots < F_i < G_{i+1} < \dots <G_d < \hat{1}$$
In particular each maximal flag of flat
 can be identified by the rank at which the chain moves from the left to right side of the lattice. So, for $0\leq i\leq d$ let $\tau(i)$ be the flag
which has $F_i$ but not $F_{i+1}$. We think of $\hat{0}$ as $F_0$.

Let
$$ \psi(F_i) = \begin{cases} 1 &\text{if } i\leq d/2\\ 1+\eps&\text{otherwise.}\end{cases} \quad\quad \psi(G_i) = \begin{cases} 1+\eps &\text{if } i\leq d/2\\ 1&\text{otherwise.}\end{cases}.  $$
Then, let $\mu(\tau)\propto \prod_{F\in \tau} \psi(F).
$ %Note that to make this a probability distribution we have to normalize the sum; but for simplicity of notation we work with this un-normalized $\mu$.

It follows that for any $0\leq i\leq d$ we have 
$$ \mu(\tau(i)) \propto (1+\epsilon)^{|d/2-i|}.$$
% \[
% \sigma(j) = \begin{cases}
%     i+1-j & \text{if $j \leq i$}\\
%     j - i & \text{if $j \geq i+1$}.
% \end{cases}
% \]

Let $X$ be the $d$-partite complex corresponding to all flags of flats of $\cL$; then $(X,\mu)$ is a path complex.
First, we claim that $X$ is a $\frac{1+\eps}{2+\eps}$-top link expander. To see that it turns out that the worst link of co-dimension 2 is attained at $\sigma=\{F_1,\dots,F_{d/2-1},G_{d/2+2},\dots,G_d\}.$ Then, $X_\sigma(1)=\{F_{d/2},G_{d/2},F_{d/2+1},G_{d/2+1}\}$ and $P_\sigma=D_\sigma^{-1} A_\sigma$ w.r.t order $(G_{d/2}, F_{d/2}, G_{d/2+1}, F_{d/2+1})$ is
$$ A_\sigma=\begin{bmatrix}
    0 & 0 & 1+\eps & 0 \\ 0& 0 &1 & 1+\eps \\ 1+\eps & 1 & 0 & 0 \\ 0 & 1+\eps & 0 & 0
\end{bmatrix}\quad \Rightarrow\quad  P_\sigma=\begin{bmatrix}0 & 0 & 1 & 0 \\ 0 & 0 & \frac{1}{2+\eps} & \frac{1+\eps}{2+\eps} \\ \frac{1+\eps}{2+\eps} & \frac{1}{2+\eps} & 0 & 0\\ 0 & 1 & 0 & 0 \end{bmatrix}$$

So, $\det(P_\sigma) = \left(\frac{1+\eps}{2+\eps}\right)^2$. Since $G_\sigma$ is a bipartite graph, by \cref{lem:AB} it follows that $\lambda_2(P_\sigma)=\frac{1+\eps}{2+\eps}$.

We show that $\lambda_2(P_{\varnothing}) \geq 1-\frac{4}{\eps(1+\eps) d}$. To do that we use Cheeger's inequality, \cref{thm:Cheeger}.
%we will show there is a bottleneck in the conductance between $\mathcal{L}$ and $\mathcal{R}$ which we denote $\phi(\mathcal{L})$.

Let $S=\{F_1,\dots,F_d\}$ and $\overline{S}=\{G_1,\dots,G_d\}$. By symmetry, we have $\vol(S)=\vol(\overline{S}).$
We show that
$$ \phi(S) \leq \frac{2}{\eps(1+\eps) d}$$
which then completes the proof using \cref{thm:Cheeger}.

The main observation is that the maximal flag $\tau(i)$ has exactly $i\cdot (d-i)$ pairs across the cut, i.e., for all $1\leq j\leq i$ and $i+1\leq k\leq d$ it contributes to the edge  $\{F_j, G_k\}$. So, it contributes $\mu(\tau(i))\cdot i\cdot (d-i)$ to $A_\varnothing(S,\overline{S})$ and $\mu(\tau(i))\cdot 2\binom{d}{2}$ to $\vol(S\cup \overline{S})$. Therefore, we can write
\begin{align*}
    \phi(S) = \frac{A_{\varnothing}(S,\overline{S})}{\vol(S)}& \underset{\vol(S)=\vol(\overline{S})}{=} \frac{\sum_{i=0}^d \mu(\tau(i))\cdot i (d-i)}{\frac12 \vol(S\cup \overline{S})}\\
    &= \frac{\sum_{i=1}^{d-1} (1+\eps)^{|d/2-i|}\cdot i(d-i)}{ \sum_{i=0}^d \mu(\tau(i))\binom{d}{2}}\\
    &\underset{\max\{i,d-i\}\leq d-1}{\leq} 2 \frac{\sum_{i=1}^{d-1} \min\{i,d-i\}(1+\eps)^{|d/2-i|}}{d\sum_{i=0}^d (1+\eps)^{|d/2-i|}}\\
    &\underset{\text{\cref{fact:sum1+epsd-i}}}{\leq} 2\frac{\frac2{\eps}\sum_{i=1}^{d/2}(1+\eps)^{d/2-i}}{2d \sum_{i=0}^{d/2-1} (1+\eps)^{d/2-i}} %\leq 2\max_{0\leq i\leq d/2-1}\frac{\frac{2}{\eps} (1+\eps)^{d/2-(i+1)}}{2d(1+\eps)^{d/2-i}}  
    = \frac{2}{\eps(1+\eps) d}
\end{align*}
as desired.
%In the last inequality we used that for $b_1,\dots,b_n\geq 0$, $\frac{a_1+\dots+a_n}{b_1+\dots+b_n}\leq \max\frac{a_i}{b_i}.$
%Now 
%Computing the conductance is equivalent to computing the probability that starting in $\mathcal{L}$ according to the stationary distribution of $P_{\varnothing}$, we move to $\mathcal{R}$ after one step. If we condition on moving to $\tau_j$ as the intermediate step, then the probability we move to $\mathcal{R}$ is $\frac{d-j}{d-1}$. The probability that we move to $\tau_j$ from $\mathcal{L}$ is
% \[
% \frac{j(1+\epsilon)^{|d/2-j|}}{\sum_{k=1}^{d-1} (d-k)(1+\epsilon)^{|d/2-k|}} = \frac{j(1+\epsilon)^{|d/2-j|}}{d/2\sum_{k=1}^{d-1} (1+\epsilon)^{|d/2-k|}}
% \]
% We first deduce a standard bound by extending to infinite geometric series iteratively
% Now by the law of total probability we have
% \begin{align*}
% \phi(\mathcal{L}) &= \sum_{j=1}^{d-1}\frac{d-j}{d-1} \cdot \frac{j(1+\epsilon)^{|d/2-j|}}{d/2\sum_{k=1}^{d-1} (1+\epsilon)^{|d/2-k|}}  \\&\leq \frac{d-1}{(d-1)d/2\sum_{k=1}^{d-1} (1+\epsilon)^{|d/2-k|}} \left(\sum_{j=1}^{d/2} j(1+\epsilon)^{|d/2-j|} + \sum_{j=d/2}^{d-1} (d-j)(1+\epsilon)^{|d/2-j|}\right)
% \\&\leq\frac{2}{d/2\sum_{k=1}^{d-1} (1+\epsilon)^{|d/2-k|}} \sum_{j=1}^{d/2} j(1+\epsilon)^{d/2 -j}
% \\&\leq \frac{2}{d/2\sum_{k=1}^{d-1} (1+\epsilon)^{|d/2-k|}}\frac{1}{\epsilon} \sum_{j=1}^{d/2} (1+\epsilon)^{d/2 -j}
% \\&\leq \frac{4}{d\epsilon}
% \end{align*}
% By \cref{thm:Cheeger}, $1-\lambda_2(P_{\varnothing}^i) \leq 2 \phi(G) \leq 2\phi(\mathcal{L}) \leq 2\frac{4}{d\epsilon}$ and so we're done.
\end{proof}
\begin{fact}\label{fact:sum1+epsd-i}
    We have the following bounds
    $$ \sum_{i=0}^{d} \min\{i,d-i\}(1+\epsilon)^{|d/2 - i|}\leq \frac{2}{\epsilon}\sum_{i=1}^{d/2} (1+\epsilon)^{d/2 - i}$$
\end{fact}
\begin{proof}
% First, observe
% $$ \sum_{i=0}^d (1+\eps)^{|d/2-i|} \leq 2\sum_{i=0}^{d/2} (1+\eps)^{d/2-i}\leq  2(1+\eps)^{d/2}\sum_{i=0}^{\infty} (1+\eps)^{-i} = \frac{2(1+\eps)^{d/2}}{\eps} $$
We write,
    \begin{align*}
\sum_{i=0}^{d} \min\{i,d-i\}(1+\epsilon)^{|d/2 - i|}&\leq 2\sum_{i=0}^{d/2} i(1+\eps)^{d/2-i}\\
&= 2(1+\epsilon)^{d/2} (1\cdot(1+\epsilon)^{-1} + 2\cdot(1+\epsilon)^{-2} + \dots + \frac{d}{2}\cdot (1+\eps)^{-d/2}) \\
&\leq2 (1+\epsilon)^{d/2} \left(\frac{(1+\eps)^{-1}}{\epsilon} + \frac{(1+\eps)^{-2}}{\epsilon} + \dots\right) = \frac{2}{\epsilon}\sum_{i=1}^{d/2} (1+\epsilon)^{d/2 - i}.
\end{align*}
In the last inequality we used $\sum_{i=0}^\infty (1+\eps)^{-i}=\frac{1}{\eps}.$
\end{proof}
%just for reference from NBC paper
%\textcolor{red}{Sh: Kasper, i added a background info abou RW and Cheeger inequaltiy and conductance in section 2. You don't need to write those here.}
% \begin{definition}[Conductance]
%     Given a weighted $d$-regular graph $G=(V,E,w)$, with weights $w:E\to\R_{\geq 0}$, for $S \subseteq V$, the conductance of $S$ is defined as
%     \[
%     \phi(S) = \frac{w(S,\overline{S})}{d|S|},
%     \]
%     where  $w(S,\overline{S})$ is the sum of the weights of edges in the cut $(S,\overline{S})$. Note that since $G$ is regular, the weighted degree of every vertex is $d$.
%     The conductance of  $G$ is defined as
%     \[
%     \phi(G) = \min_{S:|S|\leq |V|/2} \phi(S).
%     \]
% \end{definition}
% %\begin{definition}[Walk operator]
% Given a weighted graph $G=(V,E,w)$, the simple random walk is the following stochastic process: Given $X_0=v\in V$, for every $u\sim v$, we have $X_1=u$ with probability $\frac{w_{\{u,v\}}}{d_w(v)}$  and we let $P$ be the transition probability matrix of the walk.

% The following theorem is well-known and follows from the easy side of the Cheeger's inequality. 
% \begin{theorem}\label{thm:bottlkeneck}
% For any regular graph $G=(V,E)$ and any set $S\subseteq V$ and $|S|\leq |V|/2$
% $$ \frac{1-\lambda_2(P)}{2}\leq \phi(G) \leq \phi(S) \leq \frac{|N(S)|}{|S|}$$
% where $1-\lambda_2(P)$ is the spectral gap of the simple random walk on $G$.
% \end{theorem}
\printbibliography
\end{document}